\documentclass[12pt]{article}
\usepackage{mathrsfs,amsthm,color,verbatim,bbm,amsmath,amsfonts,amssymb,nicefrac,geometry,color,enumitem,hyperref}

\usepackage[nosort,capitalise]{cleveref} 

\crefname{enumi}{item}{items}

\usepackage{xparse}

\ExplSyntaxOn

\NewDocumentCommand{\enum}{ O{;} m o }
 {
  \my_enum:nnn { #1 } { #2 } { #3 }
 }

\seq_new:N \l__my_enum_seq
\tl_new:N \l__my_enum_item_tl
\prop_new:N \l__verbs
\prop_put:Nnn \l__verbs {show}{shows}
\prop_put:Nnn \l__verbs {imply}{implies}
\prop_put:Nnn \l__verbs {demonstrate}{demonstrates}
\prop_put:Nnn \l__verbs {prove}{proves}
\prop_put:Nnn \l__verbs {establish}{establishes}
\prop_put:Nnn \l__verbs {ensure}{ensures}
\prop_put:Nnn \l__verbs {assure}{assures}
\int_new:N \l__number_of_args

\cs_new_protected:Nn \my_enum:nnn
 {
  \seq_set_split:Nnn \l__my_enum_seq { #1 } { #2 }
  \seq_remove_all:Nn \l__my_enum_seq {}
  \int_set_eq:NN \l__number_of_args { \seq_count:N \l__my_enum_seq }
  \seq_use:Nnnn \l__my_enum_seq { ~and~ } { ,~ } { ,~and~ }
  \IfNoValueTF{#3}{}{
    \space
    \int_compare:nNnTF{ \l__number_of_args } < {2}{ \prop_item:Nn \l__verbs {#3} }{ #3 }
  }
 }

\ExplSyntaxOff

\theoremstyle{plain}
\newtheorem{theorem}{Theorem}[section]
\newtheorem{lemma}[theorem]{Lemma}
\newtheorem{prop}[theorem]{Proposition}
\newtheorem{cor}[theorem]{Corollary}

\theoremstyle{definition}
\newtheorem{definition} [theorem]{Definition}

\newcommand{\E}{\mathbb{E}}
\renewcommand{\P}{\mathbb{P}}

\newcommand{\R}{\mathbb{R}}
\newcommand{\N}{\mathbb{N}}
\newcommand{\smallsum}{\textstyle\sum}
\newcommand{\pmat}[1]{\begin{pmatrix}#1\end{pmatrix}}

\newcommand{\is}{\leftarrow}

\newcommand{\eps}{\varepsilon}

\newcommand{\Borel}{\mathcal{B}}
\newcommand{\cB}{\mathcal{B}}
\newcommand{\cE}{\mathcal{E}}
\newcommand{\cM}{\mathcal{M}}
\newcommand{\bN}{\mathbf{N}}
\newcommand{\bP}{\mathbf{P}}
\newcommand{\cD}{\mathcal{D}}
\newcommand{\cL}{\mathcal{L}}
\newcommand{\cR}{\mathcal{R}}

\newcommand{\cW}{\mathcal{W}}

\newcommand{\idRelu}{\mathfrak I}

\newcommand{\MappingStructuralToVectorized}{\mathcal{T}}

\newcommand{\Exists}{\exists\,}
\newcommand{\Forall}{\forall\,}

\newcommand{\mc}[1]{\mathcal{#1}}
\newcommand{\mf}[1]{\mathfrak{#1}}

\newcommand{\abs}[1]{\lvert #1\rvert}
\newcommand{\babs}[1]{\bigl\lvert #1\bigr\rvert}

\newcommand{\norm}[1]{\lVert #1 \rVert}

\makeatletter
\newcommand{\opnorm}{\@ifstar\@opnorms\@opnorm}
\newcommand{\@opnorms}[1]{%
  \left|\mkern-1.5mu\left|\mkern-1.5mu\left|
   #1
  \right|\mkern-1.5mu\right|\mkern-1.5mu\right|
}
\newcommand{\@opnorm}[2][]{%
  \mathopen{#1|\mkern-1.5mu#1|\mkern-1.5mu#1|}
  #2
  \mathclose{#1|\mkern-1.5mu#1|\mkern-1.5mu#1|}
}
\makeatother

\newcommand{\infnorm}[1]{\opnorm{#1}}
\newcommand{\asinfnorm}[1]{\opnorm*{#1}}

\newcommand{\Aff}{\mathcal{A}}
\newcommand{\RealV}{\mathcal{N}}

\newcommand{\ClippedRealV}[4]{\mathscr{N}^{#1,#2}_{#3,#4}}
\newcommand{\UnclippedRealV}[2]{\mathscr{N}^{#1,#2}_{-\infty,\infty}}

\newcommand{\CovNum}[1]{\mathcal C_{#1}}

\newcommand{\idMatrix}{\operatorname{I}}

\newcommand{\ANNs}{\mathbf{N}}
\newcommand{\MatrixANN}[1]{\mf N_{#1}}
\newcommand{\activation}{a}
\newcommand{\activationDim}[1]{\mathfrak{M}_{\activation,#1}}
\newcommand{\functionANN}[1]{\mathcal{R}_{#1}}
\newcommand{\paramANN}{\mathcal{P}}

\newcommand{\lengthANN}{\mathcal{L}}
\newcommand{\inDimANN}{\mathcal{I}}
\newcommand{\compANN}[2]{{#1 \bullet #2}}

\newcommand{\paraANN}[1]{\mathbf{P}_{#1}}

\newcommand{\outDimANN}{\mathcal{O}}
\newcommand{\longerANN}[1]{\mathcal{E}_{#1}}

\newcommand{\dims}{\mathcal{D}}
\newcommand{\hiddenLength}{\mathcal{H}}

\newcommand{\parallelizationSpecial}{\mathbf{P}}

\newcommand{\qandq}{\qquad\text{and}\qquad}

\newcommand{\andShort}{\text{ and }}

\newcommand{\pa}[1]{\left({#1}\right)}

\newcommand{\rect}{\mathfrak r}
\newcommand{\Rect}{\mathfrak R}
\newcommand{\multdim}{\mathfrak M}

\newcommand{\clip}[2]{\mf c_{#1,#2}}
\newcommand{\Clip}[3]{\mf C_{#1,#2,#3}}

\DeclareMathOperator{\id}{id}

\newcommand{\card}[1]{\abs{#1}}
\newcommand{\stdStartDim}{\delta}

\DeclareSymbolFont{largesymbolsA}{U}{jkpexa}{m}{n}
\SetSymbolFont{largesymbolsA}{bold}{U}{jkpexa}{bx}{n}
\DeclareMathSymbol{\bigtimes}{\mathop}{largesymbolsA}{16}

\begin{document}

\title{Full error analysis for the \\ training of deep  neural networks}

\author{Christian Beck$^1$, Arnulf Jentzen$^{2,3}$, and Benno Kuckuck$^{4,5}$\bigskip\\
\small{$^1$ Department of Mathematics, ETH Zurich, Z\"urich,}\\
\small{Switzerland, e-mail: christian.beck@math.ethz.ch}\\
\small{$^2$ Department of Mathematics, ETH Zurich, Z\"urich,}\\
\small{Switzerland, e-mail: arnulf.jentzen@sam.math.ethz.ch}\\
\small{$^3$ Faculty of Mathematics and Computer Science, University of M\"unster, }\\
\small{M\"unster, Germany, e-mail: ajentzen@uni-muenster.de}\\
\small{$^4$ Institute of Mathematics, University of D\"usseldorf, D\"usseldorf,}\\
\small{Germany, e-mail: kuckuck@math.uni-duesseldorf.de}\\
\small{$^5$ Faculty of Mathematics and Computer Science, University of M\"unster, }\\
\small{M\"unster, Germany, e-mail: bkuckuck@uni-muenster.de}
}

\maketitle

\begin{abstract}
Deep learning algorithms have been applied very successfully in recent years to a range of problems out of reach for classical solution paradigms. Nevertheless, there is no completely rigorous mathematical error and convergence analysis which explains the success of deep learning algorithms. The error of a deep learning algorithm can in many situations be decomposed into three parts, the approximation error, the generalization error, and the optimization error. In this work we estimate for a certain deep learning algorithm each of these three errors and combine these three error estimates to obtain an overall error analysis for the deep learning algorithm under consideration. In particular, we thereby establish convergence with a suitable convergence speed for the overall error of the deep learning algorithm under consideration. Our convergence speed analysis is far from optimal and the convergence speed that we establish is rather slow, increases exponentially in the dimensions, and, in particular, suffers from the curse of dimensionality. The main contribution of this work is, instead, to provide a full error analysis (i) which covers each of the three different sources of
errors usually emerging in deep learning algorithms and (ii) which merges these three sources of errors into one overall error estimate for the considered deep learning algorithm.
\end{abstract}

\tableofcontents

\section{Introduction}
\label{sec:intro}

In problems like image recognition, text analysis, speech recognition, 
or playing various games, to name a few, it is very hard and seems at 
the moment entirely impossible to provide a function or to hard-code a 
computer program which attaches to the input -- be it a picture, a piece 
of text, an audio recording, or a certain game situation -- a meaning or a recommended action. 
Nevertheless deep learning has been applied very successfully in recent years to such and related problems. The success of deep learning in applications is even more surprising as, to this day, the reasons for its performance are not entirely rigorously understood. In particular, there is no rigorous mathematical error and convergence analysis which explains the success of deep learning algorithms. 

In contrast to traditional approaches, machine learning methods in 
general and deep learning methods in particular attempt to infer the unknown target function or at least a good enough approximation thereof from examples encountered during the training. 
Often a deep learning algorithm has three ingredients: (i) the \emph{hypothesis class}, a parametrizable class of functions in which we try to find a reasonable approximation of the unknown target function, (ii) a \emph{numerical approximation of the expected loss function} based on the training examples, and (iii) an \emph{optimization algorithm} which tries to approximately calculate an element of the hypothesis class which minimizes the numerical approximation of the expected loss function from (ii) given the training examples. Common approaches are to choose a set of suitable fully connected deep neural networks (DNNs) as hypothesis class in (i), empirical risks as approximations of the expected loss function in (ii), and stochastic gradient descent-type algorithms with random initializations as optimization algorithms in (iii). Each of these three ingredients contributes to the overall error of the considered approximation algorithm. The choice of the hypothesis class results in the so-called \emph{approximation error} (cf., e.g., \cite{barron1993universal,
barron1994approximation,
Cybenko1989,
hornik1991approximation,
hornik1989multilayer,
hornik1990universal} and the references mentioned at the beginning of \cref{sec:error_sources}), 
replacing the exact expected loss function by a numerical approximation leads 
to the so-called \emph{generalization error} 
(cf., e.g., \cite{Bartlett2005,
	BernerGrohsJentzen2018,CuckerSmale2002FoundationsLearning,
	GyoerfiKohlerKrzyzakWalk2002,
	Massart2007,
	shalev2014understanding,
	VanDeGeer2000} and the references mentioned therein), 
and the employed optimization algorithm introduces 
the \emph{optimization error} (cf., e.g.,~\cite{bach2013non,kolmogorov2018solving,bercu2011generic,chau2019stochastic,DereichMuellerGronbach2019,FehrmanGessJentzen2019convergence,JentzenKuckuckNeufeldWurstemberger2018,JentzenWurstemberger2018LowerBound} and the references mentioned therein). 

In this work we estimate the approximation error, the generalization 
error, as well as the optimization error and we also combine these three 
errors to establish convergence with a suitable convergence speed for 
the overall error of the deep learning algorithm under consideration. 
Our convergence speed analysis is far from optimal and the convergence 
speed that we establish is rather slow, increases exponentially in the dimensions, and, in particular, suffers from the curse of 
dimensionality (cf., e.g., 
Bellman~\cite{Bellman1957}, Novak \& Wo\'zniakowski~\cite[Chapter 1]{NovakWozniakowski2008}, and Novak \& Wo\'zniakowski~\cite[Chapter 9]{NovakWozniakowski2010}). The main contribution of this work is, instead, to 
provide a full error analysis (i) which covers each of the three different 
sources of errors usually emerging in deep learning algorithms and 
(ii) which merges these three sources of errors into one overall error 
estimate for the considered deep learning algorithm. In the next result, 
\cref{thm:intro}, we briefly illustrate the findings of this article in 
a special case and we refer to \cref{subsec:overall_analysis} below for 
the more general convergence results which we develop in this article.

\begin{theorem}\label{thm:intro}
	Let 
	$ d \in \N $, 
  $ L,a
  \in \R $, 
	$ b \in (a,\infty) $,
	$ R \in [\max\{2, L, |a|, |b|\}, \infty)  $,  
	let 
	$ ( \Omega, \mathcal{F}, \P ) $ 
	be a probability space, 
	let 
	$ X_m \colon \Omega \to [a,b]^d $, 
	$ m \in \N $,
	be i.i.d.\ random variables, 
	let 
	$ \norm{\cdot}\colon \R^d \to [0,\infty) $ 
	be the standard norm on $ \R^d $, 
	let 
	$ \varphi \colon [a,b]^d \to [0,1] $ 
	satisfy for all 
	$ x, y \in [a,b]^d $ 
	that 
	$| \varphi(x) - \varphi(y) | \leq L\norm{x-y}$,  
	for every $\mf d,r,s \in \N$, $\stdStartDim \in \N_0$, 
	$ \theta = ( \theta_1,\theta_2, \dots, \theta_{\mf d} ) \in \R^{\mf d} $ with $\mf d \geq \stdStartDim + r s + r$ let
	$\Aff_{r,s}^{\theta, \stdStartDim}\colon \R^{s} \to \R^{r}$ satisfy
	for all $x = (x_1,x_2,\ldots, x_{s}) \in \R^{s}$ that
	\begin{equation} \label{intro:affine_transformations}
	\Aff_{r,s}^{\theta,\stdStartDim}( x ) 
	= 
	\left( \textstyle
	\left[ \sum\limits_{i = 1}^s x_i \theta_{\stdStartDim+i} \right] + \theta_{\stdStartDim + rs + 1} ,
	\left[ \sum\limits_{i = 1}^s  x_i\theta_{\stdStartDim+s+i} \right] + \theta_{\stdStartDim + rs + 2} , 
	\ldots, \textstyle
	\left[ \sum\limits_{i = 1}^s  x_i \theta_{\stdStartDim+(r-1)s+i} \right] + \theta_{\stdStartDim + rs + r}
	\right)\!
	, 
	\end{equation}
let $\mf c\colon \R \to [0,1]$ and $\mf R_{\tau} \colon \R^{\tau} \to \R^{\tau}$, $\tau\in\N$, satisfy for all 
	$ \tau \in \N $, 
	$ x=(x_1,x_2,\ldots,x_{\tau})\in\R^{\tau} $,
	$ y \in \R $ 
	that 
	$\mf c(y) = \min\{1,\max\{0,y\}\}$ 
	and 
	$\mf R_{\tau}(x) = 
	(\max\{x_1,0\},\max\{x_2,0\}, \ldots, \max\{x_{\tau},0\})$, 
for every $\mf d, \tau\in\{3,4,\ldots\}$, $\theta\in\R^{\mf d}$ 
with $\mf d \geq \tau(d+1) + (\tau-3)\tau(\tau+1)+\tau+1$ let $\mf N^{\theta,\tau}\colon \R^d \to \R$ satisfy for all 
	$x\in\R^d$
that 
\begin{equation} \label{intro:neural_net_dfn}
    \bigl( {\mf N}^{\theta, \tau} \bigr)(x)
	=  
	\bigl(
	\mf c 
	\circ \Aff_{1,\tau}^{\theta, \tau(d+1) + (\tau-3)\tau(\tau+1)}
	\circ \mf R_{\tau}
	\circ \Aff_{\tau,\tau}^{\theta, \tau(d+1) + (\tau-4)\tau(\tau + 1)}
	\circ  \mf R_{\tau}
	\circ 
	\ldots
	\circ  \Aff_{\tau,\tau}^{\theta, \tau(d + 1)}
	\circ \mf R_{\tau}  
	\circ \Aff_{\tau,d}^{\theta,0}
	\bigr)(x), 
	\end{equation}
	let	$ \mathfrak{E}_{\mf d,M,\tau} \colon [-R,R]^{\mf d} \times \Omega \to [0,\infty) $, $ \mf d, M, \tau \in \N $, 
	satisfy for all 
	$ \mf d, M \in \N $, 
	$ \tau \in \{3,4,\ldots\} $, 
	$ \theta \in [-R,R]^{\mf d} $,  
	$ \omega \in \Omega $ 
with $\mf d \geq \tau(d+1) + (\tau-3)\tau(\tau+1)+\tau+1$
	that
	\begin{equation} \label{intro:risk}
	\mathfrak{E}_{ \mf d, M, \tau }( \theta, \omega )
	=
	\frac{ 1 }{ M }\!
	\left[
	\smallsum\limits_{ m = 1 }^M
	| \mf N^{\theta,\tau}( X_m( \omega ) ) - \varphi(X_m( \omega )) |^2
	\right]\!
	,
	\end{equation}
	for every $\mf d \in \N$ let 
	$\Theta_{\mf d,k}\colon\Omega\to [-R,R]^{\mf d}$, 
	$k\in\N$, 
	be i.i.d.~random variables, 
	assume for all $\mf d\in\N$ that 
	$\Theta_{\mf d,1}$ 
	is continuous uniformly distributed on 
	$[-R,R]^{\mf d}$, 
	and let 
	$\Xi_{\mf d,K,M,\tau}\colon\Omega\to [-R,R]^{\mf d}$, $\mf d,K,M,\tau\in\N$, 
	satisfy for all $ \mf d,K,M,\tau \in \N $ that
	$\Xi_{\mf d,K,M,\tau} = \Theta_{\mf d, \min\{k\in\{1,2,\ldots,K\}\colon\mf{E}_{\mf d,M,\tau}(\Theta_{\mf d,k}) = \min_{l\in\{1,2,\ldots,K\}} \mf{E}_{\mf d,M,\tau}(\Theta_{\mf d,l}) \}} 
	$. 
	Then there exists $c\in (0,\infty)$ such that for all 
	$\mathfrak{d}, K, M, \tau \in \N$, 
  $ \varepsilon \in (0,1] $  
  with \enum{
    $\tau\geq  2d(2dL\varepsilon^{-1}+2)^d $ ;
    $\mf d \geq \tau(d+1) + (\tau-3)\tau(\tau+1)+\tau+1$
  }
	it holds that
	\begin{equation} \label{intro:claim}
	\begin{split}
	&
	\P\!\left(
	\int_{[a,b]^d}
	| 
	\mf N^{ \Xi_{\mf d,K,M,\tau} , \tau }( x )
	- 
	\varphi( x ) 
	|
	\,  
	\P_{ X_1 }( dx )
	> 
	\eps 
	\right)
	\leq 
	\exp\bigl(-K (c\tau)^{-\tau\mf d}\varepsilon^{2\mf{d}} \bigr)
	+
	2\exp\bigl(\mf d\ln\!\big( (c\tau)^{\tau}\eps^{-2} \bigr) - 
\tfrac{\varepsilon^4 M}{c} \big). 
	\end{split}
	\end{equation}
\end{theorem}

\cref{thm:intro} is an immediate consequence of \cref{cor:l1_norm} in Section \ref{subsec:overall_analysis} below. \cref{cor:l1_norm} follows from  \cref{cor:generic_constants} which, in turn, is implied by \cref{prop:error_decomposition}, the main result of this article. In the following we add some comments and explanations regarding the mathematical objects which appear in \cref{thm:intro} above. For every $\mf{d}, \tau \in \{ 3, 4, ... \}, \theta \in \R^{\mf d}$ with $\mf{d} \geq \tau(d+1) + (\tau-3)\tau(\tau+1)+\tau+1$ the function $\mf{N}^{ \theta, \tau } \colon \R^{\mf d} \to \R$ in \eqref{intro:neural_net_dfn} above describes the realization of a fully connected deep neural network with $\tau$ layers (1 input layer with $d$ neurons [$d$ dimensions], $1$ output layer with $1$ neuron [$1$ dimension], as well as $\tau - 2$ hidden layers with $\tau$ neurons on each hidden layer [$\tau$ dimensions in each hidden layer]). 
The vector $\theta \in \R^{ \mf d }$ in \eqref{intro:neural_net_dfn} in \cref{thm:intro} above stores the real parameters (the weights and the biases) for the concrete considered neural network. In particular, the architecture of the deep neural network in \eqref{intro:neural_net_dfn} is chosen so that we have $\tau d + (\tau - 3)\tau^2 + \tau$ real parameters in the weight matrices and $ (\tau-2)\tau + 1$ real parameters in the bias vectors resulting in $ [\tau d + (\tau - 3)\tau^2 + \tau] + [(\tau-2)\tau + 1] =  \tau(d+1) + (\tau-3)\tau(\tau+1) + \tau + 1$ real parameters for the deep neural network overall. 
This explains why the dimension $\mathfrak{d}$ of the parameter vector $\theta \in \R^{ \mathfrak{d} }$ must be larger or equal than the number of real parameters used to describe the deep neural network in \eqref{intro:neural_net_dfn} in the sense that $\mathfrak{d} \geq  \tau(d+1) + (\tau-3)\tau(\tau+1) + \tau + 1$ (see above \eqref{intro:neural_net_dfn}). The affine linear transformations for the deep neural network, which appear just after the input layer and just after each hidden layer in \eqref{intro:neural_net_dfn}, are specified in \eqref{intro:affine_transformations} above. 
The functions ${\mf R}_{ \tau }\colon \R^{\tau} \to \R$, $\tau \in \N$, describe the multi-dimensional rectifier functions which are employed as activation functions in \eqref{intro:neural_net_dfn}. 
Realizations of the random variables $(X_m,Y_m):=(X_m,\varphi(X_m))$, $m\in\{1,2,\ldots,M\}$, act as training data and the neural network parameter vector $\theta\in\R^{\mf d}$ should be chosen so that the empirical risk in \eqref{intro:risk} gets minimized. In \cref{thm:intro} above, we use as an optimization algorithm just random initializations and perform no gradient descent steps. The inequality in \eqref{intro:claim} in \cref{thm:intro} above provides a quantitative error estimate for the probability that the $L^1$-distance between the trained deep neural network approximation $\mf N^{ \Xi_{\mf d,K,M,\tau} , \tau }( x )$, $x \in [a,b]^d$, and the function $\varphi(x)$, $x \in [a,b]^d$, which we actually want to learn, is larger than a possibly arbitrarily small real number $ \eps \in (0,1] $. 
In \eqref{intro:claim} in Theorem 1.1 above we measure the error between 
the deep neural 
network and the function $\varphi \colon [a,b]^d \to [0,1]$, which we intend to learn, in the $L^1$-distance instead of in the $L^2$-distance. However, in the more general results in Section~\ref{subsec:overall_analysis} below we measure the error in the 
$L^2$-distance and, just to keep the statement in \cref{thm:intro} as easily accessible as possible, we restrict ourselves in \cref{thm:intro} above to the $L^1$-distance. Observe that for every $\eps \in (0,1]$ and every $\mathfrak{d}, \tau \in \{3,4,\ldots\}$ with $ \mf d \geq \tau (d+1) + (\tau - 3)\tau (\tau+1) + \tau + 1 $ we have that the right hand side of \eqref{intro:claim} converges to zero as $K$ and $M$ tend to infinity. 
The right hand side of \eqref{intro:claim} also specifies a concrete speed of convergence and in this sense \cref{thm:intro} provides a full error analysis for the deep learning algorithm under consideration. Our analysis is in parts inspired by Maggi~\cite{Maggi2012GMT}, Berner et al.~\cite{BernerGrohsJentzen2018}, Cucker \& Smale~\cite{CuckerSmale2002FoundationsLearning}, Beck et al.~\cite{kolmogorov2018solving}, and Fehrman et al.~\cite{FehrmanGessJentzen2019convergence}. 

The remainder of this article is organized as follows. 
In \cref{sec:dnns} we present two elementary approaches how DNNs can be described 
in a mathematical fashion. 
Both approaches will be used in our error analyses in the later parts of this article. 
In \cref{sec:error_sources} we separately analyze the approximation error, the generalization error, and the optimization error of the considered algorithm. 
In \cref{sec:overall_error_analysis} we combine the separate error analyses in \cref{sec:error_sources} to obtain an overall error analysis of the considered algorithm.  

\section{Deep neural networks (DNNs)}
\label{sec:dnns}

In this section we present two elementary approaches on how DNNs can be described in a mathematical fashion. 
More specifically, we present in \cref{subsec:vectorized_description} a vectorized description for DNNs and we 
present in \cref{subsec:structured_description} a structured description for DNNs. Both approaches will be used 
in our error analyses in the later parts of this article. 
\cref{subsec:vectorized_description,subsec:structured_description,subsec:lipschitz} are partially based on material 
in publications from the scientific literature such as 
Beck et al.~\cite{kolmogorov2018solving,beck2019machine}, 
Berner et al.~\cite{BernerGrohsJentzen2018},
Goodfellow et al.~\cite{Goodfellow2016DeepLearning}, 
and 
Grohs et al.~\cite{Grohs2019ANNCalculus,GrohsJentzenSalimova2019}. 
In particular, 
	\cref{def:affine} is inspired by, e.g., (25) in \cite{beck2019machine}, 
	\cref{def:FFNN} is inspired by, e.g., (26) in \cite{beck2019machine},
	\cref{def:multidim_version} is, e.g., \cite[Definition~2.2]{Grohs2019ANNCalculus}, 
	\cref{def:relu1,def:relu,def:clip1,def:clip,def:rectclippedFFANN} are inspired by, e.g., \cite[Setting~2.3]{BernerGrohsJentzen2018}, 
	\cref{def:ANN} is, e.g., \cite[Definition~2.1]{Grohs2019ANNCalculus}, 
	\cref{def:ANNrealization} is, e.g., \cite[Definition~2.3]{Grohs2019ANNCalculus}, 
  \cref{def:simpleParallelization} is, e.g., \cite[Definition~2.17]{Grohs2019ANNCalculus},
  \cref{def:matrixInputDNN} is, e.g., \cite[Definition~3.10]{GrohsJentzenSalimova2019}, 
  \cref{def:ReLu:identity} is, e.g., \cite[Definition~3.15]{GrohsJentzenSalimova2019},
  \cref{def:ANNcomposition} is, e.g., \cite[Definition~2.5]{Grohs2019ANNCalculus},
  \cref{def:iteratedANNcomposition} is, e.g., \cite[Definition~2.11]{Grohs2019ANNCalculus},
  \cref{def:ANNenlargement} is, e.g., \cite[Definition~2.12]{Grohs2019ANNCalculus},
and 
  \cref{thm:RealNNLipsch} is a strengthened version of \cite[Theorem 4.2]{BernerGrohsJentzen2018}.

\subsection{Vectorized description of DNNs}
\label{subsec:vectorized_description} 

\subsubsection{Affine functions}

\begingroup
\newcommand{\s}{\stdStartDim}
\begin{definition}[Affine function]
  \label{def:affine}
  Let $d,r,s \in \N$, $\s \in \N_0$, 
  $ \theta = ( \theta_1, \theta_2, \dots, \theta_d ) \in \R^d $ 
  satisfy $d \geq \s + r s + r$.
  Then we denote by $\Aff_{r,s}^{\theta, \s}\colon \R^{s} \to \R^{r}$ the function which satisfies 
  for all $x = (x_1,x_2,\ldots, x_{s}) \in \R^{s}$ that
  \begin{equation}
  \begin{split}
     &\Aff_{r,s}^{\theta,\s}( x ) 
  = 
    \left(
      \begin{array}{cccc}
        \theta_{ \s + 1 }
      &
        \theta_{ \s + 2 }
      &
        \cdots
      &
        \theta_{ \s + s }
      \\
        \theta_{ \s + s + 1 }
      &
        \theta_{ \s + s + 2 }
      &
        \cdots
      &
        \theta_{ \s + 2 s }
      \\
        \theta_{ \s + 2 s + 1 }
      &
        \theta_{ \s + 2 s + 2 }
      &
        \cdots
      &
        \theta_{ \s + 3 s }
      \\
        \vdots
      &
        \vdots
      &
        \ddots
      &
        \vdots
      \\
        \theta_{ \s + ( r - 1 ) s + 1 }
      &
        \theta_{ \s + ( r - 1 ) s + 2 }
      &
        \cdots
      &
        \theta_{ \s + r s }
      \end{array}
    \right)
    \left(
      \begin{array}{c}
        x_1
      \\
        x_2
      \\
        x_3
      \\
        \vdots 
      \\
        x_{s}
      \end{array}
    \right)
    +
    \left(
      \begin{array}{c}
        \theta_{ \s + r s + 1 }
      \\
        \theta_{ \s + r s + 2 }
      \\
        \theta_{ \s + r s + 3 }
      \\
        \vdots 
      \\
        \theta_{ \s + r s + r }
      \end{array}
    \right) \\
  &=
    \Big( \textstyle
      \Big[ \sum_{k = 1}^s x_k \theta_{\s+k} \Big] + \theta_{\s + rs + 1} ,
      \Big[ \sum_{k = 1}^s  x_k\theta_{\s+s+k} \Big] + \theta_{\s + rs + 2} , 
      \ldots,
      \Big[ \sum_{k = 1}^s  x_k \theta_{\s+(r-1)s+k} \Big] + \theta_{\s + rs + r}
	\Big)
	.
  \end{split}
  \end{equation}
\end{definition}
\endgroup

\subsubsection{Vectorized description of DNNs}

\begingroup
\newcommand{\s}{\stdStartDim}
\begin{definition}
  \label{def:FFNN}
  Let $d,L \in \N$, $l_0,l_1,\ldots, l_L \in \N$, $\s \in \N_0$, $\theta \in \R^d$ satisfy
  \begin{equation}
    d \geq \s + \sum_{k=1}^{L} l_k(l_{k-1} + 1)
  \end{equation}
  and let $\Psi_k \colon \R^{l_k} \to \R^{l_k}$, $k \in \{1,2,\ldots,L\}$, be functions.
  Then we denote by 
    $\RealV^{\theta,\delta,l_0}_{\Psi_1,\Psi_2,\ldots,\Psi_L} \colon\R^{l_0} \to \R^{l_L}$ 
  the function which satisfies for all $x \in \R^{l_0}$ that
  \begin{multline}
  \label{eq:FFNN}
    \bigl( \RealV^{\theta,\delta,l_0}_{ \Psi_1, \Psi_2, \ldots, \Psi_L } \bigr)(x)
  =  
    \bigl(
      \Psi_L
      \circ \Aff_{l_L,l_{L-1}}^{\theta,\s + \sum_{k = 1}^{L-1} l_k (l_{k-1} + 1)}
      \circ \Psi_{L-1} 
      \circ \Aff_{l_{L-1},l_{L-2}}^{\theta,\s + \sum_{k = 1}^{L-2} l_k (l_{k-1} + 1)}
      \circ 
      \ldots  \\
      \ldots
      \circ \Psi_{2}  
      \circ  \Aff_{l_2,l_1}^{\theta,\s + l_1(l_0 + 1)}
      \circ \Psi_{1}  
      \circ \Aff_{l_1,l_0}^{\theta,\s}
    \bigr)(x)
  \end{multline}
  (cf.\ Definition~\ref{def:affine}).
\end{definition}
\endgroup

\subsubsection{Activation functions}

\begin{definition}[Multidimensional version]
  \label{def:multidim_version}
  Let $d \in \N$ and let $\psi \colon \R \to \R$ be a function.
  Then we denote by $\multdim_{\psi, d} \colon \R^d \to \R^d$ the function which satisfies for all 
  $ x = ( x_1, x_2, \dots, x_{d} ) \in \R^{d} $ that
  \begin{equation}
  \label{multidim_version:Equation}
    \multdim_{\psi, d}( x ) 
  =
    \left(
      \psi(x_1)
      ,
      \psi(x_2)
      ,
      \ldots
      ,
      \psi(x_d)
    \right).
  \end{equation}
  \end{definition}

\begin{definition}[Rectifier function]
  \label{def:relu1}
  We denote by $ \rect \colon \R \to \R $ the function which satisfies 
  for all $ x \in \R $ that 
  \begin{equation}
    \rect (x) = \max\{ x, 0 \}.
  \end{equation}
\end{definition}

\begin{definition}[Multidimensional rectifier function]
  \label{def:relu}
  Let $d \in \N$. Then we denote by $ \Rect_{d} \colon \R^{d} \to \R^{d} $ the function given by
  \begin{equation}
  \label{eq:relu}
    \Rect_d
  =
    \multdim_{\rect, d}
  \end{equation}
  (cf.\ \cref{def:multidim_version,def:relu1}).
\end{definition}

\begin{definition}[Clipping function]
  \label{def:clip1}
  Let $u\in[-\infty,\infty)$, $v\in(u,\infty]$.
  Then we denote by $ \clip uv \colon \R \to \R $ the function which satisfies 
  for all $ x \in \R $ that 
  \begin{equation}
    \clip uv (x) = \max\{u,\min\{x,v\}\}.
  \end{equation}
\end{definition}

\begin{definition}[Multidimensional clipping function]
  \label{def:clip}
  Let 
    $d \in \N$,
	$u\in [-\infty,\infty)$,
	$v\in (u,\infty]$.
  Then we denote by $ \Clip uvd \colon \R^{d} \to \R^{d} $ the function given by
  \begin{equation}
    \Clip uvd
  =
    \multdim_{\clip uv, d}
  \end{equation}
  (cf.\ \cref{def:multidim_version,def:clip1}).
\end{definition}

\subsubsection{Rectified DNNs}


\begin{definition}[Rectified clipped DNN]
  \label{def:rectclippedFFANN}
  Let 
    $L,d\in\N$,
	$u\in[-\infty,\infty)$,
	$v\in(u,\infty]$,
    $\mathbf l=(l_0,l_1,\dots,l_L)\in\N^{L+1}$,
    $\theta\in\R^d$
  satisfy
  \begin{equation}
    d\geq \sum_{k=1}^Ll_k(l_{k-1}+1).
  \end{equation}
  Then we denote by 
    $\ClippedRealV{\theta}{\mathbf l} uv\colon \R^{l_0}\to\R^{l_L}$
  the function which satisfies for all 
    $x\in\R^{l_0}$
  that
  \begin{equation}
    \label{eq:rectclippedFFANN}
    \ClippedRealV{\theta}{\mathbf l}uv(x)
    =
    \begin{cases}
      \bigl( 
        \RealV^{ \theta, 0, l_0 }_{ \Clip uv{l_L} }
      \bigr)(x)
    &
      \colon 
      L = 1
    \\
      \bigl(\RealV^{\theta,0,l_0}_{\Rect_{l_1},\Rect_{l_2},\dots,\Rect_{l_{L-1}},\Clip uv{l_L}}\bigr)(x)
    &
      \colon 
      L > 1
    \end{cases}
  \end{equation}
  (cf.\ \cref{def:clip,def:relu,def:FFNN}).
\end{definition}


\subsection{Structured description of DNNs}
\label{subsec:structured_description}

\subsubsection{Structured description of DNNs}
\label{subsubsec:structured_description_of_DNNs}

\begin{definition}
  \label{def:ANN}
  We denote by $\ANNs$ the set given by 
  \begin{equation}
    \label{eq:defANN}
  \begin{split}
    \textstyle
    \ANNs
  & \textstyle
    =
    \bigcup_{L \in \N}
    \bigcup_{ (l_0,l_1,\ldots, l_L) \in \N^{L+1} }
    \left(
      \bigtimes_{k = 1}^L (\R^{l_k \times l_{k-1}} \times \R^{l_k})
    \right)
  \end{split}
  \end{equation}
  and we denote by 	
  $
    \paramANN, 
    \lengthANN,  
    \inDimANN, 
    \outDimANN \colon \ANNs \to \N
  $, 
  $
    \hiddenLength \colon \ANNs \to \N_0
  $, 
  and
  $
    \dims \colon \ANNs \to \bigl(\bigcup_{ L = 2 }^\infty \N^{L}\bigr)
  $
  the functions which satisfy
  for all 
  $ 
    L\in\N$, $l_0,l_1,\ldots, l_L \in \N$, 
  $
    \Phi 
    \in  \allowbreak
    \bigl( \bigtimes_{k = 1}^L\allowbreak(\R^{l_k \times l_{k-1}} \times \R^{l_k})\bigr)$
    that
  $\paramANN(\Phi)
    =
    \sum_{k = 1}^L l_k(l_{k-1} + 1) 
  $, $\lengthANN(\Phi)=L$,  $\inDimANN(\Phi)=l_0$,  $\outDimANN(\Phi)=l_L$, $\hiddenLength(\Phi)=L-1$, and $\dims(\Phi)= (l_0,l_1,\ldots, l_L)$.
\end{definition}

\subsubsection{Realizations of DNNs}
\label{subsubsec:realizations_of_dnns}

\begin{definition}[Realization associated to a DNN]
  \label{def:ANNrealization}
  Let $a\in C(\R,\R)$.
  Then we denote by 
  $
    \functionANN{a} \colon \ANNs \to \bigl(\bigcup_{k,l\in\N}\,C(\R^k,\R^l)\bigr)
  $
  the function which satisfies
  for all  $ L\in\N$, $l_0,l_1,\ldots, l_L \in \N$, 
  $
  \Phi 
  =
  ((W_1, B_1),(W_2, B_2),\allowbreak \ldots, (W_L,\allowbreak B_L))
  \in  \allowbreak
  \bigl( \bigtimes_{k = 1}^L\allowbreak(\R^{l_k \times l_{k-1}} \times \R^{l_k})\bigr)
  $,
  $x_0 \in \R^{l_0}, x_1 \in \R^{l_1}, \ldots, x_{L-1} \in \R^{l_{L-1}}$ 
  with $\forall \, k \in \N \cap (0,L) \colon x_k =\activationDim{l_k}(W_k x_{k-1} + B_k)$  
  that
  \begin{equation}
  \label{setting_NN:ass2}
  \functionANN{a}(\Phi) \in C(\R^{l_0},\R^{l_L})\qandq
  ( \functionANN{a}(\Phi) ) (x_0) = W_L x_{L-1} + B_L
  \end{equation}
  (cf.\ \cref{def:ANN,def:multidim_version}).
\end{definition}

\subsubsection{On the connection to the vectorized description of DNNs}

\begin{definition} \label{def:TranslateStructuredIntoVectorizedDescription} 
  We denote by 
  $\MappingStructuralToVectorized\colon \ANNs \to \bigl(\bigcup_{d\in\N} \R^d \bigr)$
  the function which satisfies for all 
  $ L,d \in \N $, 
  $ l_0, l_1, \ldots, l_L \in \N $, 
  $ \Phi = ( (W_1, B_1), (W_2, B_2), \ldots, (W_L, B_L) ) \in \bigl( \bigtimes_{m = 1}^L\allowbreak(\R^{l_{m} \times l_{m-1}} \times \R^{l_{m}})\bigr)$,
  $\theta=(\theta_1,\theta_2,\dots,\theta_d)\in\R^d$,
  $k\in\{1,2,\dots,L\}$
  with $\MappingStructuralToVectorized(\Phi)=\theta$
  that
  \begin{equation}
  \label{eq:translate}
  \begin{split}
    d
    &=
    \paramANN(\Phi),
\qquad
    B_k
    =
    \begin{pmatrix}
      \theta_{(\sum_{i=1}^{k-1}l_i(l_{i-1}+1))+l_kl_{k-1}+1}\\
      \theta_{(\sum_{i=1}^{k-1}l_i(l_{i-1}+1))+l_kl_{k-1}+2}\\
      \theta_{(\sum_{i=1}^{k-1}l_i(l_{i-1}+1))+l_kl_{k-1}+3}\\
      \vdots\\
      \theta_{ (\sum_{i=1}^{k-1}l_i(l_{i-1}+1))+l_kl_{k-1}+l_k }
    \end{pmatrix}
    ,
 \qquad 
     \text{and}
    \\
    W_k
    &=
    \begin{pmatrix}
      \theta_{(\sum_{i=1}^{k-1}l_i(l_{i-1}+1))+1} & \theta_{(\sum_{i=1}^{k-1}l_i(l_{i-1}+1))+2} & \cdots & \theta_{(\sum_{i=1}^{k-1}l_i(l_{i-1}+1))+l_{k-1}} \\
      \theta_{(\sum_{i=1}^{k-1}l_i(l_{i-1}+1))+l_{k-1}+1} & \theta_{(\sum_{i=1}^{k-1}l_i(l_{i-1}+1))+l_{k-1}+2} & \cdots & \theta_{(\sum_{i=1}^{k-1}l_i(l_{i-1}+1))+2l_{k-1}} \\
      \theta_{(\sum_{i=1}^{k-1}l_i(l_{i-1}+1))+2l_{k-1}+1} & \theta_{(\sum_{i=1}^{k-1}l_i(l_{i-1}+1))+2l_{k-1}+2} & \cdots & \theta_{(\sum_{i=1}^{k-1}l_i(l_{i-1}+1))+3l_{k-1}} \\
      \vdots & \vdots & \ddots & \vdots \\
      \theta_{(\sum_{i=1}^{k-1}l_i(l_{i-1}+1))+(l_k-1)l_{k-1}+1} & \theta_{(\sum_{i=1}^{k-1}l_i(l_{i-1}+1))+(l_k-1)l_{k-1}+2} & \cdots & \theta_{(\sum_{i=1}^{k-1}l_i(l_{i-1}+1))+l_kl_{k-1}}
    \end{pmatrix},
  \end{split}
  \end{equation}
  (cf.\ \cref{def:ANN}).
\end{definition}

\begin{lemma}
\label{lem:structtovectB}
	Let 
	  $ a, b \in\N$, 
	$ 
	  W
	  = ( W_{ i, j } )_{ (i,j) \in \{1,2,\dots,a\} \times \{1,2,\dots,b\}}
	  \in\R^{ a \times b }
	$, 
	$ B = ( B_i )_{ i \in \{ 1, 2, \dots, a \} } \in \R^{a} $. 
	Then
  \begin{equation}
    \label{eq:structtovect2B}
    \MappingStructuralToVectorized\bigl( ( (W,B) ) \bigr)
    =
    \bigl( W_{1,1}, W_{1,2}, \dots, W_{1,b},
    W_{2,1}, W_{2,2}, \dots, W_{2,b}, \dots, 
    W_{a,1}, W_{a,2}, \dots, W_{a,b}, 
    B_1, B_2, \dots, B_a \bigr)
  \end{equation}
(cf.\ \cref{def:TranslateStructuredIntoVectorizedDescription}).
\end{lemma}
\begin{proof}[Proof of \cref{lem:structtovectB}]
Observe that \eqref{eq:translate} 
establishes \eqref{eq:structtovect2B}.
The proof of \cref{lem:structtovectB} is thus completed.
\end{proof}

\begin{lemma}
  \label{lem:structtovect}
	Let 
	  $L\in\N$, 
	  $l_0,l_1,\dots,l_L\in\N$, 
	let $W_k=(W_{k,i,j})_{(i,j)\in\{1,2,\dots,l_k\}\times\{1,2,\dots,l_{k-1}\}}\in\R^{l_k\times l_{k-1}}$, $k\in\{1,2,\dots,L\}$,
	and let $B_k=(B_{k,i})_{i\in\{1,2,\dots,l_k\}}\in\R^{l_k}$, $k\in\{1,2,\dots,L\}$.
	Then
\begin{enumerate}[label=(\roman{*})]
\item 
\label{item:b}
it holds for all $k\in\{1,2,\dots,L\}$ that
  \begin{multline}
    \label{eq:structtovect2}
  \MappingStructuralToVectorized\bigl(((W_k,B_k))\bigr)
    =
    \bigl(
      W_{k,1,1},W_{k,1,2},\dots,W_{k,1,l_{k-1}}, 
      W_{k,2,1},W_{k,2,2},\dots,W_{k,2,l_{k-1}}, \dots ,
      \\
      W_{k,l_k,1},W_{k,l_k,2},\dots,W_{k,l_k,l_{k-1}},
      B_{k,1},B_{k,2},\dots,B_{k,l_k}
    \bigr)
  \end{multline}
and
\item 
\label{item:a}
it holds that
  \begin{equation}
    \label{eq:structtovect1}
  \begin{split}
  &
    \MappingStructuralToVectorized\Bigl(
      \bigl(
        (W_1,B_1),(W_2,B_2),\dots,(W_L,B_L)
      \bigr)
    \Bigr)
  \\ &
    =
    \Bigl(
      W_{1,1,1},W_{1,1,2},\dots,W_{1,1,l_0}, 
      \dots ,
      W_{1,l_1,1},W_{1,l_1,2},\dots,W_{1,l_1,l_0},
      B_{1,1},B_{1,2},\dots,B_{1,l_1},
    \\
    &
    \qquad
      W_{2,1,1},W_{2,1,2},\dots,W_{2,1,l_1}, 
      \dots ,
      W_{2,l_2,1},W_{2,l_2,2},\dots,W_{2,l_2,l_1},
      B_{2,1},B_{2,2},\dots,B_{2,l_2},
    \\
    &
    \qquad
      \dots ,
    \\
    &
    \qquad 
      W_{L,1,1},W_{L,1,2},\dots,W_{L,1,l_{L-1}}, 
      \dots 
      W_{L,l_L,1},W_{L,l_L,2},\dots,W_{L,l_L,l_{L-1}},
      B_{L,1},B_{L,2},\dots,B_{L,l_L}
    \Bigr)
  \end{split}
  \end{equation}
\end{enumerate}
(cf.\ \cref{def:TranslateStructuredIntoVectorizedDescription}).
\end{lemma}
\begin{proof}[Proof of \cref{lem:structtovect}]
Note that \cref{lem:structtovectB} proves
\cref{item:b}. Moreover, observe 
that \eqref{eq:translate} 
establishes \cref{item:a}.
The proof of \cref{lem:structtovect} is thus completed.
\end{proof}


\begin{lemma}
  \label{lem:structvsvectgen}
  Let $a\in C(\R,\R)$,
  $\Phi\in\ANNs$,
  $L\in\N$,
  $l_0,l_1,\dots,l_L\in\N$
  satisfy
  $\dims(\Phi)=(l_0,l_1,\dots,l_L)$
  (cf.\ \cref{def:ANN}).
  Then it holds for all 
    $x\in\R^{l_0}$ 
  that
  \begin{equation}
  \label{eq:structvsvectgen}
    (\functionANN{a}(\Phi))(x)
    =
    \begin{cases}
      \bigl(\RealV^{\MappingStructuralToVectorized(\Phi),0,l_0}_{\id_{\R^{l_L}}}\bigr)(x)&\colon L=1\\[0.2cm]
      \bigl(\RealV^{\MappingStructuralToVectorized(\Phi),0,l_0}_{\multdim_{a,l_1},\multdim_{a,l_2},\ldots,\multdim_{a,l_{L-1}},\id_{\R^{l_L}}}\bigr)(x) & \colon L>1
    \end{cases}
  \end{equation}
  (cf.\ \cref{def:ANNrealization,def:TranslateStructuredIntoVectorizedDescription,def:multidim_version,def:FFNN}).
\end{lemma}
\begin{proof}[Proof of Lemma~\ref{lem:structvsvectgen}]
  Throughout this proof
  let 
    $W_1\in\R^{l_1\times l_0}$,
    $B_1\in\R^{l_1}$,
    $W_2\in\R^{l_2\times l_1}$,
    $B_2\in\R^{l_2}$,
    $\dots$,
    $W_L\in\R^{l_L\times l_{L-1}}$,
    $B_L\in\R^{l_L}$
  satisfy
    $\Phi=((W_1,B_1),(W_2,B_2),\dots,(W_L,B_L))$.
  Note that
    \eqref{eq:translate}
  shows that for all
    $k\in\{1,2,\dots,L\}$,
    $x\in\R^{l_{k-1}}$ 
  it holds that
  \begin{equation}
  \label{eq:asaff}
    W_kx+B_k = \bigl(\Aff_{l_k,l_{k-1}}^{\MappingStructuralToVectorized(\Phi),\sum_{i=1}^{k-1}l_i(l_{i-1}+1)}\bigr)( x )
  \end{equation}
  (cf.\ \cref{def:TranslateStructuredIntoVectorizedDescription,def:affine}).
  This demonstrates that for all
    $x_0\in\R^{l_0},\,x_1\in\R^{l_1},\dots,\,x_{L-1}\in\R^{l_{L-1}}$
    with $\forall\, k\in\N\cap(0,L)\colon x_k=\multdim_{a,l_k}(W_kx_{k-1}+B_k)$
  it holds that
  \begin{align}
    &x_{L-1}=\\
    \nonumber
    &\begin{cases}
      x_0&\colon L=1\\
    \bigl(
      \multdim_{a,l_{L-1}}
      \circ
      \Aff_{l_{L-1},l_{L-2}}^{\MappingStructuralToVectorized(\Phi),\sum_{i=1}^{L-2}l_i(l_{i-1}+1)}
      \circ
      \multdim_{a,l_{L-2}}
      \circ
      \Aff_{l_{L-2},l_{L-3}}^{\MappingStructuralToVectorized(\Phi),\sum_{i=1}^{L-3}l_i(l_{i-1}+1)}
      \circ
      \ldots
      \circ
      \multdim_{a,l_1}
      \circ
      \Aff_{l_1,l_{0}}^{\MappingStructuralToVectorized(\Phi),0}
    \bigr)(x_0)&\colon L>1
    \end{cases}
  \end{align}
  (cf.\ \cref{def:multidim_version}).
  Combining \enum{
    this ;
    \eqref{eq:asaff}
  } with \enum{
    \eqref{eq:FFNN} ;
    \eqref{setting_NN:ass2}
  } proves that for all 
    $x_0\in\R^{l_0},\,x_1\in\R^{l_1},\dots,\,x_{L-1}\in\R^{l_{L-1}}$
    with $\forall\, k\in\N\cap(0,L)\colon x_k=\multdim_{a,l_k}(W_kx_{k-1}+B_k)$
  it holds that
  \begin{equation}
  \begin{split}
    \bigl(\functionANN{a}(\Phi)\bigr)(x_0)
    &=
    W_Lx_{L-1}+B_L 
    = 
    \bigl(\Aff_{l_{L},l_{L-1}}^{\MappingStructuralToVectorized(\Phi),\sum_{i=1}^{L-1}l_i(l_{i-1}+1)}\bigr)(x_{L-1})
    \\&=
    \begin{cases}
      \bigl(\RealV^{\MappingStructuralToVectorized(\Phi),0,l_0}_{\id_{\R^{l_L}}} \bigr)(x_0) & \colon L=1 \\[0.2cm]
      \bigl(\RealV^{\MappingStructuralToVectorized(\Phi),0,l_0}_{\multdim_{a,l_1},\multdim_{a,l_2},\ldots,\multdim_{a,l_{L-1}},\id_{\R^{l_L}}} \bigr)(x_0) & \colon L>1
    \end{cases}
  \end{split}
  \end{equation}
  (cf.\ \cref{def:ANNrealization,def:FFNN}).
  The proof of Lemma~\ref{lem:structvsvectgen} is thus completed.
\end{proof}

\begin{cor}
  \label{cor:structvsvect}
  Let $\Phi\in\ANNs
$ (cf.\ \cref{def:ANN}).
  Then it holds for all 
    $x\in\R^{\inDimANN(\Phi)}$ 
  that
  \begin{equation}
  \label{eq:structvsvect}
    \bigl(\UnclippedRealV{\MappingStructuralToVectorized(\Phi)}{\dims(\Phi)}\bigr)(x)
    =
    (\functionANN{\rect}(\Phi))(x)
  \end{equation}
  (cf.\ \cref{def:TranslateStructuredIntoVectorizedDescription,def:rectclippedFFANN,def:relu1,def:ANNrealization}).
\end{cor}
\begin{proof}[Proof of \cref{cor:structvsvect}]
  Note that
    \cref{lem:structvsvectgen},
    \eqref{eq:rectclippedFFANN},
    \eqref{eq:relu},
    and the fact that for all
      $d\in\N$
    it holds that 
      $\Clip{-\infty}\infty d=\id_{\R^d}$
  establish
    \eqref{eq:structvsvect}
  (cf.\ \cref{def:clip}).
  The proof of \cref{cor:structvsvect} is thus completed.
\end{proof}

\subsubsection{Parallelizations of DNNs}
\label{subsubsec:parallelizations_of_dnns}

\begin{samepage}
\begin{definition}[Parallelization of DNNs]
	\label{def:simpleParallelization}
	Let $n\in\N$. Then we denote by 
	\begin{equation}
		\parallelizationSpecial_{n}\colon \big\{(\Phi_1,\Phi_2,\dots, \Phi_n)\in\ANNs^n\colon \lengthANN(\Phi_1)= \lengthANN(\Phi_2)=\ldots =\lengthANN(\Phi_n) \big\}\to \ANNs
	\end{equation}
	the function which satisfies for all  $L\in\N$,
	$(l_{1,0},l_{1,1},\dots, l_{1,L}), (l_{2,0},l_{2,1},\dots, l_{2,L}),\dots,\allowbreak (l_{n,0},\allowbreak l_{n,1},\allowbreak\dots, l_{n,L})\in\N^{L+1}$, 
	$\Phi_1=((W_{1,1}, B_{1,1}),(W_{1,2}, B_{1,2}),\allowbreak \ldots, (W_{1,L},\allowbreak B_{1,L}))\in \bigl( \bigtimes_{k = 1}^L\allowbreak(\R^{l_{1,k} \times l_{1,k-1}} \times \R^{l_{1,k}})\bigr)$, 
		$\Phi_2\allowbreak=\allowbreak((W_{2,1}, B_{2,1}),\allowbreak(W_{2,2}, B_{2,2}),\allowbreak \ldots, (W_{2,L},\allowbreak B_{2,L}))\in \bigl( \bigtimes_{k = 1}^L\allowbreak(\R^{l_{2,k} \times l_{2,k-1}} \times \R^{l_{2,k}})\bigr)$,
\dots,  
		$\Phi_n=((W_{n,1}, B_{n,1}),(W_{n,2}, B_{n,2}),\allowbreak \ldots, \allowbreak(W_{n,L},\allowbreak B_{n,L}))\in \bigl( \bigtimes_{k = 1}^L\allowbreak(\R^{l_{n,k} \times l_{n,k-1}} \times \R^{l_{n,k}})\bigr)$
%
%
	that
	\begin{equation}\label{parallelisationSameLengthDef}
	\begin{split}
	\parallelizationSpecial_{n}(\Phi_1,\Phi_2,\dots,\Phi_n)&=
	\left(
	\pa{\begin{pmatrix}
		W_{1,1}& 0& 0& \cdots& 0\\
		0& W_{2,1}& 0&\cdots& 0\\
		0& 0& W_{3,1}&\cdots& 0\\
		\vdots& \vdots&\vdots& \ddots& \vdots\\
		0& 0& 0&\cdots& W_{n,1}
		\end{pmatrix} ,\begin{pmatrix}B_{1,1}\\B_{2,1}\\B_{3,1}\\\vdots\\ B_{n,1}\end{pmatrix}},\right.
	\\&\quad
		\pa{\begin{pmatrix}
		W_{1,2}& 0& 0& \cdots& 0\\
		0& W_{2,2}& 0&\cdots& 0\\
		0& 0& W_{3,2}&\cdots& 0\\
		\vdots& \vdots&\vdots& \ddots& \vdots\\
		0& 0& 0&\cdots& W_{n,2}
		\end{pmatrix} ,\begin{pmatrix}B_{1,2}\\B_{2,2}\\B_{3,2}\\\vdots\\ B_{n,2}\end{pmatrix}}
	,\dots,
	\\&\quad\left.
		\pa{\begin{pmatrix}
	W_{1,L}& 0& 0& \cdots& 0\\
	0& W_{2,L}& 0&\cdots& 0\\
	0& 0& W_{3,L}&\cdots& 0\\
	\vdots& \vdots&\vdots& \ddots& \vdots\\
	0& 0& 0&\cdots& W_{n,L}
	\end{pmatrix} ,\begin{pmatrix}B_{1,L}\\B_{2,L}\\B_{3,L}\\\vdots\\ B_{n,L}\end{pmatrix}}\right)
	\end{split}
	\end{equation}
	(cf.\ \cref{def:ANN}).
\end{definition}
\end{samepage}



\subsubsection{Basic examples for DNNs}
\label{subsubsec:basic_examples_for_dnns}

\begin{definition}[Linear transformations as DNNs] 
	\label{def:matrixInputDNN}
  Let $m,n\in\N$, $W\in\R^{m\times n}$. Then we denote by
  $\MatrixANN{W}\in\R^{m\times n}\times\R^m$ the pair given by
  $\MatrixANN W=(W,0)$.
\end{definition}

\begin{definition}
  \label{def:ReLu:identity}
  We denote by $\idRelu = (\idRelu_d)_{d \in \N} \colon \N \to \ANNs$ the function which satisfies for all $d \in \N$ that 
  \begin{equation}
  \label{eq:def:id:1}
  \idRelu_1 = \left( \left(\begin{pmatrix}
  1\\
  -1
  \end{pmatrix}, \begin{pmatrix}
  0\\
  0
  \end{pmatrix} \right),
  \Big( \begin{pmatrix}
  1& -1
  \end{pmatrix}, 
  0 \Big)
   \right)  \in \big((\R^{2 \times 1} \times \R^{2}) \times (\R^{1 \times 2} \times \R^1) \big)
  \end{equation} 
  and 
  \begin{equation}
    \label{eq:def:id:2}
  \idRelu_d = \paraANN{d} (\idRelu_1, \idRelu_1, \ldots, \idRelu_1)
  \end{equation}
  (cf.\ \cref{def:ANN,def:simpleParallelization}).
\end{definition}

\subsubsection{Compositions of DNNs}
\label{subsubsec:compositions_of_dnns}

\begin{samepage}
\begin{definition}[Composition of DNNs]
	\label{def:ANNcomposition}
	We denote by $\compANN{(\cdot)}{(\cdot)}\colon\allowbreak \{(\Phi_1,\Phi_2)\allowbreak\in\ANNs\times \ANNs\colon \inDimANN(\Phi_1)=\outDimANN(\Phi_2)\}\allowbreak\to\ANNs$ the function which satisfies for all 
	$ L,\mathfrak{L}\in\N$, $l_0,l_1,\ldots, l_L, \mathfrak{l}_0,\mathfrak{l}_1,\ldots, \mathfrak{l}_\mathfrak{L} \in \N$, 
	$
	\Phi_1
	=
	((W_1, B_1),(W_2, B_2),\allowbreak \ldots, (W_L,\allowbreak B_L))
	\in  \allowbreak
	\bigl( \bigtimes_{k = 1}^L\allowbreak(\R^{l_k \times l_{k-1}} \times \R^{l_k})\bigr)
	$,
	$
	\Phi_2
	=
	((\mathfrak{W}_1, \mathfrak{B}_1),\allowbreak(\mathfrak{W}_2, \mathfrak{B}_2),\allowbreak \ldots, (\mathfrak{W}_\mathfrak{L},\allowbreak \mathfrak{B}_\mathfrak{L}))
	\in  \allowbreak
	\bigl( \bigtimes_{k = 1}^\mathfrak{L}\allowbreak(\R^{\mathfrak{l}_k \times \mathfrak{l}_{k-1}} \times \R^{\mathfrak{l}_k})\bigr)
	$ 
	with $l_0=\inDimANN(\Phi_1)=\outDimANN(\Phi_2)=\mathfrak{l}_{\mathfrak{L}}$
	that
	\begin{equation}\label{ANNoperations:Composition}
	\begin{split}
	&\compANN{\Phi_1}{\Phi_2}=\\&
	\begin{cases} 
			\begin{array}{r}
			\big((\mathfrak{W}_1, \mathfrak{B}_1),(\mathfrak{W}_2, \mathfrak{B}_2),\ldots, (\mathfrak{W}_{\mathfrak{L}-1},\allowbreak \mathfrak{B}_{\mathfrak{L}-1}),
			(W_1 \mathfrak{W}_{\mathfrak{L}}, W_1 \mathfrak{B}_{\mathfrak{L}}+B_{1}),\\ (W_2, B_2), (W_3, B_3),\ldots,(W_{L},\allowbreak B_{L})\big)
			\end{array}
	&: L>1<\mathfrak{L} \\[3ex]
	\big( (W_1 \mathfrak{W}_{1}, W_1 \mathfrak{B}_1+B_{1}), (W_2, B_2), (W_3, B_3),\ldots,(W_{L},\allowbreak B_{L}) \big)
	&: L>1=\mathfrak{L}\\[1ex]
	\big((\mathfrak{W}_1, \mathfrak{B}_1),(\mathfrak{W}_2, \mathfrak{B}_2),\allowbreak \ldots, (\mathfrak{W}_{\mathfrak{L}-1},\allowbreak \mathfrak{B}_{\mathfrak{L}-1}),(W_1 \mathfrak{W}_{\mathfrak{L}}, W_1 \mathfrak{B}_{\mathfrak{L}}+B_{1}) \big)
	&: L=1<\mathfrak{L}  \\[1ex]
	\bigl((W_1 \mathfrak{W}_{1}, W_1 \mathfrak{B}_1+B_{1})\bigr)
	&: L=1=\mathfrak{L} 
	\end{cases}
	\end{split}
	\end{equation}
	(cf.\ \cref{def:ANN}).
\end{definition}
\end{samepage}

\begin{definition}[Maximum norm]
  \label{def:infnorm}
  We denote by $\infnorm\cdot\colon\bigl(\bigcup_{d=1}^\infty \R^d\bigr)\to[0,\infty)$ the function
  which satisfies for all $d\in\N$, $\theta=(\theta_1,\theta_2,\dots,\theta_d)\in\R^d$ that
  \begin{equation}
    \infnorm{\theta}
    =
    \max_{i\in\{1,2,\dots,d\}}\abs{\theta_i}
    .
  \end{equation}
\end{definition}

\begin{lemma}
  \label{lem:composition_infnorm}
  Let
  $ L,\mathfrak{L}\in\N$, $l_0,l_1,\ldots, l_L, \mathfrak{l}_0,\mathfrak{l}_1,\ldots, \mathfrak{l}_\mathfrak{L} \in \N$, 
	$
	\Phi_1
	=
	((W_1, B_1),(W_2, B_2),\allowbreak \ldots, (W_L,\allowbreak B_L))
	\in  \allowbreak
	\bigl( \bigtimes_{k = 1}^L\allowbreak(\R^{l_k \times l_{k-1}} \times \R^{l_k})\bigr)
	$,
	$
	\Phi_2
	=
	((\mathfrak{W}_1, \mathfrak{B}_1),\allowbreak(\mathfrak{W}_2, \mathfrak{B}_2),\allowbreak \ldots, (\mathfrak{W}_\mathfrak{L},\allowbreak \mathfrak{B}_\mathfrak{L}))
	\in  \allowbreak
	\bigl( \bigtimes_{k = 1}^\mathfrak{L}\allowbreak(\R^{\mathfrak{l}_k \times \mathfrak{l}_{k-1}} \times \R^{\mathfrak{l}_k})\bigr)
	$.
  Then
  \begin{equation}
    \label{eq:composition_infnorm}
    \infnorm{\MappingStructuralToVectorized(\compANN{\Phi_1}{\Phi_2})}
    \leq
    \max\bigl\{
      \infnorm{\MappingStructuralToVectorized(\Phi_1)},
      \infnorm{\MappingStructuralToVectorized(\Phi_2)},
      \asinfnorm{\MappingStructuralToVectorized\bigl(((W_1\mf W_{\mf L},W_1\mf B_{\mf L}+B_1))\bigr)}
    \bigr\}
  \end{equation}
  (cf.\ \cref{def:ANNcomposition,def:TranslateStructuredIntoVectorizedDescription,def:infnorm}).
\end{lemma}
\begin{proof}[Proof of \cref{lem:composition_infnorm}]
  Note that \enum{
    \eqref{ANNoperations:Composition} ;
    \cref{lem:structtovect}
  }[establish] \eqref{eq:composition_infnorm}.
  The proof of \cref{lem:composition_infnorm} is thus completed.
\end{proof}

\subsubsection{Powers and extensions of DNNs}

\begin{definition}\label{def:identityMatrix}
	Let $d\in\N$. Then we denote by $\idMatrix_{d}\in \R^{d\times d}$ the identity matrix in $\R^{d\times d}$.
\end{definition}

\begin{definition}\label{def:iteratedANNcomposition}
	We denote by $(\cdot)^{\bullet n}\colon \{\Phi\in \ANNs\colon \inDimANN(\Phi)=\outDimANN(\Phi)\}\allowbreak\to\ANNs$, $n\in\N_0$, the functions
		 which satisfy for all $n\in\N_0$, $\Phi\in\ANNs$ with $\inDimANN(\Phi)=\outDimANN(\Phi)$ that 
	\begin{equation}\label{iteratedANNcomposition:equation}
		\begin{split}
		\Phi^{\bullet n}=
		\begin{cases} \big(\idMatrix_{\outDimANN(\Phi)},(0,0,\dots, 0)\big)\in\R^{\outDimANN(\Phi)\times \outDimANN(\Phi)}\times \R^{\outDimANN(\Phi)}
		&: n=0 \\
		\,\compANN{\Phi}{(\Phi^{\bullet (n-1)})} &: n\in\N
		\end{cases}
		\end{split}
	\end{equation}	
	(cf.\ \cref{def:ANN,def:identityMatrix,def:ANNcomposition}).
\end{definition}

\begin{definition}[Extension of DNNs]\label{def:ANNenlargement}
	Let $L\in\N$, $\Psi\in \ANNs$ satisfy $\inDimANN(\Psi)=\outDimANN(\Psi)$.
	Then
	we denote by $\longerANN{L,\Psi}\colon \{\Phi\in\ANNs\colon (\lengthANN(\Phi)\le L \andShort \outDimANN(\Phi)=\inDimANN(\Psi)) \}\to \ANNs$ the function which satisfies for all $\Phi\in\ANNs$ with $\lengthANN(\Phi)\le L$ and $\outDimANN(\Phi)=\inDimANN(\Psi)$ that
		\begin{equation}\label{ANNenlargement:Equation}
	\longerANN{L,\Psi}(\Phi)=	 \compANN{\bigl(\Psi^{\bullet (L-\lengthANN(\Phi))}\bigr)}{\Phi}
	\end{equation}
	(cf.\ \cref{def:ANN,def:iteratedANNcomposition,def:ANNcomposition}).
\end{definition}

\begin{lemma}
  \label{lem:extension_dims}
  Let $d,\mf i,L,\mf L\in\N$,
  $l_0,l_1,\dots,l_{L-1}\in\N$,
  $\Phi,\Psi\in\ANNs$
  satisfy
    $\mf L\geq L$,
    $\dims(\Phi)=(l_0,l_1,\dots,l_{L-1},d)$
    and $\dims(\Psi)=(d,\mf i,d)$
  (cf.\ \cref{def:ANN}).
  Then it holds that
    $\dims(\longerANN{\mf L,\Psi}(\Phi))\in\N^{\mf L+1}$
    and
    \begin{equation}
      \label{eq:extension_dims_claim}
      \dims(\longerANN{\mf L,\Psi}(\Phi))
      =
      \begin{cases}
        (l_0,l_1,\dots,l_{L-1},d) & \colon \mf L=L \\
        (l_0,l_1,\dots,l_{L-1},\mf i,\mf i,\dots,\mf i,d) & \colon \mf L>L
      \end{cases}
    \end{equation}
  (cf.\ \cref{def:ANNenlargement}).
\end{lemma}
\begin{proof}[Proof of \cref{lem:extension_dims}]
  Observe that \enum{
    item (i) in \cite[Lemma~2.13]{Grohs2019ANNCalculus}
  }[ensure] that 
    $\lengthANN(\Psi^{\bullet(\mf L-L)})=\mf L-L+1$,
    $\dims(\Psi^{\bullet(\mf L-L)})\in\N^{\mf L-L+2}$,
    and
    \begin{equation}
      \dims(\Psi^{\bullet(\mf L-L)})
      =
      \begin{cases}
        (d,d) & \colon \mf L=L \\
        (d,\mf i,\mf i,\dots,\mf i,d) & \colon \mf L>L
      \end{cases}
    \end{equation}
  (cf.\ \cref{def:iteratedANNcomposition}).
  Combining this with \enum{
    \cite[Proposition~2.6]{Grohs2019ANNCalculus}
  } shows that
  $\lengthANN(\compANN{(\Psi^{\bullet(\mf L-L)})}{\Phi})
  =
  \lengthANN(\Psi^{\bullet(\mf L-L)})+\lengthANN(\Phi)-1
  =
  \mf L$,
  $\dims(\compANN{(\Psi^{\bullet(\mf L-L)})}{\Phi})\in\N^{\mf L+1}$,
  and
  \begin{equation}
    \dims(\compANN{(\Psi^{\bullet(\mf L-L)})}{\Phi})
    =
    \begin{cases}
      (l_0,l_1,\dots,l_{L-1},d) & \colon \mf L=L \\
      (l_0,l_1,\dots,l_{L-1},\mf i,\mf i,\dots,\mf i,d) & \colon \mf L>L.
    \end{cases}
  \end{equation}
  \enum{
    This ;
    \eqref{ANNenlargement:Equation}
  }[establish]
    \eqref{eq:extension_dims_claim}.
  The proof of \cref{lem:extension_dims} is thus completed.
\end{proof}

\begin{lemma}
  \label{lem:extension}
  Let $d,L\in\N$,
  $\Phi\in\ANNs$
  satisfy
    $L\geq\lengthANN(\Phi)$
    and $d=\outDimANN(\Phi)$
  (cf.\ \cref{def:ANN}).
  Then
  \begin{equation}
    \infnorm{\MappingStructuralToVectorized(\longerANN{L,\idRelu_d}(\Phi))}
    \leq
    \max\{1,\infnorm{\MappingStructuralToVectorized(\Phi)}\}
  \end{equation}
  (cf.\ \cref{def:ReLu:identity,def:ANNenlargement,def:TranslateStructuredIntoVectorizedDescription,def:infnorm}).
\end{lemma}
\begin{proof}[Proof of \cref{lem:extension}]
  Throughout this proof 
  assume w.l.o.g.\ that $L>\lengthANN(\Phi)$
  and let $l_0,l_1,\dots,l_{L-\lengthANN(\Phi)+1}\in\N$
  satisfy $(l_0,l_1,\dots,l_{L-\lengthANN(\Phi)+1})=(d,2d,2d,\dots,2d,d)$.
  Note that \enum{
    \cite[Lemma~3.16]{GrohsJentzenSalimova2019}
  }[ensure] that $\dims(\idRelu_d)=(d,2d,d)\in\N^3$
  (cf.\ \cref{def:ReLu:identity}).
  \enum{
    Item (i) in \cite[Lemma 2.13]{Grohs2019ANNCalculus}
  } hence establishes that
  \begin{equation}
    \lengthANN((\idRelu_d)^{\bullet(L-\lengthANN(\Phi)})=L-\lengthANN(\Phi)+1
    \qandq
    \dims((\idRelu_d)^{\bullet(L-\lengthANN(\Phi)})=(l_0,l_1,\dots,l_{L-\lengthANN(\Phi)+1})\in\N^{L-\lengthANN(\Phi)+2}
  \end{equation}
  (cf.\ \cref{def:iteratedANNcomposition}).
  \enum{
    This
  }[show] that there exist
    $W_k\in\R^{l_k\times l_{k-1}}$, $k\in\{1,2,\dots,L-\lengthANN(\Phi)+1\}$,
    and $B_k\in\R^{l_k}$, $k\in\{1,2,\dots,L-\lengthANN(\Phi)+1\}$,
    which satisfy
  \begin{equation}
    (\idRelu_d)^{\bullet ( L-\lengthANN(\Phi))}=((W_1,B_1),(W_2,B_2),\dots,(W_{L-\lengthANN(\Phi)+1},B_{L-\lengthANN(\Phi)+1})).
  \end{equation}
  Next observe that \enum{
    \eqref{parallelisationSameLengthDef} ;
    \eqref{eq:def:id:1} ;
    \eqref{eq:def:id:2} ;
    \eqref{ANNoperations:Composition} ;
    \eqref{iteratedANNcomposition:equation} ;
  }[demonstrate] that
  \begin{equation}
    \label{eq:extension1}
    \begin{split}
    W_1
    &=
    \pmat{
      1&0&\cdots&0\\
      -1&0&\cdots&0\\
      0&1&\cdots&0\\
      0&-1&\cdots&0\\
      \vdots&\vdots&\ddots&\vdots\\
      0&0&\cdots&1\\
      0&0&\cdots&-1
    }\in\R^{(2d)\times d}
    \\\text{and}\qquad
    W_{L-\lengthANN(\Phi)+1}
    &=
    \pmat{
      1&-1&0&0&\cdots&0&0\\
      0&0&1&-1&\cdots&0&0\\
      \vdots&\vdots&\vdots&\vdots&\ddots&\vdots&\vdots\\
      0&0&0&0&\cdots&1&-1
    }\in\R^{d\times (2d)}
    .
    \end{split}
  \end{equation}
  Moreover, note that \enum{
    \eqref{parallelisationSameLengthDef} ;
    \eqref{eq:def:id:1} ;
    \eqref{eq:def:id:2} ;
    \eqref{ANNoperations:Composition} ;
    \eqref{iteratedANNcomposition:equation} ;
  }[prove] that for all
    $k\in\N\cap(1,L-\lengthANN(\Phi)+1)$
  it holds that
  \begin{equation}
    \label{eq:extension4}
    \begin{split}
    W_k
    &=
    \underbrace{
    \pmat{
      1&-1&0&0&\cdots&0&0\\
      0&0&1&-1&\cdots&0&0\\
      \vdots&\vdots&\vdots&\vdots&\ddots&\vdots&\vdots\\
      0&0&0&0&\cdots&1&-1
    }
    }_{\in\R^{d\times(2d)}}
    \underbrace{
    \pmat{
      1&0&\cdots&0\\
      -1&0&\cdots&0\\
      0&1&\cdots&0\\
      0&-1&\cdots&0\\
      \vdots&\vdots&\ddots&\vdots\\
      0&0&\cdots&1\\
      0&0&\cdots&-1
    }
    }_{\in\R^{(2d)\times d}}
    \\&=
    \pmat{
      1&-1&0&0&\cdots&0&0\\
      -1&1&0&0&\cdots&0&0\\
      0&0&1&-1&\cdots&0&0\\
      0&0&-1&1&\cdots&0&0\\
      \vdots&\vdots&\vdots&\vdots&\ddots&\vdots&\vdots\\
      0&0&0&0&\cdots&1&-1\\
      0&0&0&0&\cdots&-1&1
    }
    \in\R^{(2d)\times (2d)}
    .
    \end{split}
  \end{equation}
  In addition, observe that \enum{
    \eqref{eq:def:id:1} ;
    \eqref{eq:def:id:2} ;
    \eqref{parallelisationSameLengthDef} ;
    \eqref{iteratedANNcomposition:equation} ;
    \eqref{ANNoperations:Composition} ;
  }[show] that for all
    $k\in\N\cap[1,L-\lengthANN(\Phi)]$
  it holds that
  \begin{equation}
    \label{eq:extension5}
    B_k=0\in\R^{2d}
    \qandq
    B_{L-\lengthANN(\Phi)+1}=0\in\R^d
    .
  \end{equation}
  Combining \enum{
    this ;
    \eqref{eq:extension1} ;
    \eqref{eq:extension4}
  } establishes that
  \begin{equation}
    \label{eq:extension3}
    \asinfnorm{\MappingStructuralToVectorized\bigl((\idRelu_d)^{\bullet ( L-\lengthANN(\Phi))}\bigr)}
    =
    1
  \end{equation}
  (cf.\ \cref{def:TranslateStructuredIntoVectorizedDescription,def:infnorm}).
  Furthermore, note that 
    \eqref{eq:extension1}
  demonstrates that for all
    $k\in\N$,
    $\mf W=(w_{i,j})_{(i,j)\in\{1,2,\dots,d\}\times\{1,2,\dots,k\}}\in\R^{d\times k}$
  it holds that
  \begin{equation}
    \label{eq:extension2}
    W_1\mf W
    =
    \pmat{
      w_{1,1} & w_{1,2} & \cdots & w_{1,k}\\
      -w_{1,1} & -w_{1,2} & \cdots & -w_{1,k}\\
      w_{2,1} & w_{2,2} & \cdots & w_{2,k}\\
      -w_{2,1} & -w_{2,2} & \cdots & -w_{2,k}\\
      \vdots & \vdots & \ddots &\vdots\\
      w_{d,1} & w_{d,2} & \cdots & w_{d,k}\\
      -w_{d,1} & -w_{d,2} & \cdots & -w_{d,k}
    }\in\R^{(2d)\times k}
    .
  \end{equation}
  Next observe that
  \enum{
    \eqref{eq:extension1} ;
    \eqref{eq:extension5} ;
  }[show] that for all
    $\mf B=(b_1,b_2,\dots,b_d)\in\R^{d}$
  it holds that
  \begin{equation}
    W_1 \mf B+B_1
    =
    \pmat{
      1&0&\cdots&0\\
      -1&0&\cdots&0\\
      0&1&\cdots&0\\
      0&-1&\cdots&0\\
      \vdots&\vdots&\ddots&\vdots\\
      0&0&\cdots&1\\
      0&0&\cdots&-1
    }\pmat{b_1\\b_2\\\vdots\\b_d}
    =
    \pmat{b_1\\-b_1\\b_2\\-b_2\\\vdots\\b_d\\-b_d}
    \in\R^{2d}
    .
  \end{equation}
  Combining
    this
  with
    \eqref{eq:extension2}
  proves that for all
    $k\in\N$,
    $\mf W\in\R^{d\times k}$,
    $\mf B\in\R^{d}$
  it holds that
  \begin{equation}
    \asinfnorm{\MappingStructuralToVectorized\bigl(((W_1\mf W,W_1\mf B+B_1))\bigr)}
    =
    \asinfnorm{\MappingStructuralToVectorized\bigl(((\mf W,\mf B))\bigr)}.
  \end{equation}
  \enum{
    This ;
    \cref{lem:composition_infnorm} ;
    \eqref{eq:extension3} ;
  }[establish] that
  \begin{equation}
    \begin{split}
    &\infnorm{\MappingStructuralToVectorized(\longerANN{L,\idRelu_d}(\Phi))}
    =
    \asinfnorm{\MappingStructuralToVectorized\bigl(\compANN{((\idRelu_d)^{\bullet(L-\lengthANN(\Phi))})}{\Phi}\bigr)}
    \\&\leq
    \max\bigl\{
      \asinfnorm{\MappingStructuralToVectorized\bigl((\idRelu_d)^{\bullet(L-\lengthANN(\Phi))}\bigr)},
      \asinfnorm{\MappingStructuralToVectorized(\Phi)}
    \bigr\}
    =
    \max\{
      1,
      \infnorm{\MappingStructuralToVectorized(\Phi)}
    \}
    \end{split}
  \end{equation}
  (cf.\ \cref{def:ANNenlargement}).
  The proof of \cref{lem:extension} is thus completed.
\end{proof}

\subsubsection{Embedding DNNs in larger architectures}


\begin{lemma}
\label{lem:embednet}
  Let 
    $a\in C(\R,\R)$,
    $L\in\N$, 
    $l_0,l_1,\dots,l_L,\mf l_0,\mf l_1,\dots,\mf l_L\in\N$
  satisfy for all 
    $k\in\{1,2,\dots,L\}$
  that
    $\mf l_0=l_0$,
    $\mf l_L=l_L$,
    and $\mf l_k\geq l_k$,
  for every $k\in\{1,2,\dots,L\}$ let
    $W_k=(W_{k,i,j})_{(i,j)\in\{1,2,\dots,l_k\}\times\{1,2,\dots,l_{k-1}\}}\in\R^{l_k\times l_{k-1}}$,
    $\mf W_k=(\mf W_{k,i,j})_{(i,j)\in\{1,2,\dots,\mf l_k\}\times\{1,2,\dots,\mf l_{k-1}\}}\in\R^{\mf l_k\times \mf l_{k-1}}$,
    $B_k=(B_{k,i})_{i\in\{1,2,\dots,l_k\}}\in\R^{l_k}$,
    $\mf B_k=(\mf B_{k,i})_{i\in\{1,2,\dots,\mf l_k\}}\in\R^{\mf l_k}$,
  assume for all
    $k\in\{1,2,\dots,L\}$,
    $i\in\{1,2,\dots,l_k\}$,
    $j\in\N\cap(0,l_{k-1}]$
  that 
    $\mf W_{k,i,j}=W_{k,i,j}$ and $\mf B_{k,i}=B_{k,i}$,
  and assume for all 
    $k\in\{1,2,\dots,L\}$,
    $i\in\{1,2,\dots,l_k\}$,
    $j\in\N\cap(l_{k-1},\mf l_{k-1}+1)$
  that
    $\mf W_{k,i,j}=0$.
  Then
  \begin{equation}
    \functionANN{a}\bigl(((W_1,B_1),(W_2,B_2),\dots,(W_L,B_L))\bigr)
    =
    \functionANN{a}\bigl(((\mf W_1,\mf B_1),(\mf W_2,\mf B_2),\dots,(\mf W_L,\mf B_L))\bigr)
  \end{equation}
  (cf.\ \cref{def:ANNrealization}).
\end{lemma}
\begin{proof}[Proof of \cref{lem:embednet}]
  Throughout this proof let
    $\pi_k\colon\R^{\mf l_k}\to\R^{l_k}$, $k\in\{0,1,\dots,L\}$,
  satisfy for all
    $k\in\{0,1,\dots,L\}$,
    $x=(x_1,x_2,\dots,x_{\mf l_k})$
  that
  \begin{equation}
    \label{eq:embednet.defpik}
    \pi_k(x)
    =
    (x_1,x_2,\dots,x_{l_k})
    .
  \end{equation}
  Observe that
    the hypothesis that $\mf l_0=l_0$ and $\mf l_L=l_L$
  shows that
  \begin{equation}
    \label{eq:embednet2}
    \functionANN{a}\bigl(((W_1,B_1),(W_2,B_2),\dots,(W_L,B_L))\bigr)\in C(\R^{\mf l_0},\R^{\mf l_L})
  \end{equation}
  (cf.\ \cref{def:ANNrealization}).
  Furthermore, note that
    the hypothesis that for all
      $k\in\{1,2,\dots,L\}$,
      $i\in\{1,2,\dots,l_k\}$,
      $j\in\N\cap(l_{k-1},\mf l_{k-1}+1)$
    it holds that
      $\mf W_{k,i,j}=0$
  ensures that for all 
    $k\in\{1,2,\dots,L\}$,
    $x=(x_1,x_2,\dots,x_{\mf l_{k-1}})\in\R^{\mf l_{k-1}}$
  it holds that
  \begin{equation}
  \begin{split}
    \pi_k(\mf W_k x+\mf B_k)
    &=
    \left(
      \left[\sum_{i=1}^{\mf l_{k-1}} \mf W_{k,1,i}x_{i}\right]+\mf B_{k,1},
      \left[\sum_{i=1}^{\mf l_{k-1}} \mf W_{k,2,i}x_{i}\right]+\mf B_{k,2},
      \dots,
      \left[\sum_{i=1}^{\mf l_{k-1}} \mf W_{k,l_k,i}x_{i}\right]+\mf B_{k,l_k}
    \right)
    \\&=
    \left(
      \left[\sum_{i=1}^{l_{k-1}} \mf W_{k,1,i}x_{i}\right]+\mf B_{k,1},
      \left[\sum_{i=1}^{l_{k-1}} \mf W_{k,2,i}x_{i}\right]+\mf B_{k,2},
      \dots,
      \left[\sum_{i=1}^{l_{k-1}} \mf W_{k,l_k,i}x_{i}\right]+\mf B_{k,l_k}
    \right)
    \!.
  \end{split}
  \end{equation}
  Combining 
    this
  with 
    the hypothesis that for all
      $k\in\{1,2,\dots,L\}$,
      $i\in\{1,2,\dots,l_k\}$,
      $j\in\N\cap(0,l_{k-1}]$
    it holds that 
      $\mf W_{k,i,j}=W_{k,i,j}$ and $\mf B_{k,i}=B_{k,i}$
  shows that for all
    $k\in\{1,2,\dots,L\}$,
    $x=(x_1,x_2,\dots,x_{\mf l_{k-1}})\in\R^{\mf l_{k-1}}$
  it holds that
  \begin{equation}
    \label{eq:embednet1}
  \begin{split}
    \pi_k(\mf W_kx+\mf B_k)
    &=
    \Biggl(
      \left[\sum_{i=1}^{l_{k-1}} W_{k,1,i}x_{i}\right]+B_{k,1},
      \left[\sum_{i=1}^{l_{k-1}} W_{k,2,i}x_{i}\right]+B_{k,2},
      \dots,
      \left[\sum_{i=1}^{l_{k-1}} W_{k,l_k,i}x_{i}\right]+B_{k,l_k}
    \Biggr)
    \\&=
    W_k(\pi_{k-1}(x))+B_k
    .
  \end{split}
  \end{equation}
  %
  Moreover, observe that \enum{
    \eqref{eq:embednet.defpik} ;
    \eqref{multidim_version:Equation}
  }[ensure] that for all
    $k\in\{0,1,\dots,L\}$,
    $x=(x_1,x_2,\dots,x_{\mf l_k})\in\R^{\mf l_k}$
  it holds that
  \begin{equation}
    \pi_k(\activationDim{\mf l_k}(x))
    =
    \pi_k(a(x_1),a(x_2),\dots,a(x_{\mf l_k}))
    =
    (a(x_1),a(x_2),\dots,a(x_{l_k}))
    =
    \activationDim{l_k}(\pi_k(x))
    .  
  \end{equation}
  Combining \enum{
    this ;
    \eqref{eq:embednet1}
  } demonstrates that for all
    $x_0
    \in \R^{\mf l_0}
    ,\, 
    x_1
    \in \R^{\mf l_1}
    , 
    \ldots,\, 
    x_{L-1}
    \in \R^{\mf l_{L-1}}
    $,
    $k\in\N\cap(0,L)$
    with $\forall \, m \in \N \cap (0,L) \colon x_m =\activationDim{\mf l_m}(\mf W_m x_{m-1} + \mf B_m)$ 
  it holds that
  \begin{equation}
    \pi_k(x_k)
    =
    \pi_k(\activationDim{\mf l_k}(\mf W_k x_{k-1} + \mf B_k))
    =
    \activationDim{l_k}(\pi_k(\mf W_kx_{k-1}+\mf B_k))
    =
    \activationDim{l_k}(W_k\pi_{k-1}(x_{k-1})+B_k)
  \end{equation}
  (cf.\ \cref{def:multidim_version}).
    The hypothesis that $l_0=\mf l_0$ and $l_L=\mf l_L$
    and \eqref{eq:embednet1}
    therefore
  prove that for all
    $x_0\in\R^{\mf l_0},\, 
    x_1\in\R^{\mf l_1}, 
    \ldots,\, 
    x_{L-1}\in \R^{\mf l_{L-1}}$
    with $\forall \, k \in \N \cap (0,L) \colon x_k =\activationDim{\mf l_k}(\mf W_k x_{k-1} + \mf B_k)$ 
  it holds that
  \begin{equation}
  \begin{split}
    \bigl(\functionANN{a}\bigl(((W_1,B_1),(W_2,B_2),\dots,(W_L,B_L))\bigr)\bigr)(x_0)
    &=
    \bigl(\functionANN{a}\bigl(((W_1,B_1),(W_2,B_2),\dots,(W_L,B_L))\bigr)\bigr)(\pi_0(x_0))
    \\&=
    W_L\pi_{L-1}(x_{L-1})+B_L
    \\&=
    \pi_L(\mf W_Lx_{L-1}+\mf B_L)
    =
    \mf W_Lx_{L-1}+\mf B_L
    \\&=
    \bigl(\functionANN{a}\bigl(((\mf W_1,\mf B_1),(\mf W_2,\mf B_2),\dots,(\mf W_L,\mf B_L))\bigr)\bigr)(x_0)
  \end{split}
  \end{equation}
  (cf.\ \cref{def:ANNrealization}).
  The proof of Lemma~\ref{lem:embednet} is thus completed.
\end{proof}

\begin{lemma}
  \label{lem:embednet_exist}
  Let $a\in C(\R,\R)$,
  $L\in\N$, 
  $l_0,l_1,\dots,l_L,\mf l_0,\mf l_1,\dots,\mf l_L\in\N$
  satisfy for all 
    $k\in\{1,2,\dots,L\}$
  that
    $\mf l_0=l_0$,
    $\mf l_L=l_L$,
    and $\mf l_k\geq l_k$
  and let $\Phi\in\ANNs$ satisfy
    $\dims(\Phi)=(l_0,l_1,\dots,l_L)$
  (cf. \cref{def:ANN}).
  Then there exists $\Psi\in\ANNs$ such that
  \begin{equation}
    \dims(\Psi)=(\mf l_0,\mf l_1,\dots,\mf l_L),
    \qquad
    \infnorm{\MappingStructuralToVectorized(\Psi)}
    =
    \infnorm{\MappingStructuralToVectorized(\Phi)},
    \qandq
    \functionANN{a}(\Psi)
    =
    \functionANN{a}(\Phi)
  \end{equation}
  (cf.\ \cref{def:TranslateStructuredIntoVectorizedDescription,def:infnorm,def:ANNrealization}).
\end{lemma}
\begin{proof}[Proof of \cref{lem:embednet_exist}]
  Throughout this proof let
    $B_k=(B_{k,i})_{i\in\{1,2,\dots,l_k\}}\in\R^{l_k}$, $k\in\{1,2,\dots,L\}$,
    and $W_k=(W_{k,i,j})_{(i,j)\in\{1,2,\dots,l_k\}\times\{1,2,\dots,l_{k-1}\}}\in\R^{l_k\times l_{k-1}}$, $k\in\{1,2,\dots,L\}$,
  satisfy
	  $\Phi=((W_1,B_1),(W_2,B_2),\dots,(W_L,\allowbreak B_L))$
  and let
    $\mf W_k=(\mf W_{k,i,j})_{(i,j)\in\{1,2,\dots,\mf l_k\}\times\{1,2,\dots,\mf l_{k-1}\}}\in\R^{\mf l_k\times \mf l_{k-1}}$, $k\in\{1,2,\dots,L\}$,
    and $\mf B_k=(\mf B_{k,i})_{i\in\{1,2,\dots,\mf l_k\}}\in\R^{\mf l_k}$, $k\in\{1,2,\dots,L\}$,
  satisfy for all
    $k\in\{1,2,\dots,L\}$,
    $i\in\{1,2,\dots,\mf l_k\}$,
    $j\in\{1,2,\dots,\mf l_{k-1}\}$
  that
  \begin{equation}
  \label{eq:embednet2.WB}
    \mf W_{k,i,j}
    =
    \begin{cases}
      W_{k,i,j}&\colon (i\leq l_k)\land(j\leq l_{k-1})\\
      0 & \colon (i>l_k)\lor(j>l_{k-1})
    \end{cases}
    \qquad\text{and}\qquad
    \mf B_{k,i}
    =
    \begin{cases}
      B_{k,i} &\colon i\leq l_k\\
      0 & \colon i>l_k.
    \end{cases}
  \end{equation}
  Note that
    \eqref{eq:defANN}
  ensures that
  $
    ((\mf W_1,\mf B_1),(\mf W_2,\mf B_2),\dots,(\mf W_L,\mf B_L))
    \in
    \bigl(\textstyle\bigtimes_{i=1}^{L}(\R^{\mf l_i\times \mf l_{i-1}}\times\R^{\mf l_i})\bigr)
    \subseteq
    \ANNs
  $ and
  \begin{equation}
    \dims\bigl(((\mf W_1,\mf B_1),(\mf W_2,\mf B_2),\dots,(\mf W_L,\mf B_L))\bigr)
    =
    (\mf l_0,\mf l_1,\dots,\mf l_{L}).
  \end{equation}
  Furthermore, observe that \enum{
    \cref{lem:structtovect} ;
    \eqref{eq:embednet2.WB}
  }[show] that
  \begin{equation}
    \infnorm{\MappingStructuralToVectorized\bigl(((\mf W_1,\mf B_1),(\mf W_2,\mf B_2),\dots,(\mf W_L,\mf B_L))\bigr)}
    =
    \infnorm{\MappingStructuralToVectorized(\Phi)}
  \end{equation}
  (cf.\ \cref{def:TranslateStructuredIntoVectorizedDescription,def:infnorm}).
  In addition, note that
    \cref{lem:embednet}
  establishes that
  \begin{equation}
    \functionANN{a}(\Phi)
    =
    \functionANN{a}\bigl(((W_1,B_1),(W_2,B_2),\dots,(W_L,\allowbreak B_L))\bigr)
    =
    \functionANN{a}\bigl(((\mf W_1,\mf B_1),(\mf W_2,\mf B_2),\dots,(\mf W_L,\mf B_L))\bigr)
  \end{equation}
  (cf.\ \cref{def:ANNrealization}).
  The proof of \cref{lem:embednet_exist} is thus completed.
\end{proof}

\begin{lemma}
  \label{lem:embednet3}
  Let $L,\mf L\in\N$, 
  $l_0,l_1,\dots,l_L,\mf l_0,\mf l_1,\dots,\mf l_{\mf L}\in\N$
  satisfy for all
    $i\in\N\cap[0,L)$
  that
    $\mf L\geq L$
    $\mf l_0=l_0$,
    $\mf l_{\mf L}=l_L$,
    and $\mf l_i\geq l_i$,
  assume for all 
    $i\in \N\cap(L-1,\mf L)$
  that
    $\mf l_i\geq 2l_L$,
  and let $\Phi\in\ANNs$
  satisfy $\dims(\Phi)=(l_0,l_1,\dots,l_L)$
  (cf.\ \cref{def:ANN}).
  Then there exists
    $\Psi\in\ANNs$
  such that
  \begin{equation}
    \label{eq:embednet3.claim}
    \dims(\Psi)=(\mf l_0,\mf l_1,\dots,\mf l_{\mf L}),
    \qquad
    \infnorm{\MappingStructuralToVectorized(\Psi)}
    \leq
    \max\{
      1,\infnorm{\MappingStructuralToVectorized(\Phi)}
    \},
    \qandq
    \functionANN{\rect}(\Psi)=\functionANN{\rect}(\Phi)
  \end{equation}
  (cf.\ \cref{def:TranslateStructuredIntoVectorizedDescription,def:infnorm,def:relu1,def:ANNrealization}).
\end{lemma}  
\begin{proof}[Proof of \cref{lem:embednet3}]
  Throughout this proof
  let $\Xi\in\ANNs$ satisfy
    $\Xi=\longerANN{\mf L,\idRelu_{l_L}}(\Phi)$
  (cf.\ \cref{def:ReLu:identity,def:ANNenlargement}).
  Note that \enum{
    item (i) in \cite[Lemma~3.16]{GrohsJentzenSalimova2019}
  }[demonstrate] that
  $\dims(\idRelu_{l_L})=(l_L,2l_L,l_L)\in\N^3$.
  Combining
    this
  with
    \cref{lem:extension_dims}
  shows that
    $\dims(\Xi)\in\N^{\mf L+1}$
    and
    \begin{equation}
      \label{eq:embednet3.dimsXi}
      \dims(\Xi)
      =
      \begin{cases}
        (l_0,l_1,\dots,l_L) & \colon \mf L=L \\
        (l_0,l_1,\dots,l_{L-1},2l_L,2l_L,\dots,2l_L,l_L) & \colon \mf L>L.
      \end{cases}
    \end{equation}
  Moreover, observe that
    \cref{lem:extension}
    (with
      $d\is l_L$,
      $L\is\mf L$,
      $\Phi\is\Phi$
    in the notation of \cref{lem:extension})
  establishes that
  \begin{equation}
    \label{eq:embednet3.norm}
    \infnorm{\MappingStructuralToVectorized(\Xi)}
    \leq
    \max\{
      1,
      \infnorm{\MappingStructuralToVectorized(\Phi)}
    \}
  \end{equation}
  (cf.\ \cref{def:TranslateStructuredIntoVectorizedDescription,def:infnorm}).
  In addition, note that \enum{
    item (iii) in \cite[Lemma 3.16]{GrohsJentzenSalimova2019}
  }[ensure] that for all
    $x\in\R^{l_L}$
  it holds that
  \begin{equation}
    (\functionANN{\rect}(\idRelu_{l_L}))(x)=x
  \end{equation}
  (cf.\ \cref{def:relu1,def:ANNrealization}).
  \enum{
    This ;
    item (ii) in \cite[Lemma 2.14]{Grohs2019ANNCalculus}
  }[prove] that
  \begin{equation}
    \label{eq:embednet3.real}
    \functionANN{\rect}(\Xi)
    =
    \functionANN{\rect}(\Phi).
  \end{equation}
  In the next step, we observe that
  \enum{
    \eqref{eq:embednet3.dimsXi} 
    ;
    the hypothesis that for all 
      $i\in[0,L)$
    it holds that 
      $\mf l_0=l_0$,
      $\mf l_{\mf L}=l_L$,
      and $\mf l_i\leq l_i$
    ;
    the hypothesis that for all 
      $i\in\N\cap(L-1,\mf L)$
    it holds that
      $\mf l_i\geq 2l_L$
    ;
    \cref{lem:embednet_exist}
    (with
      $a\is\rect$,
      $L\is\mf L$,
      $(l_0,l_1,\dots,l_L)\is\dims(\Xi)$,
      $(\mf l_0,\mf l_1,\dots,\mf l_{\mf L})\is(\mf l_0,\mf l_1,\dots,\mf l_{\mf L})$,
      $\Phi\is\Xi$
    in the notation of \cref{lem:embednet_exist})
  }[ensure] that there exists
    $\Psi\in\ANNs$
  such that
  \begin{equation}
    \dims(\Psi)
    =
    (\mf l_0,\mf l_1,\dots,\mf l_{\mf L}),
    \qquad
    \infnorm{\MappingStructuralToVectorized(\Psi)}
    =
    \infnorm{\MappingStructuralToVectorized(\Xi)},
    \qandq
    \functionANN{\rect}(\Psi)
    =
    \functionANN{\rect}(\Xi)
    .
  \end{equation}
  Combining
    this
  with \enum{
    \eqref{eq:embednet3.norm} ;
    \eqref{eq:embednet3.real}
  } establishes \eqref{eq:embednet3.claim}.
  The proof of \cref{lem:embednet3} is thus completed.
\end{proof}

\begin{lemma}
  \label{lem:embednet_vectorized}
  Let 
  $u\in[-\infty,\infty)$,
  $v\in(u,\infty]$,
  $L,\mf L,d,\mf d\in\N$, 
  $\theta\in\R^d$,
  $l_0,l_1,\dots,l_L,\mf l_0,\mf l_1,\dots,\mf l_{\mf L}\in\N$
  satisfy for all
    $i\in\N\cap[0,L)$
  that
    $d\geq\sum_{i=1}^L l_i(l_{i-1}+1)$,
    $\mf d\geq\sum_{i=1}^{\mf L} \mf l_i(\mf l_{i-1}+1)$,
    $\mf L\geq L$,
    $\mf l_0=l_0$,
    $\mf l_{\mf L}=l_L$,
    and $\mf l_i\geq l_i$
  and assume for all 
    $i\in \N\cap(L-1,\mf L)$
  that
    $\mf l_i\geq 2l_L$.
  Then there exists $\vartheta\in\R^{\mf d}$ such that
  \begin{equation}
    \label{eq:embednet_vectorized:claim}
    \infnorm{\vartheta}\leq\max\{1,\infnorm{\theta}\}
    \qandq
    \ClippedRealV{\vartheta}{(\mf l_0,\mf l_1,\dots,\mf l_{\mf L})}uv
    =
    \ClippedRealV{\theta}{(l_0,l_1,\dots,l_L)}uv
  \end{equation}
  (cf.\ \cref{def:infnorm,def:rectclippedFFANN}).
\end{lemma}
\begin{proof}[Proof of \cref{lem:embednet_vectorized}]
  Throughout this proof let
    $\eta_1,\eta_2,\dots,\eta_d\in\R$
  satisfy
  \begin{equation}
    \label{eq:embednet_vectorized.4}
    \theta=(\eta_1,\eta_2,\dots,\eta_d)
  \end{equation}
  and let
    $\Phi\in\bigl(\bigtimes_{i=1}^L\R^{l_i\times l_{i-1}}\times\R^{l_i}\bigr)$
  satisfy
  \begin{equation}
    \label{eq:embednet_vectorized.3}
    \MappingStructuralToVectorized(\Phi)=(\eta_1,\eta_2,\dots,\eta_{\paramANN(\Phi)})
  \end{equation}
  (cf.\ \cref{def:TranslateStructuredIntoVectorizedDescription}).
  Note that \enum{
    \cref{lem:embednet3}
  }[establish] that there exists $\Psi\in\ANNs$ which satisfies
  \begin{equation}
    \label{eq:embednet_vectorized.2}
    \dims(\Psi)=(\mf l_0,\mf l_1,\dots,\mf l_{\mf L}),
    \qquad
    \infnorm{\MappingStructuralToVectorized(\Psi)}
    \leq
    \max\{
      1,\infnorm{\MappingStructuralToVectorized(\Phi)}
    \},
    \qandq
    \functionANN{\rect}(\Psi)=\functionANN{\rect}(\Phi)
  \end{equation}
  (cf.\ \cref{def:ANN,def:infnorm,def:relu1,def:ANNrealization}).
  Next let
    $\vartheta=(\vartheta_1,\vartheta_2,\dots,\vartheta_{\mf d})\in\R^{\mf d}$
  satisfy 
  \begin{equation}
    \label{eq:embednet_vectorized.1}
    (\vartheta_1,\vartheta_2,\dots,\vartheta_{\paramANN(\Psi)})
    =
    \MappingStructuralToVectorized(\Psi)
    \qandq
    \Forall i\in\N\cap(\paramANN(\Psi),\mf d+1)\colon \vartheta_i=0
    .
  \end{equation}
  Note that \enum{
    \eqref{eq:embednet_vectorized.4} ;
    \eqref{eq:embednet_vectorized.3} ;
    \eqref{eq:embednet_vectorized.2} ;
    \eqref{eq:embednet_vectorized.1} ;
  }[show] that
  \begin{equation}
    \label{eq:embednet_vectorized:normclaim}
    \infnorm{\vartheta}
    =
    \infnorm{\MappingStructuralToVectorized(\Psi)}
    \leq
    \max\{1,\infnorm{\MappingStructuralToVectorized(\Phi)}\}
    \leq
    \max\{1,\infnorm\theta\}
    .
  \end{equation}
  Next observe that \enum{
    \cref{cor:structvsvect} ;
    \eqref{eq:embednet_vectorized.3} ;
  }[establish] that for all
    $x\in\R^{l_0}$
  it holds that
  \begin{equation}
    \label{eq:embednet_vectorized.realPhi}
    \bigl(\UnclippedRealV{\theta}{(l_0,l_1,\dots,l_L)}\bigr)(x)
    =
    \bigl(\UnclippedRealV{\MappingStructuralToVectorized(\Phi)}{\dims(\Phi)}\bigr)(x)
    =
    (\functionANN{\rect}(\Phi))(x)
    .
  \end{equation}
  In addition, observe that \enum{
    \cref{cor:structvsvect} ;
    \eqref{eq:embednet_vectorized.2} ;
    \eqref{eq:embednet_vectorized.1} ;
  }[prove] that for all
    $x\in\R^{\mf l_0}$
  it holds that
  \begin{equation}
    \bigl(\UnclippedRealV{\vartheta}{(\mf l_0,\mf l_1,\dots,\mf l_{\mf L})}\bigr)(x)
    =
    \bigl(\UnclippedRealV{\MappingStructuralToVectorized(\Psi)}{\dims(\Psi)}\bigr)(x)
    =
    (\functionANN{\rect}(\Psi))(x)
    .
  \end{equation}
  Combining \enum{
    this ;
    \eqref{eq:embednet_vectorized.realPhi}
  } with \enum{
    \eqref{eq:embednet_vectorized.2} ;
    the hypothesis that $\mf l_0=l_0$ and $\mf l_{\mf L}=l_L$
  } demonstrates that
  \begin{equation}
    \UnclippedRealV{\theta}{(l_0,l_1,\dots,l_L)}
    =
    \UnclippedRealV{\vartheta}{(\mf l_0,\mf l_1,\dots,\mf l_{\mf L})}
    .
  \end{equation}
  Hence, we obtain that
  \begin{equation}
    \ClippedRealV{\theta}{(l_0,l_1,\dots,l_L)}uv
    =
    \Clip uv{l_L}\circ\UnclippedRealV{\theta}{(l_0,l_1,\dots,l_L)}
    =
    \Clip uv{\mf l_{\mf L}}\circ\UnclippedRealV{\vartheta}{(\mf l_0,\mf l_1,\dots,\mf l_{\mf L})}
    =
    \ClippedRealV{\vartheta}{(\mf l_0,\mf l_1,\dots,\mf l_{\mf L})}uv
  \end{equation}
  (cf.\ \cref{def:clip}).
  \enum{
    This ;
    \eqref{eq:embednet_vectorized:normclaim}
  }[establish] \eqref{eq:embednet_vectorized:claim}.
  The proof of \cref{lem:embednet_vectorized} is thus completed.
\end{proof}

\subsection{Local Lipschitz continuity of the parametrization function}
\label{subsec:lipschitz}

\begin{lemma}
  \label{lem:max_estimate}
  Let $ a, x, y \in \R $. Then 
  \begin{equation}
    \left| \max\{ x, a \} - \max\{ y, a \} \right|
  \leq
    \max\{ x, y \} - \min\{ x, y \}
  =
    | x - y |
    .
  \end{equation}
  \end{lemma}

  \begin{proof}[Proof of \cref{lem:max_estimate}]
  Observe that 
  \begin{equation}
  \begin{split}
  &
    \left| \max\{ x, a \} - \max\{ y, a \} \right|
    =
    \left| 
      \max\{ \max\{ x, y \}, a \} - \max\{ \min\{ x, y \}, a \} 
    \right|
  \\ &
  =
    \max\bigl\{ \max\{ x, y \}, a \bigr\} 
    - 
    \max\bigl\{ \min\{ x, y \}, a \bigr\} 
  \\ &
  =
    \max\Bigl\{ 
      \max\{ x, y \}
      - 
      \max\bigl\{ \min\{ x, y \}, a \bigr\} 
      , 
      a 
      - 
      \max\bigl\{ \min\{ x, y \}, a \bigr\} 
    \Bigr\} 
  \\ &
  \leq
    \max\Bigl\{ 
      \max\{ x, y \}
      - 
      \max\bigl\{ \min\{ x, y \}, a \bigr\} 
      , 
      a 
      - 
      a 
    \Bigr\} 
  \\ & 
  =
    \max\Bigl\{ 
      \max\{ x, y \}
      - 
      \max\bigl\{ \min\{ x, y \}, a \bigr\} 
      , 
      0
    \Bigr\} 
  \leq
    \max\Bigl\{ 
      \max\{ x, y \}
      - 
      \min\{ x, y \}
      , 
      0
    \Bigr\} 
  \\ &
  =
    \max\{ x, y \}
    - 
    \min\{ x, y \}
  =
    \left|
      \max\{ x, y \}
      - 
      \min\{ x, y \}
    \right|
  =
    \left| x - y \right|
    .
  \end{split}
  \end{equation}
  The proof of \cref{lem:max_estimate} is thus completed.
  \end{proof}

  \begin{cor}
  \label{cor:min_estimate}
  Let $ a, x, y \in \R $. Then 
  \begin{equation}
    \left| \min\{ x, a \} - \min\{ y, a \} \right|
  \leq
    \max\{ x, y \} - \min\{ x, y \}
  =
    | x - y |
    .
  \end{equation}
  \end{cor}

  \begin{proof}[Proof of \cref{cor:min_estimate}]
  Note that \cref{lem:max_estimate} ensures that
  \begin{equation}
  \begin{split}
    \left| \min\{ x, a \} - \min\{ y, a \} \right|
  & =
    \left| - \left( \min\{ x, a \} - \min\{ y, a \} \right) \right|
  =
    \left| \max\{ - x, - a \} - \max\{ - y, - a \} \right|
  \\ & 
  \leq
    \left| ( - x ) - ( - y ) \right|
  =
    \left| x - y \right|
    .
  \end{split}
  \end{equation}
  The proof of \cref{cor:min_estimate} is thus completed.
  \end{proof}
  
  \begin{lemma}
    \label{result:Clip_contraction}
    \label{lem:clipcontr}
    Let $d\in\N$, 
      $u\in[-\infty,\infty)$, 
      $v\in(u,\infty]$.
    Then it holds for all $x,y\in\R^d$ that
    \begin{equation}
      \infnorm{
        \Clip uvd(x)
        -
        \Clip uvd(y)
      }
      \leq
      \infnorm{x-y}
    \end{equation}
    (cf.\ \cref{def:clip,def:infnorm}).
  \end{lemma}
  \begin{proof}[Proof of \cref{result:Clip_contraction}]
    Note that
      \cref{lem:max_estimate},
      \cref{cor:min_estimate},
      and the fact that for all
        $x\in\R$
      it holds that
        $\max\{-\infty,x\}=x=\min\{x,\infty\}$
    show that for all
      $x,y\in\R$
    it holds that
    \begin{equation}
      \abs{\clip uv(x)-\clip uv(y)}
      =
      \abs{
        \max\{u,\min\{x,v\}\}
        -
        \max\{u,\min\{y,v\}\}
      }
      \leq
      \abs{
        \min\{x,v\}
        -
        \min\{y,v\}
      }
      \leq
      \abs{x-y}
    \end{equation}
    (cf.\ \cref{def:clip1}).
    Hence, we obtain that for all
      $x=(x_1,x_2,\dots,x_d),y=(y_1,y_2,\dots,y_d)\in\R^d$
    it holds that
    \begin{equation}
      \infnorm{\Clip uvd(x)-\Clip uvd(y)}
      =
      \max_{i\in\{1,2,\dots,d\}} \abs{\clip uv(x_i)-\clip uv(y_i)}
      \leq
      \max_{i\in\{1,2,\dots,d\}} \abs{x_i-y_i}
      =
      \infnorm{x-y}
    \end{equation}
    (cf.\ \cref{def:clip,def:infnorm}).
    The proof of \cref{result:Clip_contraction} is thus completed.
  \end{proof}

  \begin{lemma}
    \label{result:ReLU_contraction}
    Let $ d \in \N $. Then it holds for all $ x, y \in \R^d $ that
    \begin{equation}
    \label{eq:inf_norm}
      \infnorm{
        \Rect_d( x ) 
        -
        \Rect_d( y )
      }
      \leq 
      \infnorm{ x - y }
    \end{equation}
    (cf.\ \cref{def:relu,def:infnorm}).
  \end{lemma}
  \begin{proof}[Proof of \cref{result:ReLU_contraction}]
    Note that 
      \cref{result:Clip_contraction}
      and the fact that
        $\Rect_d=\Clip 0{\infty}d$
    establish
    \eqref{eq:inf_norm}.
    The proof of \cref{result:ReLU_contraction} is thus completed.
  \end{proof}

\begin{lemma}
  \label{result:row_sum_norm}
  Let $ a, b \in \N $, 
  $ M = ( M_{ i, j } )_{ (i,j) \in \{ 1, 2, \dots, a \} \times \{ 1, 2, \dots, b \} } \in \R^{ a \times b } $.
  Then 
  \begin{equation}
    \sup_{ v \in \R^b \backslash \{ 0 \} }
    \left[
      \frac{ 
        \infnorm{ M v }
      }{ 
        \infnorm{ v } 
      }
    \right]
  =
    \max_{ i \in \{ 1, 2, \dots, a \} }
    \left[
      \textstyle
      \sum\limits_{ j = 1 }^b
      \displaystyle
      \left|
        M_{ i, j } 
      \right|
    \right]
  \leq 
    b 
    \left[
      \max_{ i \in \{ 1, 2, \dots, a \} }
      \max_{ j \in \{ 1, 2, \dots, b \} }
      \left|
        M_{ i, j } 
      \right|
    \right]
  \end{equation}
  (cf.\ \cref{def:infnorm}).
  \end{lemma}
  \begin{proof}[Proof of \cref{result:row_sum_norm}]
  Observe that 
  \begin{equation}
  \begin{split}
    \sup_{ v \in \R^b }
    \left[
      \frac{ 
        \infnorm{ M v }
      }{ 
        \infnorm{ v } 
      }
    \right]
  & =
    \sup_{ 
      v \in \R^b , \, \infnorm{ v } \leq 1
    }
    \infnorm{ M v }
  =
    \sup_{ 
      v = ( v_1, v_2, \dots, v_b ) \in [-1,1]^b 
    }
    \infnorm{ M v }
  \\ & 
  =
    \sup_{ 
      v = ( v_1, v_2, \dots, v_b ) \in [-1,1]^b 
    }
    \left(
      \max_{ i \in \{ 1, 2, \dots, a \} }
      \left|
      \textstyle
        \sum\limits_{ j = 1 }^b
      \displaystyle
        M_{ i, j } v_j
      \right|
    \right)
  \\ & 
  =
    \max_{ i \in \{ 1, 2, \dots, a \} }
    \left(
    \sup_{ 
      v = ( v_1, v_2, \dots, v_b ) \in [-1,1]^b 
    }
    \left|
      \textstyle
      \sum\limits_{ j = 1 }^b
      \displaystyle
      M_{ i, j } v_j
    \right|
    \right)
  =
    \max_{ i \in \{ 1, 2, \dots, a \} }
    \left(
      \textstyle
      \sum\limits_{ j = 1 }^b
      \displaystyle
    \left|
      M_{ i, j } 
    \right|
    \right)
  \end{split}
  \end{equation}
  (cf.\ \cref{def:infnorm}).
  The proof of \cref{result:row_sum_norm} is thus completed.
  \end{proof}

  \begin{theorem}
    \label{thm:RealNNLipsch}
    Let 
    $ a \in \R $,
    $ b \in [a,\infty) $,
    $ d, L \in \N$,
    $ l = (l_0,l_1,\dots,l_L) \in \N^{ L + 1 } $
    satisfy
    \begin{equation}
      d \geq 
      \sum_{ k =1 }^L 
      l_k ( l_{ k - 1 } + 1 ) 
      .
    \end{equation}
    Then it holds for all
    $ \theta, \vartheta \in \R^d $
    that
    \begin{equation}
    \begin{split}
    &
      \sup_{ x \in [a,b]^{ l_0 } }
      \infnorm{ \UnclippedRealV{\theta}{ l }(x)-\UnclippedRealV\vartheta{ l }(x) }
    \\ 
    & \leq
      \max\{ 1, \abs{a}, \abs{b} \} 
      \infnorm{ \theta - \vartheta }
      \left[
        \prod_{ m = 0 }^{ L - 1 }
        ( l_m + 1 )
      \right]
      \left[ 
        \sum_{ n = 0 }^{ L - 1 }
        \left(
          \max\{1,\infnorm{ \theta }^n\}
          \,
          \infnorm{ \vartheta }^{ L - 1 - n }
        \right)
      \right]
    \\ 
    & \leq 
      L
      \max\{ 1, \abs{a}, \abs{b} \} 
      \left( 
        \max\{ 1, \infnorm{ \theta }, \infnorm{ \vartheta } \}
      \right)^{ L - 1 }
      \left[
        \prod_{ m = 0 }^{ L - 1 }
        (
          l_m + 1 
        )
      \right]
      \infnorm{ \theta - \vartheta }
    \\
    & \leq 
      L 
      \max\{ 1, \abs{a}, \abs{b} \}
      \, 
      ( \infnorm{ l } + 1 )^L 
      \,
      ( \max\{ 1, \infnorm\theta,\infnorm\vartheta\} )^{ L - 1 }
      \,
      \infnorm{ \theta - \vartheta }
    \end{split}
    \end{equation}
    (cf.\ \cref{def:rectclippedFFANN,def:infnorm}).
    \end{theorem}
    \begin{proof}[Proof of \cref{thm:RealNNLipsch}]
    Throughout this proof let 
    $ \theta_j = ( \theta_{j, 1}, \theta_{j, 2}, \dots, \theta_{j, d} ) \in \R^d $, 
    $ j \in \{ 1, 2 \} $,
    let $ \mathfrak{d} \in \N $ satisfy that
    \begin{equation}
      \mathfrak{d} = 
      \sum_{ k =1 }^L 
      l_k ( l_{ k - 1 } + 1 ) 
      ,
    \end{equation} 
    let 
    $ W_{ j, k } \in \R^{ l_k \times l_{ k - 1 } } $, 
    $ k \in \{ 1, 2, \dots, L \} $, 
    $ j \in \{ 1, 2 \} $, 
    and 
    $ B_{ j, k } \in \R^{ l_k } $, $ k \in \{ 1, 2, \dots, L \} $, $ j \in \{ 1, 2 \} $,
    satisfy 
    for all $ j \in \{ 1, 2 \} $,
    $ k \in \{ 1, 2, \dots, L \} $
    that
    \begin{equation}
      \MappingStructuralToVectorized\big( 
        \big( 
          ( W_{j, 1}, B_{j, 1} ), 
          ( W_{j, 2}, B_{j, 2} ), 
          \dots, 
          ( W_{j, L}, B_{j, L} ) 
        \big)
      \big) = 
      ( \theta_{j, 1}, \theta_{j, 2}, \dots, \theta_{ {j, \mathfrak{d}} } )
      ,
    \end{equation}
    let $ \phi_{ j, k } \in \ANNs $, $ k \in \{ 1, 2, \dots, L \} $, $ j \in \{ 1, 2 \} $, 
    satisfy for all $ j \in \{ 1, 2 \} $,
    $ k \in \{ 1, 2, \dots, L \} $ that
    \begin{equation}
      \phi_{ j, k } 
      =
      \big( 
        ( W_{j, 1}, B_{j, 1} ), 
        ( W_{j, 2}, B_{j, 2} ), 
        \dots, 
        ( W_{j, k}, B_{j, k} ) 
      \big)
      \in
      \left[
        \textstyle
          \bigtimes_{ i = 1 }^k
          \left(
            \R^{ l_i \times l_{ i - 1 } } \times \R^{ l_i }
          \right)
        \displaystyle
      \right]
      ,
    \end{equation}
    let 
    $ 
      D = [a,b]^{ l_0 }
    $,
    let $ \mathfrak{m}_{ j, k } \in [0,\infty) $, 
    $ j \in \{ 1, 2 \} $,
    $ k \in \{ 0, 1, \dots, L \} $,
    satisfy for all 
    $ j \in \{ 1, 2 \} $, 
    $ k \in \{ 0, 1, \dots, L \} $ that 
    \begin{equation}
      \mathfrak{m}_{ j, k } 
      = 
      \begin{cases}
        \max\{ 1, | a | , | b | \}
      &
        \colon
        k = 0
      \\ 
        \max\!\left\{ 
          1 
          , 
          \sup\nolimits_{ x \in D } 
          \,
          \infnorm{ 
            ( \functionANN{\rect}( \phi_{ j, k } ) )( x ) 
          }
        \right\} 
      &
        \colon 
        k > 0
        ,
      \end{cases}
    \end{equation}
    and let $ \mathfrak{e}_k \in [0,\infty) $, $ k \in \{ 0, 1, \dots, L \} $, 
    satisfy for all $ k \in \{ 0, 1, \dots, L \} $ that
    \begin{equation}
      \mathfrak{e}_k 
      =
      \begin{cases}
        0
      &
        \colon
        k = 0
      \\
        \sup_{ x \in D }
        \,
        \infnorm{
          ( \functionANN{\rect}( \phi_{ 1, k } ) )( x ) 
          -
          ( \functionANN{\rect}( \phi_{ 2, k } ) )( x ) 
        }
      &
        \colon
        k > 0
      \end{cases}
    \end{equation}
    (cf.\ \cref{def:TranslateStructuredIntoVectorizedDescription,def:relu1,def:ANNrealization,def:infnorm}).
    Note that \cref{result:row_sum_norm} demonstrates that
    \begin{equation}
    \label{eq:e0_estimate_DNN}
    \begin{split}
      \mathfrak{e}_1
    & 
      =
      \sup_{ x \in D }
      \,
      \infnorm{
        ( \functionANN{\rect}( \phi_{ 1, 1 } ) )( x ) 
        -
        ( \functionANN{\rect}( \phi_{ 2, 1 } ) )( x ) 
      }
    =
      \sup_{ x \in D }
      \,
      \infnorm{
        \left(
          W_{ 1, 1 } x + B_{ 1, 1 }
        \right)
        -
        \left(
          W_{ 2, 1 } x + B_{ 2, 1 }
        \right)
      }
    \\ &
    \leq
      \left[
        \sup_{ x \in D }
        \,
        \infnorm{
          (
            W_{ 1, 1 } - W_{ 2, 1 }
          )
          x
        }
      \right]
      +
      \infnorm{
        B_{ 1, 1 }
        -
        B_{ 2, 1 }
      }
    \\ & 
    \leq 
      \left[
        \sup_{ v \in \R^{ l_0 } \backslash \{ 0 \} }
        \left(
          \frac{
            \infnorm{
              (
                W_{ 1, 1 } - W_{ 2, 1 }
              )
              v
            }
          }{
            \infnorm{ v }
          }
        \right)
      \right]
      \left[ 
        \sup_{ x \in D }
        \,
        \infnorm{ x }
      \right]
      +
      \infnorm{
        B_{ 1, 1 }
        -
        B_{ 2, 1 }
      }
    \\ & 
    \leq 
      l_0
      \,
      \infnorm{ \theta_1 - \theta_2 }
      \max\{ | a | , | b | \}
      +
      \infnorm{
        B_{ 1, 1 }
        -
        B_{ 2, 1 }
      }
    \leq 
      l_0
      \,
      \infnorm{ \theta_1 - \theta_2 }
      \max\{ | a | , | b | \}
      +
      \infnorm{ \theta_1 - \theta_2 }
    \\ &
    =
      \infnorm{ \theta_1 - \theta_2 }
      \left(
        l_0
        \max\{ | a | , | b | \}
        +
        1
      \right)
    \leq 
      \mathfrak{m}_{ 1, 0 }
      \,
      \infnorm{ \theta_1 - \theta_2 }
      \left(
        l_0
        +
        1
      \right)
      .
    \end{split}
    \end{equation}
    Moreover, observe that the triangle inequality 
    assures that 
    for all $ k \in \{ 1, 2, \dots, L \} \cap (1,\infty) $
    it holds that
    \begin{equation}
    \begin{split}
      \mathfrak{e}_k
    & 
      =
      \sup_{ x \in D }
      \,
      \infnorm{
        ( \functionANN{\rect}( \phi_{ 1, k } ) )( x ) 
        -
        ( \functionANN{\rect}( \phi_{ 2, k } ) )( x ) 
      }
    \\ &
    =
      \sup_{ x \in D }
      \,
      \asinfnorm{
        \left[
          W_{ 1, k }\Big(
            \Rect_{ l_{ k - 1 } }\big(
              ( \functionANN{\rect}( \phi_{ 1, k - 1 } ) )( x ) 
            \big)
          \Big)
          +
          B_{ 1, k }
        \right]
        -
        \left[
          W_{ 2, k }\Big(
            \Rect_{ l_{ k - 1 } }\big(
              ( \functionANN{\rect}( \phi_{ 2, k - 1 } ) )( x ) 
            \big)
          \Big)
          +
          B_{ 2, k }
        \right]
      }
    \\ &
    \leq
      \left[
        \sup_{ x \in D }
        \asinfnorm{
          W_{ 1, k }\Big(
            \Rect_{ l_{ k - 1 } }\big(
              ( \functionANN{\rect}( \phi_{ 1, k - 1 } ) )( x ) 
            \big)
          \Big)
          -
          W_{ 2, k }\Big(
            \Rect_{ l_{ k - 1 } }\big(
              ( \functionANN{\rect}( \phi_{ 2, k - 1 } ) )( x ) 
            \big)
          \Big)
        }
      \right]
      +
      \infnorm{ \theta_1 - \theta_2 }
      .
    \end{split}
    \end{equation}
    The triangle inequality hence implies that 
    for all 
    $ j \in \{ 1, 2 \} $, $ k \in \{ 1, 2, \dots, L \} \cap (1,\infty) $
    it holds that
    \begin{equation}
    \begin{split}
      \mathfrak{e}_k
    & 
    \leq
      \left[
        \sup_{ x \in D }
        \asinfnorm{
          \big(
            W_{ 1, k } - W_{ 2, k }
          \big)\big(
            \Rect_{ l_{ k - 1 } }\big(
              ( \functionANN{\rect}( \phi_{ j, k - 1 } ) )( x ) 
            \big)
          \big)
        }
      \right]
    \\
    &
    +
      \left[
        \sup_{ x \in D }
        \asinfnorm{
          W_{ 3 - j, k }\Big(
            \Rect_{ l_{ k - 1 } }\big(
              ( \functionANN{\rect}( \phi_{ 1, k - 1 } ) )( x ) 
            \big)
            -
            \Rect_{ l_{ k - 1 } }\big(
              ( \functionANN{\rect}( \phi_{ 2, k - 1 } ) )( x ) 
            \big)
          \Big)
        }
      \right]
      +
      \infnorm{ \theta_1 - \theta_2 }
    \\ & 
    \leq
      \left[ 
        \sup_{ v \in \R^{ l_{ k - 1 } } \backslash \{ 0 \} }
        \left(
        \frac{
          \infnorm{ ( W_{ 1, k } - W_{ 2, k } ) v }
        }{
          \infnorm{ v }
        }
        \right)
      \right]
      \left[
        \sup_{ x \in D }
        \asinfnorm{
          \Rect_{ l_{ k - 1 } }\big(
            ( \functionANN{\rect}( \phi_{ j, k - 1 } ) )( x ) 
          \big)
        }
      \right]
      +
      \infnorm{ \theta_1 - \theta_2 }
    \\
    &
    +
      \left[ 
        \sup_{ v \in \R^{ l_{ k - 1 } } \backslash \{ 0 \} }
        \left(
        \frac{
          \infnorm{ W_{ 3 - j, k } v }
        }{
          \infnorm{ v }
        }
        \right)
      \right]
      \left[
        \sup_{ x \in D }
        \asinfnorm{
            \Rect_{ l_{ k - 1 } }\big(
              ( \functionANN{\rect}( \phi_{ 1, k - 1 } ) )( x ) 
            \big)
            -
            \Rect_{ l_{ k - 1 } }\big(
              ( \functionANN{\rect}( \phi_{ 2, k - 1 } ) )( x ) 
            \big)
        }
      \right]
      .
    \end{split}
    \end{equation}
    \cref{result:row_sum_norm} and \cref{result:ReLU_contraction} therefore show that 
    for all $ j \in \{ 1, 2 \} $, $ k \in \{ 1, 2, \dots, L \} \cap (1,\infty) $
    it holds that
    \begin{equation}
    \begin{split}
      \mathfrak{e}_k
    & 
    \leq
      l_{ k - 1 }
      \, 
      \infnorm{ 
        \theta_1 - \theta_2
      }
      \left[
        \sup_{ x \in D }
        \asinfnorm{
          \Rect_{ l_{ k - 1 } }\big(
            ( \functionANN{\rect}( \phi_{ j, k - 1 } ) )( x ) 
          \big)
        }
      \right]
      +
      \infnorm{ \theta_1 - \theta_2 }
    \\
    &
    +
      l_{ k - 1 }
      \,
      \infnorm{ \theta_{ 3 - j } }
      \left[
        \sup_{ x \in D }
        \asinfnorm{
            \Rect_{ l_{ k - 1 } }\big(
              ( \functionANN{\rect}( \phi_{ 1, k - 1 } ) )( x ) 
            \big)
            -
            \Rect_{ l_{ k - 1 } }\big(
              ( \functionANN{\rect}( \phi_{ 2, k - 1 } ) )( x ) 
            \big)
        }
      \right]
    \\ & \leq
      l_{ k - 1 }
      \, 
      \infnorm{ 
        \theta_1 - \theta_2
      }
      \left[
        \sup_{ x \in D }
        \asinfnorm{
          ( \functionANN{\rect}( \phi_{ j, k - 1 } ) )( x ) 
        }
      \right]
      +
      \infnorm{ \theta_1 - \theta_2 }
    \\
    &
    +
      l_{ k - 1 }
      \,
      \infnorm{ \theta_{ 3 - j } }
      \left[
        \sup_{ x \in D }
        \asinfnorm{
          ( \functionANN{\rect}( \phi_{ 1, k - 1 } ) )( x ) 
          -
          ( \functionANN{\rect}( \phi_{ 2, k - 1 } ) )( x ) 
        }
      \right]
    \\ &
    \leq 
      \infnorm{ \theta_1 - \theta_2 }
      \left(
        l_{ k - 1 } \, \mathfrak{m}_{ j, k - 1 } + 1 
      \right)
      + 
      l_{ k - 1 } 
      \, \infnorm{ \theta_{ 3 - j } }
      \, \mathfrak{e}_{ k - 1 }
      .
    \end{split}
    \end{equation}
    Hence, we obtain that
    for all $ j \in \{ 1, 2 \} $, $ k \in \{ 1, 2, \dots, L \} \cap (1,\infty) $
    it holds that
    \begin{equation}
    \begin{split}
      \mathfrak{e}_k
    & 
    \leq 
      \mathfrak{m}_{ j, k - 1 } 
      \,
      \infnorm{ \theta_1 - \theta_2 }
      \left(
        l_{ k - 1 } + 1 
      \right)
      + 
      l_{ k - 1 } 
      \, \infnorm{ \theta_{ 3 - j } }
      \, \mathfrak{e}_{ k - 1 }
      .
    \end{split}
    \end{equation}
    Combining this with \eqref{eq:e0_estimate_DNN}, 
    the fact that $ \mathfrak{e}_0 = 0 $, 
    and the fact that $ \mathfrak{m}_{ 1, 0 } = \mathfrak{m}_{ 2, 0 } $
    demonstrates that 
    for all $ j \in \{ 1, 2 \} $, 
    $ k \in \{ 1, 2, \dots, L \} $
    it holds that
    \begin{equation}
    \begin{split}
      \mathfrak{e}_k
    & 
    \leq 
      \mathfrak{m}_{ j, k - 1 } 
      \left(
        l_{ k - 1 } + 1 
      \right)
      \infnorm{ \theta_1 - \theta_2 }
      + 
      l_{ k - 1 } 
      \, \infnorm{ \theta_{ 3 - j } }
      \, \mathfrak{e}_{ k - 1 }
      .
    \end{split}
    \end{equation}
    This shows that for all 
    $
      j = ( j_n )_{ n \in \{ 0, 1, \dots, L \} } \colon \{ 0, 1, \dots, L \} \to \{ 1, 2 \} 
    $
    and all
    $ k \in \{ 1, 2, \dots, L \} $ 
    it holds that
    \begin{equation}
    \begin{split}
      \mathfrak{e}_k
    & 
    \leq 
      \mathfrak{m}_{ j_{ k - 1 }, k - 1 } 
      \left(
        l_{ k - 1 } + 1 
      \right)
      \infnorm{ \theta_1 - \theta_2 }
      + 
      l_{ k - 1 } 
      \, \infnorm{ \theta_{ 3 - j_{ k - 1 } } }
      \, \mathfrak{e}_{ k - 1 }
      .
    \end{split}
    \end{equation}
    Therefore, we obtain that for all 
    $
      j = ( j_n )_{ n \in \{ 0, 1, \dots, L \} } \colon \{ 0, 1, \dots, L \} \to \{ 1, 2 \} 
    $
    and all
    $ k \in \{ 1, 2, \dots, L \} $ 
    it holds that
    \begin{equation}
    \label{eq:error_difference_estimate_DNN}
    \begin{split}
      \mathfrak{e}_k
    & \leq
      \sum_{ n = 0 }^{ k - 1 }
      \left(
        \left[
          \prod_{ m = n + 1 }^{ k - 1 }
          \big(
            l_m
            \, \infnorm{ \theta_{ 3 - j_m } }
          \big)
        \right]
        \mathfrak{m}_{ j_n, n }
        \left(
          l_n + 1 
        \right)
        \infnorm{ \theta_1 - \theta_2 }
      \right)
    \\ & 
    = 
      \infnorm{ \theta_1 - \theta_2 }
      \left[ 
        \sum_{ n = 0 }^{ k - 1 }
        \left(
          \left[
            \prod_{ m = n + 1 }^{ k - 1 }
            \big(
              l_m
              \, \infnorm{ \theta_{ 3 - j_m } }
            \big)
          \right]
          \mathfrak{m}_{ j_n, n }
          \left(
            l_n + 1 
          \right)
        \right)
      \right]
      .
    \end{split}
    \end{equation}
    Next observe that \cref{result:row_sum_norm} ensures that
    for all $ j \in \{ 1, 2 \} $, $ k \in \{ 1, 2, \dots, L \} \cap (1, \infty) $, $ x \in D $
    it holds that
    \begin{equation}
    \begin{split}
      \infnorm{ 
        ( \functionANN{\rect}( \phi_{ j, k } ) )( x ) 
      }
    &
      =
      \asinfnorm{
        W_{ j, k }
        \Big(
          \Rect_{ l_{ k - 1 } }\big(
            ( \functionANN{\rect}( \phi_{ j, k - 1 } ) )( x ) 
          \big)
        \Big)
        +
        B_{ j, k }
      }
    \\ &
      \leq 
      \left[
        \sup_{ 
          v \in \R^{ l_{ k - 1 } } \backslash \{ 0 \}
        }
        \frac{ 
          \infnorm{
            W_{ j, k }
            v
          }
        }{
          \infnorm{ v }
        }
      \right]
      \asinfnorm{
        \Rect_{ l_{ k - 1 } }\big(
          ( \functionANN{\rect}( \phi_{ j, k - 1 } ) )( x ) 
        \big)
      }
      +
      \infnorm{ B_{ j, k } }
    \\ & \leq
      l_{ k - 1 }
      \,
      \infnorm{ \theta_j }
      \asinfnorm{
        \Rect_{ l_{ k - 1 } }\big(
          ( \functionANN{\rect}( \phi_{ j, k - 1 } ) )( x ) 
        \big)
      }
      +
      \infnorm{ \theta_j }
    \\ & \leq
      l_{ k - 1 }
      \,
      \infnorm{ \theta_j }
      \asinfnorm{
        ( \functionANN{\rect}( \phi_{ j, k - 1 } ) )( x ) 
      }
      +
      \infnorm{ \theta_j }
    \\ & =
      \left(
        l_{ k - 1 }
        \asinfnorm{
          ( \functionANN{\rect}( \phi_{ j, k - 1 } ) )( x ) 
        }
        +
        1
      \right)
      \infnorm{ \theta_j }
    \\ & 
    \leq
      \left(
        l_{ k - 1 }
        \mathfrak{m}_{ j, k - 1 }
        +
        1
      \right)
      \infnorm{ \theta_j }
    \leq
      \mathfrak{m}_{ j, k - 1 }
      \left(
        l_{ k - 1 }
        +
        1
      \right)
      \infnorm{ \theta_j }
      .
    \end{split}
    \end{equation}
    Hence, we obtain for all 
    $ j \in \{ 1, 2 \} $,
    $ k \in \{ 1, 2, \dots, L \} \cap (1,\infty) $ that
    \begin{equation}
    \label{eq:mk_recursion}
      \mathfrak{m}_{ j, k }
    \leq
      \max\{ 1,
      \mathfrak{m}_{ j, k - 1 }
      \left(
        l_{ k - 1 }
        +
        1
      \right)
      \infnorm{ \theta_j }
      \}
      .
    \end{equation}
    Furthermore, note that 
    \cref{result:row_sum_norm} assures that
    for all $ j \in \{ 1, 2 \} $, $ x \in D $
    it holds that
    \begin{equation}
    \begin{split}
      \infnorm{ 
        ( \functionANN{\rect}( \phi_{ j, 1 } ) )( x ) 
      }
    &
      =
      \asinfnorm{
        W_{ j, 1 } x
        +
        B_{ j, 1 }
      }
    \\ &
      \leq 
      \left[
        \sup_{ 
          v \in \R^{ l_0 } \backslash \{ 0 \}
        }
        \frac{ 
          \infnorm{
            W_{ j, 1 }
            v
          }
        }{
          \infnorm{ v }
        }
      \right]
      \infnorm{
        x
      }
      +
      \infnorm{ B_{ j, 1 } }
    \\ & \leq
      l_0
      \,
      \infnorm{ \theta_j }
      \,
      \infnorm{
        x
      }
      +
      \infnorm{ \theta_j }
    \leq
      l_0
      \,
      \infnorm{ \theta_j }
      \max\{ | a | , | b | \}
      +
      \infnorm{ \theta_j }
    \\ &
    =
      \left(
        l_0
        \max\{ | a | , | b | \}
        +
        1
      \right)
      \infnorm{ \theta_j }
      \leq 
      \mathfrak{m}_{ 1, 0 }
      \left(
        l_0
        +
        1
      \right)
      \infnorm{ \theta_j }
      .
    \end{split}
    \end{equation}
    Therefore, we obtain that
    for all $ j \in \{ 1, 2 \} $ 
    it holds that
    \begin{equation}
      \mathfrak{m}_{ j, 1 } 
      \leq
      \max\{1,
      \mathfrak{m}_{ j, 0 }
      \left(
        l_0
        +
        1
      \right)
      \infnorm{ \theta_j }
      \}
      .
    \end{equation}
    Combining this with \eqref{eq:mk_recursion} demonstrates that
    for all $ j \in \{ 1, 2 \} $, 
    $ k \in \{ 1, 2, \dots, L \} $ it holds that
    \begin{equation}
      \mathfrak{m}_{ j, k }
    \leq
      \max\{1,
      \mathfrak{m}_{ j, k - 1 }
      \left(
        l_{ k - 1 }
        +
        1
      \right)
      \infnorm{ \theta_j }
      \}
      .
    \end{equation}
    Hence, we obtain that
    for all $ j \in \{ 1, 2 \} $, $ k \in \{ 0, 1, \dots, L \} $ it holds that
    \begin{equation}
      \mathfrak{m}_{ j, k }
    \leq
      \mathfrak{m}_{ j, 0 }
      \left[
        \prod_{ n = 0 }^{ k - 1 }
        \left(
          l_n
          +
          1
        \right)
      \right]
      \bigl[\max\{1, \infnorm{ \theta_j }\}\bigr]^k
      .
    \end{equation}
    Combining this with \eqref{eq:error_difference_estimate_DNN} proves that for all
    $
      j = ( j_n )_{ n \in \{ 0, 1, \dots, L \} } \colon \{ 0, 1, \dots, L \} \to \{ 1, 2 \} 
    $
    and all
    $ k \in \{ 1, 2, \dots, L \} $ 
    it holds that
    \begin{equation}
    \begin{split}
      \mathfrak{e}_k
    & 
    \leq
      \infnorm{ \theta_1 - \theta_2 }
      \left[ 
        \sum_{ n = 0 }^{ k - 1 }
        \left(
          \left[
            \prod_{ m = n + 1 }^{ k - 1 }
            \big(
              l_m
              \, \infnorm{ \theta_{ 3 - j_m } }
            \big)
          \right]
          \left(
            \mathfrak{m}_{ j_n, 0 }
            \left[
            \prod_{ v = 0 }^{ n - 1 }
            ( l_v + 1 )
            \right]
            \max\{1,\infnorm{ \theta_{ j_n } }^n\}
            \left(
              l_n + 1 
            \right)
          \right)
        \right)
      \right]
    \\ &
    =
      \mathfrak{m}_{ 1, 0 }
      \,
      \infnorm{ \theta_1 - \theta_2 }
      \left[ 
        \sum_{ n = 0 }^{ k - 1 }
        \left(
          \left[
            \prod_{ m = n + 1 }^{ k - 1 }
            \big(
              l_m
              \, \infnorm{ \theta_{ 3 - j_m } }
            \big)
          \right]
          \left(
            \left[
            \prod_{ v = 0 }^n
            ( l_v + 1 )
            \right]
            \max\{1,\infnorm{ \theta_{ j_n } }^n\}
          \right)
        \right)
      \right]
    \\ &
    \leq
      \mathfrak{m}_{ 1, 0 }
      \,
      \infnorm{ \theta_1 - \theta_2 }
      \left[ 
        \sum_{ n = 0 }^{ k - 1 }
        \left(
          \left[
            \prod_{ m = n + 1 }^{ k - 1 }
            \infnorm{ \theta_{ 3 - j_m } }
          \right]
          \left[
            \prod_{ v = 0 }^{ k - 1 }
            ( l_v + 1 )
          \right]
          \max\{1,\infnorm{ \theta_{ j_n } }^n\}
        \right)
      \right]
    \\ &
    =
      \mathfrak{m}_{ 1, 0 }
      \,
      \infnorm{ \theta_1 - \theta_2 }
      \left[
        \prod_{ n = 0 }^{ k - 1 }
        ( l_n + 1 )
      \right]
      \left[ 
        \sum_{ n = 0 }^{ k - 1 }
        \left(
          \left[
            \prod_{ m = n + 1 }^{ k - 1 }
            \infnorm{ \theta_{ 3 - j_m } }
          \right]
          \max\{1,\infnorm{ \theta_{ j_n } }^n\}
        \right)
      \right]
      .
    \end{split}
    \end{equation}
    Therefore, we obtain that for all 
    $
      j \in \{ 1, 2 \}
    $, 
    $ k \in \{ 1, 2, \dots, L \} $ 
    it holds that
    \begin{equation}
    \begin{split}
      \mathfrak{e}_k
    & 
    \leq
      \mathfrak{m}_{ 1, 0 }
      \,
      \infnorm{ \theta_1 - \theta_2 }
      \left[
        \prod_{ n = 0 }^{ k - 1 }
        ( l_n + 1 )
      \right]
      \left[ 
        \sum_{ n = 0 }^{ k - 1 }
        \left(
          \left[
            \prod_{ m = n + 1 }^{ k - 1 }
            \infnorm{ \theta_{ 3 - j } }
          \right]
          \max\{1,\infnorm{ \theta_{ j } }^n\}
        \right)
      \right]
    \\
    & 
    =
      \mathfrak{m}_{ 1, 0 }
      \,
      \infnorm{ \theta_1 - \theta_2 }
      \left[
        \prod_{ n = 0 }^{ k - 1 }
        ( l_n + 1 )
      \right]
      \left[ 
        \sum_{ n = 0 }^{ k - 1 }
        \left(
        \max\{1,\infnorm{ \theta_{ j } }^n\}
          \,
          \infnorm{ \theta_{ 3 - j } }^{ k - 1 - n }
        \right)
      \right]
    \\
    & 
    \leq
      k \,
      \mathfrak{m}_{ 1, 0 }
      \,
      \infnorm{ \theta_1 - \theta_2 }
      \left( 
        \max\{ 1, \infnorm{ \theta_1 }, \infnorm{ \theta_2 } \}
      \right)^{ k - 1 }
          \left[
            \prod_{ m = 0 }^{ k - 1 }
            \big(
              l_m + 1 
            \big)
          \right]
      .
    \end{split}
    \end{equation}
    The proof of \cref{thm:RealNNLipsch} is thus completed.
    \end{proof}

\begin{cor}
  \label{cor:ClippedRealNNLipsch}
  Let 
    $a\in\R$,
    $b\in[a,\infty)$,
	$u\in[-\infty,\infty)$,
	$v\in(u,\infty]$,
    $d,L\in\N$,
    $l=(l_0,l_1,\dots,l_L)\in\N^{L+1}$
  satisfy
  \begin{equation}
    d\geq \sum_{k=1}^Ll_k(l_{k-1}+1).
  \end{equation}
  Then it holds for all
    $\theta,\vartheta\in\R^d$
  that
  \begin{equation}
    \sup_{x\in[a,b]^{l_0}}\infnorm{\ClippedRealV{\theta}{l}uv(x)-\ClippedRealV\vartheta{l}uv(x)}
    \leq
    L 
    \max\{ 1, \abs{a}, \abs{b} \}
    \, 
    ( \infnorm{ l } + 1 )^L 
    \,
    ( \max\{ 1, \infnorm\theta,\infnorm\vartheta\} )^{ L - 1 }
    \,
    \infnorm{ \theta - \vartheta }
  \end{equation}
  (cf. \cref{def:rectclippedFFANN,def:infnorm}).
\end{cor}
\begin{proof}[Proof of Corollary~\ref{cor:ClippedRealNNLipsch}]
  Observe that
    Theorem~\ref{thm:RealNNLipsch}
    and \cref{lem:clipcontr}
  demonstrate that for all
    $\theta,\vartheta\in\R^d$
  it holds that
  \begin{equation}
  \begin{split}
    \sup_{x\in[a,b]^{l_0}}\infnorm{\ClippedRealV{\theta}{l} uv(x)-\ClippedRealV\vartheta{l} uv(x)}
    &=
    \sup_{x\in[a,b]^{l_0}}\asinfnorm{\Clip{u}{v}{l_L}\bigl(\smash{(\UnclippedRealV\theta{l})(x)}\bigr)-\Clip uv{l_L}\bigl(\smash{(\UnclippedRealV\vartheta{l})(x)}\bigr)}
    \\&\leq
    \sup_{x\in[a,b]^{l_0}}\infnorm{(\UnclippedRealV\theta{l})(x)-(\UnclippedRealV\vartheta{l})(x)}
    \\&\leq
    L 
    \max\{ 1, \abs{a}, \abs{b} \}
    \, 
    ( \infnorm{ l } + 1 )^L 
    \,
    ( \max\{ 1, \infnorm\theta,\infnorm\vartheta\} )^{ L - 1 }
    \,
    \infnorm{ \theta - \vartheta }
  \end{split}
  \end{equation}
  (cf. \cref{def:rectclippedFFANN,def:infnorm,def:clip}).
  This completes the proof of Corollary~\ref{cor:ClippedRealNNLipsch}.
\end{proof}

\section{Separate analyses of the error sources}
\label{sec:error_sources}

In this section we study separately the approximation error 
(see \cref{subsec:approximation_error} below), 
the generalization error (see \cref{subsec:generalization_error} below), 
and the optimization error (see \cref{subsec:optimization_error} below). 

In particular, the main result in \cref{subsec:approximation_error}, 
\cref{prop:neural_net_approximation} below, establishes an upper bound 
for the error in the approximation of a Lipschitz continuous function by DNNs. 
This approximation result is obtained by combining the essentially 
well-known approximation result in \cref{lem:lipschitz_extension} 
with the DNN calculus in \cref{subsec:structured_description} above 
(cf., e.g., Grohs et al.~\cite{Grohs2019ANNCalculus,GrohsJentzenSalimova2019}). 
Some of the results in Section~\ref{subsec:approximation_error} 
are partially based 
on material in publications from the scientific literature. 
In particular, 
the elementary result in \cref{lem:max_net} is basically 
well-known in the scientific literature. 
For further approximation results for DNNs we refer, e.g., to  \cite{bach2017breaking,
	barron1993universal,
	barron1994approximation,
	blum1991approximation,
	BolcskeiGrohsKutyniokPetersen2019OptimalApproximation,
	BurgerNeubauer2001,
	candes1998ridgelets,
	chen1995approximation,
	ChuiMhaskar1994,
	Cybenko1989,
	DeVore1997,
	wang2018exponential,
	ElbraechterSchwab2018,
	eldan2016power, 
	Ellacott1994,
	funahashi1989approximate,
	gribonval2019approximation,
	GrohsWurstemberger2018,
	Grohs2019ANNCalculus,
	grohs2019deep,
	guhring2019error,
	hartman1990layered,
	hornik1991approximation,
	hornik1993some,
	hornik1989multilayer,
	hornik1990universal,
	HutzenthalerJentzenKruse2019,
	JentzenSalimovaWelti2018,
	KutyniokPetersen2019,
	leshno1993multilayer,
	mhaskar1996neural,
	MhaskarMicchelli1995,
	MhaskarPoggio2016,
	nguyen1999approximation,
	park1991universal,
	perekrestenko2018universal,
	petersen2018topological,
	petersen2018equivalence,
	petersen2018optimal,
	Pinkus1999,
	ReisingerZhang2019,
	schmitt2000lower,
	SchwabZech2019,	
	ShahamCloningerCoifman2018,
	shen2019deep,
	shen2019nonlinear,
	voigtlaender2019approximation,
	yarotsky2017error,
	yarotsky2018universal} 
 and the references mentioned therein. 

In \cref{lem:cov5,lem:cov6} in \cref{subsec:generalization_error} below 
we study the generalization error. 
Our analysis in \cref{subsec:generalization_error} is in parts inspired 
by Berner et al.~\cite{BernerGrohsJentzen2018} 
and Cucker \& Smale~\cite{CuckerSmale2002FoundationsLearning}. 
\Cref{prop:Hoeffding2} in \cref{subsec:hoeffding} is known as Hoeffding's inequality 
in the scientific literature and \cref{prop:Hoeffding2} is, e.g., 
proved as Theorem~2 in Hoeffding~\cite{Hoeffding}.
The proof of \cref{prop:covnum} can be found, e.g., in
Cucker \& Smale~\cite[Proposition~5]{CuckerSmale2002FoundationsLearning}
(cf.\ also Berner et al.~\cite[Proposition~4.3]{BernerGrohsJentzen2018}).
For further results on the generalization error we refer, e.g., to \cite{Bartlett2005,
	GyoerfiKohlerKrzyzakWalk2002,
	Massart2007,
	shalev2014understanding,
	VanDeGeer2000} and the references mentioned therein. 

In the two elementary results in \cref{subsec:optimization_error}, 
\cref{lem:estimate_opt_error0,lem:estimate_opt_error1}, we study the 
optimization error of the minimum Monte Carlo algorithm. 
A related result can be found, e.g., in \cite[Lemma 3.5]{kolmogorov2018solving}. For further results on the optimization error we refer, e.g., to \cite{bach2013non,bercu2011generic,chau2019stochastic,DereichMuellerGronbach2019,FehrmanGessJentzen2019convergence,JentzenKuckuckNeufeldWurstemberger2018,JentzenWurstemberger2018LowerBound,Karimietal2019,Leietal2019} 
and the references mentioned therein. 


\subsection{Analysis of the approximation error}
\label{subsec:approximation_error}

\subsubsection{Approximations for Lipschitz continuous functions}

\begin{lemma} \label{lem:lipschitz_extension}
		Let 
		$(E,\delta)$ 
    be a metric space, 
    let
		$\cM \subseteq E$
		satisfy
		$\cM\neq\emptyset$, 
		let 
		$L\in [0,\infty)$, 
		let 
		$f\colon E \to \R$ 
		satisfy for all 
		$x\in E$,
		$y\in\cM$
		that 
		$|f(x)-f(y)|\leq L\delta(x,y)$, 
		and let 
		$F\colon E \to \R\cup\{\infty\}$ 
		satisfy for all 
		$x\in E$ 
		that 	
		\begin{equation} \label{lipschitz_extension:definition}
		F(x) 
		= 
		\sup_{y\in\cM} \left[
		f(y) - L \delta(x,y)
		\right]\!.  
		\end{equation} 	
		Then 
		\begin{enumerate}[label=(\roman{*})]
			\item \label{lipschitz_extension:item1}
			it holds for all 
			$x\in E$ 
			that 
			$F(x) \leq f(x)$, 
			\item \label{lipschitz_extension:item2}
			it holds for all 
			$x\in\cM$ 
			that 
			$F(x) = f(x)$, 
			\item \label{lipschitz_extension:item3}
			it holds for all 
			$x,y\in E$ 
			that 
			$|F(x)-F(y)|\leq L \delta(x,y)$,    
			and 
			\item  \label{lipschitz_extension:item4}	
			it holds for all 
			$x\in E$ 
			that 
			\begin{equation} 
			|F(x)-f(x)| 
			\leq 
			2L\! \left[ \inf_{y\in\cM} \delta(x,y) \right]\!. 
			\end{equation} 	
		\end{enumerate}
	\end{lemma} 
  \begin{proof}[Proof of \cref{lem:lipschitz_extension}]
		First, observe that the hypothesis that for  all 
		$x\in E$,
		$y\in\cM$
		it holds that 
		$|f(x)-f(y)|\leq L \delta(x,y)$ 
		ensures that for all 
		$x\in E$, 
		$y\in\cM$ 
		it holds that 
		\begin{equation} 
		f(x) \geq f(y) - L \delta(x,y). 	\end{equation} 
		Hence, we obtain that for all 
		$x\in E$ 
		it holds that 
		\begin{equation} 
		f(x) \geq \sup_{y\in\cM} \left[ f(y) - L \delta(x,y) \right] 
		= 
		F(x). 
		\end{equation} 
		This establishes item~\ref{lipschitz_extension:item1}. 
		Next observe that \eqref{lipschitz_extension:definition} implies that for all 
		$x\in\cM$ 
		it holds that 
		\begin{equation} 
		F(x) \geq f(x) - L \delta(x,x) = f(x). 
		\end{equation} 
		Combining this with item~\ref{lipschitz_extension:item1} establishes item~\ref{lipschitz_extension:item2}. 
    In the next step we note that 
      \eqref{lipschitz_extension:definition}
      and the fact that for all $x\in E$ it holds that
        $F(x)\leq f(x)<\infty$
    show that
    for all 
		$x,y\in E$
		it holds that 
		\begin{equation} 
		\begin{split}
		F(x)-F(y) 
		& = 
		\biggl[
		\sup_{v\in\cM} ( 
		f(v) - L\delta(x,v)
		)
		\biggr]
		- 
		\biggl[
		\sup_{w\in\cM} 
		(
		f(w) - L\delta(y,w)
		)
		\biggr]
		\\
		& = 
		\sup_{v\in\cM} 
		\left[ 
		f(v) - L \delta(x,v) 
		- 
		\sup_{w\in\cM} 
		( f(w) - L \delta(y,w) )
		\right] 
		\\
		& \leq 
		\sup_{v\in\cM} 
		\bigl[
		f(v) - L \delta(x,v) 
		- ( f(v) - L\delta(y,v) )
		\bigr] 
		\\ & = 
		L \biggl[
		\sup_{v\in\cM} 
		(\delta(y,v) - \delta(x,v) )
		\biggr]
    \\&\leq 
    L \biggl[
    \sup_{v\in\cM} 
    (\delta(y,x) +\delta(x,v) - \delta(x,v) )
    \biggr]
    =
    L \delta(x,y)
    . 
		\end{split}
		\end{equation} 
		Combining this with the fact that for all 
		$x,y\in E$ 
		it holds that 
		$\delta(x,y)=\delta(y,x)$
		establishes item~\ref{lipschitz_extension:item3}. 
		Observe that
		item~\ref{lipschitz_extension:item2}, 
		the triangle inequality, 
		item~\ref{lipschitz_extension:item3}, 
		and the hypothesis that for all 
		$x\in E$,
		$y\in\cM$
		it holds that 
		$|f(x)-f(y)|\leq L\delta(x,y)$ 
		ensure that for all 
		$x\in E$ 
		it holds that 
		\begin{equation} 
		\begin{split}
		|F(x)-f(x)| 
		& = 
		\inf_{y\in\cM} 
    |F(x)-F(y)+f(y)-f(x)| 
    \\ & \leq 
		\inf_{y\in\cM} 
		\left( |F(x)-F(y)| + |f(y)-f(x)| \right) 
		\\
		& \leq 
		\inf_{y\in\cM} 
		(2L \delta(x,y)) 
		= 
    2L\biggl[\inf_{y\in\cM} \delta(x,y)\biggr]
    .
		\end{split}
		\end{equation} 
		This establishes item~\ref{lipschitz_extension:item4}. The proof of \cref{lem:lipschitz_extension} is thus completed. 
	\end{proof} 

	\subsubsection{DNN representations for maxima}

\begin{lemma} \label{lem:max_net}
    Let $\Phi\in\ANNs$ satisfy
    \begin{equation}
      \label{eq:defPhi}
      \Phi
      =
      \left( \rule{0cm}{1cm}
        \left(
          \begin{pmatrix} 
            1 & -1 \\
            0 & 1 \\
            0 & -1 
          \end{pmatrix}, 
          \begin{pmatrix} 
            0 \\ 0 \\ 0 	
          \end{pmatrix} 
          \right), 
          \left(
          \begin{pmatrix} 
            1 & 1 & -1
          \end{pmatrix}, 
          0  
        \right) 
      \right)
      \in
      \bigl((\R^{3\times 2}\times\R^{3})\times(\R^{1\times 3}\times\R)\bigr)
    \end{equation}
    (cf.\ \cref{def:ANN}).
    Then
    \begin{enumerate}[label=(\roman{*})]
      \item \label{max_net:item00} 
        it holds for all $k\in\N$ that $\lengthANN(\mf I_k)=2$,
      \item \label{max_net:item0} 
        there exist unique $\phi_k\in\ANNs$, $k\in\{2,3,\ldots\}$, 
        which satisfy for all
          $k\in\{2,3,\ldots\}$
        that
          $\phi_2 = \Phi$, 
          $\inDimANN(\phi_k)=\outDimANN(\bP_2(\Phi,\mf I_{k-1}))$, and
          \begin{equation}
            \label{eq:phikdef}
            \phi_{k+1}=\phi_k\bullet\bigl(\bP_2(\Phi,\mf I_{k-1})\bigr),
          \end{equation}
      \item \label{max_net:item1} 
        it holds for all 
          $k\in\{2,3,\ldots\}$  
        that 
          $\lengthANN(\phi_k) = k$,
        and  
      \item \label{max_net:item2} 
        it holds for all 
          $k\in\{2,3,\ldots\}$ 
        that 
          $\dims(\phi_k) = (k, 2k-1, 2k-3, \ldots, 3, 1) \in\N^{k+1}$,
        and
      \item \label{max_net:item3}
        it holds for all
          $k\in\{2,3,\ldots\}$,
          $x=(x_1,x_2,\dots,x_k)\in\R^k$
        that
        \begin{equation}
          \bigl(\functionANN{\rect}(\phi_k)\bigr)(x)
          =
          \max\{x_1,x_2,\dots,x_k\}
        \end{equation}
    \end{enumerate}
    (cf.\ \cref{def:ReLu:identity,def:simpleParallelization,def:ANNcomposition,def:relu1,def:ANNrealization}).
	\end{lemma} 
\begin{proof}[Proof of \cref{lem:max_net}] 
  First, note that, \enum{
    e.g., item (i) in \cite[Lemma 3.16]{GrohsJentzenSalimova2019}
  }[show] that for all
    $k\in\N$
  it holds that
  \begin{equation}
    \label{eq:dimsIk}
    \dims(\mf I_k)
    =
    (k,2k,k)
    .
  \end{equation}
  This establishes item~\ref{max_net:item00}.
  Next note that \enum{
    \eqref{eq:defPhi}
  }[demonstrate] that
  \begin{equation}
    \dims(\Phi)
    =
    (2,3,1)
    .
  \end{equation}
  Combining
    this
    and \eqref{eq:dimsIk}
  with
    item (i) in \cite[Proposition~2.20]{Grohs2019ANNCalculus}
  shows that for all
    $k\in\N$
  it holds that
  \begin{equation}
    \label{eq:dimsP2}
    \dims(\parallelizationSpecial_2(\Phi,\mf I_k))
    =
    (k+2,2k+3,k+1)
  \end{equation}
  (cf.\ \cref{def:simpleParallelization}).
  Hence, we obtain that for all
    $k\in\{2,3,\ldots\}$
  it holds that
  \begin{equation}
    \label{eq:dimsP2.2}
    \dims(\parallelizationSpecial_2(\Phi,\mf I_{k-1}))
    =
    (k+1,2k+1,k)
    .
  \end{equation}
  Combining
    this
  with
    \eqref{eq:dimsP2}
  ensures that for all
    $k\in\{2,3,\ldots\}$
  it holds that
  \begin{equation}
    \label{eq:dimsP2.3}
    \outDimANN(\parallelizationSpecial_2(\Phi,\mf I_k))
    =
    k+1
    =
    \inDimANN(\parallelizationSpecial_2(\Phi,\mf I_{k-1}))
    .
  \end{equation}
  %
  Moreover, note that
    \eqref{eq:defPhi}
    and \eqref{eq:dimsP2}
  assure that
  \begin{equation}
    \label{eq:phibase}
    \inDimANN(\Phi)
    =
    2
    =
    \outDimANN(\parallelizationSpecial_2(\Phi,\mf I_1))
    .
  \end{equation}
  Furthermore, observe that \enum{
    item (i) in \cite[Proposition~2.6]{Grohs2019ANNCalculus} ;
    \eqref{eq:dimsP2.3}
  }[show] that for all
    $k\in\{2,3,\ldots\}$,
    $\psi\in\ANNs$
    with 
      $\inDimANN(\psi)=\outDimANN(\bP_2(\Phi,\mf I_{k-1}))$
  it holds that
  \begin{equation}
    \inDimANN\bigl(\psi\bullet \bigl(\bP_2(\Phi,\mf I_{k-1})\bigr)\bigr)
    =
    \inDimANN(\parallelizationSpecial_2(\Phi,\mf I_{k-1}))
    =
    \outDimANN(\parallelizationSpecial_2(\Phi,\mf I_k))
  \end{equation}
  (cf.\ \cref{def:ANNcomposition}).
  Combining
    this
  and 
    \eqref{eq:phibase}
  with
    induction
  establishes item~\ref{max_net:item0}.
In the next step we note that
  \eqref{eq:phikdef}
  and item (ii) in \cite[Proposition~2.6]{Grohs2019ANNCalculus}
imply that for all
		$ k \in \{2,3,\ldots\} $ 
it holds that 
		\begin{equation} 
		\lengthANN(\phi_{k+1}) 
		= 
		\lengthANN(\phi_{k}) 
		+ 
		\lengthANN(\parallelizationSpecial_{2}(\phi_2,\mathfrak{I}_{k-1})) 
		- 1
		=
		\lengthANN(\phi_{k}) 
		+ 
		1.  
		\end{equation} 
Combining 
  this 
  and the fact that $\lengthANN(\phi_2)=2$ 
with
  induction 
establishes item~\ref{max_net:item1}.
Furthermore, observe that \enum{
  \eqref{eq:phikdef} ;
  \eqref{eq:dimsP2.2} ;
  item (i) in~\cite[Proposition~2.6]{Grohs2019ANNCalculus}
}[demonstrate] that for all 
    $k\in\{2,3,\ldots\}$,
    $l_0,l_1,\dots,l_k\in\N$
with 
	$\dims(\phi_k)=(l_0,l_1,\dots,l_k)$
it holds that  
\begin{equation} 
	\dims(\phi_{k+1}) 
  = 
  \dims\bigl(\phi_k\bullet\bigl(\bP_2(\Phi,\mf I_{k-1})\bigr)\bigr)
	= 
	(k+1, 2k+1, l_1, l_2,\dots,l_k). 
\end{equation} 
    This, 
    item (iii),
    the fact that $\dims(\phi_2) = (2,3,1)$, 
    and induction
  establish 
    item~\ref{max_net:item2}.
  Moreover, note that \enum{
    \eqref{eq:defPhi}
  }[ensure] that for all
    $(x_1,x_2)\in\R^2$
  it holds that
  \begin{equation}
    \label{eq:realPhi}
  \begin{split}
    \bigl(\functionANN{\rect}(\Phi)\bigr)(x_1,x_2)
    &=
    \pmat{1&1&-1}\left(\multdim_{\rect,3}\!\left(\pmat{1&-1\\0&1\\0&-1}\pmat{x_1\\x_2}+\pmat{0\\0\\0}\right)\right)+0
    \\&=
    \pmat{1&1&-1}\pmat{\max\{x_1-x_2,0\}\\\max\{x_2,0\}\\\max\{-x_2,0\}}
    \\&=
    \max\{x_1-x_2,0\}+\max\{x_2,0\}-\max\{-x_2,0\}
    \\&=
    \max\{x_1-x_2,0\}+x_2
    =
    \max\{x_1,x_2\}
  \end{split}
  \end{equation}
  (cf.\ \cref{def:relu1,def:ANNrealization,def:multidim_version}).
  Combining \enum{
    this ;
    item (iii) in \cite[Lemma 3.16]{GrohsJentzenSalimova2019}
  } with
    \cite[Proposition~2.19]{Grohs2019ANNCalculus}
  proves that for all
    $k\in\{2,3,\ldots\}$,
    $x=(x_1,x_2,\dots,\allowbreak x_{k+1})\in\R^{k+1}$
  it holds that
  \begin{equation}
  \begin{split}
    \bigl(\functionANN{\rect}(\parallelizationSpecial_2(\Phi,\mf I_{k-1}))\bigr)(x)
    &=
    \bigl(\bigl(\functionANN{\rect}(\Phi)\bigr)(x_1,x_2),\bigl(\functionANN{\rect}(\mf I_{k-1})\bigr)(x_3,x_4,\dots,x_{k+1})\bigr)
    \\&=
    (\max\{x_1,x_2\},x_3,x_4,\dots,x_{k+1})
    \in\R^k
    .
  \end{split}
  \end{equation}
  \enum{
    Item (v) in \cite[Proposition~2.6]{Grohs2019ANNCalculus} ;
    \eqref{eq:phikdef}
  } hence
  show that for all
    $k\in\{2,3,\ldots\}$,
    $x=(x_1,x_2,\dots,x_{k+1})\in\R^{k+1}$
  it holds that
  \begin{equation}
  \begin{split}
    \bigl(\functionANN{\rect}(\phi_{k+1})\bigr)(x)
    &=
    \Bigl(\functionANN{\rect}\bigl(\phi_k\bullet\bigl(\parallelizationSpecial_2(\Phi,\mf I_{k-1})\bigr)\bigr)\Bigr)(x)
    =
    \bigl(\bigl[\functionANN{\rect}(\phi_k)\bigr]\circ\bigl[\functionANN{\rect}(\parallelizationSpecial_2(\Phi,\mf I_{k-1}))\bigr]\bigr)(x)
    \\&=
    \bigl(\functionANN{\rect}(\phi_k)\bigr)(\max\{x_1,x_2\},x_3,x_4,\dots,x_{k+1})
    .
  \end{split}
  \end{equation}
    This,
    the fact that $\phi_2=\Phi$,
    \eqref{eq:realPhi},
    and induction
  establish
    item~\ref{max_net:item3}.
  The proof of \cref{lem:max_net} is thus completed. 
\end{proof} 
	
	\begin{lemma} \label{lem:max_net2}
    Let
    $ A_k \in \R^{ (2k-1) \times k } $, 
    $ k \in \{2,3,\ldots\}$,  
		and 
		$ C_k \in \R^{ (k-1) \times (2k-1) } $, $ k \in \{2,3,\ldots\} $, 
		satisfy for all 
		$ k \in \{2,3,\ldots\} $ 
		that 
		\begin{equation} \label{max_net2:definition_of_Ak_and_Ck}
		A_k = 
		\begin{pmatrix} 
		1 		& -1 	 &  0 	  & \cdots & 	0   \\
		0 		&  1 	 &  0 	  & \cdots & 	0   \\
		0 		& -1 	 &  0 	  & \cdots & 	0   \\
		0 		&  0 	 &  1 	  & \cdots & 	0   \\
		0 		&  0 	 & -1 	  & \cdots & 	0   \\
		\vdots & \vdots & \vdots & \ddots & \vdots \\ 
		0 	 	&  0 	 &  0 	  & \cdots & 	1   \\
		0 		&  0 	 &  0 	  & \cdots &   -1
		\end{pmatrix} 
		\qquad\text{and}\qquad 
		C_k = 
		\begin{pmatrix} 
		1 & 1 & -1	& 0 &  0 & \cdots & 0 & 0 \\
		0 & 0 & 0  & 1 & -1 & \cdots & 0 & 0 \\
		\vdots & \vdots & \vdots & \vdots & \vdots & \ddots & \vdots & \vdots \\
		0 & 0 & 0  & 0 &  0 & \cdots & 1 & -1
		\end{pmatrix}
		\end{equation} 
		and let 
		$\phi_k=( (W_{k,1}, B_{k,1}), (W_{k,2}, B_{k,2}), \ldots, (W_{k,k}, B_{k,k}) )\in\bN$, $k\in\{2,3,\ldots\}$, 
		satisfy for all 
		$k\in\{2,3,\ldots\}$ 
    that
    $\inDimANN(\phi_k)=\outDimANN(\bP_2(\phi_2,\mf I_{k-1}))$,
		$\phi_{k+1} = \phi_{k} \bullet (\bP_{2}(\phi_2, \mathfrak{I}_{k-1}))$,
		and 
		\begin{equation} \label{max_net2:definition_max_net}
		\phi_2 
		= 
		\left( \rule{0cm}{1cm}
		\left(
		\begin{pmatrix} 
		1 & -1 \\
		0 & 1 \\
		0 & -1 
		\end{pmatrix}, 
		\begin{pmatrix} 
		0 \\ 0 \\ 0 	
		\end{pmatrix} 
		\right), 
		\left(
		\begin{pmatrix} 
		1 & 1 & -1
		\end{pmatrix}, 
		0  
		\right) 
    \right)
    \in
    \bigl((\R^{3\times 2}\times\R^{3})\times(\R^{1\times 3}\times\R)\bigr)
		\end{equation} 
		(cf.\ \cref{def:ANN,def:ReLu:identity,def:ANNcomposition,def:simpleParallelization} and \cref{lem:max_net}). 
		Then 
		\begin{enumerate}[label=(\roman{*})]
			\item \label{max_net2:item1} it holds for all 
			$k\in\{2,3,\ldots\}$ 
			that 
			$ W_{k,1} = A_{k} $, 
			\item \label{max_net2:item2} 
			it holds for all 
			$k\in\{2,3,\ldots\}$, 
			$l\in\{1,2,\ldots,k\}$ 
			that 
			$B_{k,l} = 0 \in \R^{2(k-l)+1}$,  
			\item \label{max_net2:item3} 
			it holds for all 
			$k\in\{2,3,\ldots\}$, 
			$l\in\{3,4,\ldots,k+1\}$ 
			that 
			$(W_{k+1,l}, B_{k+1,l}) = (W_{k,l-1}, B_{k,l-1})$, 	
			\item \label{max_net2:item4} 
			it holds for all 
			$k\in\{2,3,\ldots\}$ 
			that
			$ W_{k+1,2} = W_{k,1} C_{k+1} $,    
      \item \label{max_net2:itemn5}
      it holds for all $k\in\{2,3,\ldots\}$
      that
      $(0,0,\dots,0)\neq\MappingStructuralToVectorized(\phi_k)\in\bigl(\{-1,0,1\}^{\paramANN(\phi_k)}\bigr)$,
      and
			\item \label{max_net2:item5} 
			it holds for all 
			$k\in\{2,3,\ldots\}$ 
			that 
      $ \infnorm{\MappingStructuralToVectorized(\phi_k)} = 1$
    \end{enumerate} 
    (cf.\ \cref{def:TranslateStructuredIntoVectorizedDescription,def:infnorm}).
	\end{lemma} 
	\begin{proof}[Proof of \cref{lem:max_net2}] 
		First, note that \enum{
      \eqref{eq:def:id:1} ;
      \eqref{eq:def:id:2} ;
      \eqref{max_net2:definition_of_Ak_and_Ck} ;
      \eqref{max_net2:definition_max_net}
    }[ensure] that for all 
		$k\in\{2,3,\ldots\}$ 
		it holds that 
		\begin{equation}
		\parallelizationSpecial_2(\phi_2,\mathfrak{I}_{k-1}) 
		= 
    ( \MatrixANN{A_{k+1}}, \MatrixANN{C_{k+1}} )
    \end{equation}
    (cf.\ \cref{def:matrixInputDNN}).
    This
    and \eqref{max_net2:definition_max_net}
		imply that for all 
		$ k \in \{2,3,\ldots\} $ 
		it holds that 
		\begin{equation} \label{max_net2:structure_recursion}
		\begin{split}
		\phi_{ k + 1 } 
		& = 
		\phi_{ k } \bullet \bigl( \parallelizationSpecial_2(\phi_2, \mathfrak{I}_{k-1}) \bigr)
		\\
		& = ( (W_{k,1}, B_{k,1}), (W_{k,2}, B_{k,2}), \ldots, (W_{k,k}, B_{k,k}) ) \bullet (\MatrixANN{A_{k+1}}, \MatrixANN{C_{k+1}})
		\\
		& = 
		( \MatrixANN{A_{k+1}}, (W_{k,1}C_{k+1}, B_{k,1}), (W_{k,2}, B_{k,2}), \ldots, (W_{k,k}, B_{k,k}) ). 
		\end{split}
		\end{equation} 
    %
      This, 
      \eqref{max_net2:definition_of_Ak_and_Ck}, 
      and \eqref{max_net2:definition_max_net} 
    establish 
      item~\ref{max_net2:item1}. 
    Next observe that \enum{
      \eqref{max_net2:definition_max_net} ;
      \eqref{max_net2:structure_recursion} ;
      \cref{max_net:item2} in \cref{lem:max_net} ;
      induction
    }[prove]
      item~\ref{max_net2:item2}. 
    Moreover, note that 
      \eqref{max_net2:structure_recursion} 
    establishes
      \cref{max_net2:item3,max_net2:item4}. 
		In addition, observe that item~\ref{max_net2:item1} proves that for all 
		$k\in\{2,3,\ldots\}$ 
		it holds that 
		\begin{equation} 
		\begin{split}
    W_{k,1}C_{k+1} 
    = A_kC_{k+1}
		& = 
		\begin{pmatrix} 
		1 		& -1 	 &  0 	  & \cdots & 	0   \\
		0 		&  1 	 &  0 	  & \cdots & 	0   \\
		0 		& -1 	 &  0 	  & \cdots & 	0   \\
		0 		&  0 	 &  1 	  & \cdots & 	0   \\
		0 		&  0 	 & -1 	  & \cdots & 	0   \\
		\vdots & \vdots & \vdots & \ddots & \vdots \\ 
		0 	 	&  0 	 &  0 	  & \cdots & 	1   \\
		0 		&  0 	 &  0 	  & \cdots &   -1
		\end{pmatrix} 
		\begin{pmatrix} 
		1 & 1 & -1	& 0 &  0 & \cdots & 0 & 0 \\
		0 & 0 & 0  & 1 & -1 & \cdots & 0 & 0 \\
		\vdots & \vdots & \vdots & \vdots & \vdots & \ddots & \vdots & \vdots \\
		0 & 0 & 0  & 0 &  0 & \cdots & 1 & -1
		\end{pmatrix}
		\\
		& = 
		\begin{pmatrix} 
		1 & 1 & -1 & -1 & 1 & 0 & 0 & \cdots & 0 & 0 \\
		0 & 0 &  0 &  1 &-1 & 0 & 0 & \cdots & 0 & 0 \\
		0 & 0 &  0 & -1 & 1 & 0 & 0 & \cdots & 0 & 0 \\ 
		0 & 0 &  0 &  0 & 0 & 1 &-1 & \cdots & 0 & 0 \\
		0 & 0 &  0 &  0 & 0 &-1 & 1 & \cdots & 0 & 0 \\ 
		\vdots & 	\vdots & 	\vdots & 	\vdots & 	\vdots & 	\vdots & 	\vdots & \ddots & 	\vdots &  	\vdots \\
		0 & 0 &  0 &  0 & 0 & 0 & 0 & \cdots & 1 & -1 \\
		0 & 0 &  0 &  0 & 0 & 0 & 0 & \cdots & -1 & 1
		\end{pmatrix}.  
		\end{split}
    \end{equation} 
  Combining \enum{
    this ;
    \eqref{max_net2:definition_max_net} ;
    \eqref{max_net2:structure_recursion} ;
  } with 
    induction
  proves \cref{max_net2:itemn5}.
  Next note that
    \cref{max_net2:itemn5}
  establishes
    \cref{max_net2:item5}.
  The proof of \cref{lem:max_net2} is thus completed. 
\end{proof}

	\subsubsection{Interpolations through DNNs}
	
\newcommand{\bW}{\cW}%
\newcommand{\bB}{\cB}%
\newcommand{\sW}[1]{W}%
\newcommand{\sB}[1]{B_{#1}}%
\begin{lemma} \label{lem:neural_net_representation}
		Let 
		$\phi_k\in\bN$, $k\in\{2,3,\ldots\}$, 
		satisfy for all 
		$k\in\{2,3,\ldots\}$ 
    that 
    $\inDimANN(\phi_k)=\outDimANN(\bP_2(\phi_2,\mf I_{k-1}))$,
		$\phi_{k+1} = \phi_{k} \bullet (\bP_{2}(\phi_2, \mathfrak{I}_{k-1}))$,
		and 
		\begin{equation} \label{neural_net_representation:max_net}
		\phi_2 
		= 
		\left( \rule{0cm}{1cm}
		\left(
		\begin{pmatrix} 
		1 & -1 \\
		0 & 1 \\
		0 & -1 
		\end{pmatrix}, 
		\begin{pmatrix} 
		0 \\ 0 \\ 0 	
		\end{pmatrix} 
		\right), 
		\left(
		\begin{pmatrix} 
		1 & 1 & -1
		\end{pmatrix}, 
		0  
		\right) 
    \right)
    \in 
    \bigl((\R^{3\times 2}\times\R^{3})\times(\R^{1\times 3}\times\R)\bigr)
    ,
    \end{equation} 
    let
		$d\in\N$, 
		$L\in [0,\infty)$, 
		let 
		$ \cM \subseteq \R^d $
		satisfy 
		$ \card\cM \in \{2,3,\ldots\} $, 
		let 
		$ m \colon \{1,2,\ldots,\card{\cM}\} \to \cM $ 
		be bijective, 
		let 
    $ f \colon \cM \to \R $
    and
		$ F \colon \R^d \to \R $ 
		satisfy for all 
		$ x = (x_1,x_2,\ldots,x_d) \in \R^d $ 
		that 
		\begin{equation}  
		F(x) 
		= 
		\max_{y=(y_1,y_2\ldots,y_d)\in\cM} 
		\left[ 
		f(y) - L\!\left(\sum_{i=1}^{d} |x_i-y_i| \right)
		\right]\!, 
		\end{equation}
		let
		$ \sW{z}_1 \in \R^{(2d)\times d} $, 
		$ \sW{z}_2 \in \R^{1\times (2d)} $,
		and $ \sB{z}\in\R^{2d}$, $z\in\cM$, 
		satisfy for all $z=(z_1,z_2,\dots,z_d)\in\cM$ that 
		\begin{equation} \label{neural_net_representation:subnets}
		\sW{z}_1 = 
		\begin{pmatrix} 
		1 &  0 & \cdots &  0 \\
		-1 &  0 & \cdots &  0 \\
		0 &  1 & \cdots &  0 \\
		0 & -1 & \cdots &  0 \\
		\vdots  & \vdots &  \ddots & \vdots \\
		0 &  0 & \cdots &  1 \\ 
		0 &  0 & \cdots & -1
		\end{pmatrix}, 
		\qquad 
		\sB{z}
		= 
		\begin{pmatrix} 
		-z_1 \\
		z_1 \\
		-z_2 \\
		z_2 \\
		\vdots \\
		-z_d \\
		z_d
		\end{pmatrix}, 
		\qquad\text{and}\qquad
		\sW{z}_2 
		= 
		\begin{pmatrix} 
		-L &-L &\cdots &-L 
		\end{pmatrix}, 
		\end{equation} 
    let 
		$ \bW_1 \in \R^{(2d\card\cM)\times d}$, 
		$ \bB_1 \in \R^{2d\card\cM}$, 
		$ \bW_2 \in \R^{\card\cM \times (2d\card\cM)}$, 
		$ \bB_2 \in \R^{\card\cM}$ 
		satisfy 
		\begin{equation} 
		\label{neural_net_representation:weights_and_biases_Phi_net}
		\bW_1 = \begin{pmatrix} 
		\sW{m(1)}_1 \\
		\sW{m(2)}_1 \\
		\vdots \\
		\sW{m(\card\cM)}_1
		\end{pmatrix}, 
		\;\;
		\bB_1 = \begin{pmatrix} 
		\sB{m(1)} \\
		\sB{m(2)} \\
		\vdots \\
		\sB{m(\card\cM)}
		\end{pmatrix}, 
		\;\;
		\bW_2 = \begin{pmatrix} 
      \sW{m(1)}_2 & 0 & \cdots & 0 \\  
      0 & \sW{m(2)}_2 & \cdots & 0  \\
      \vdots & \vdots & \ddots & \vdots \\
      0 & 0 & \cdots &   
      \sW{m(\card\cM)}_2
      \end{pmatrix} 
		,\;\; \text{and}\;\;
		\bB_2 
		= 
		\begin{pmatrix} 
		f(m(1)) \\
		f(m(2)) \\
		\vdots \\
		f(m(\card\cM))
		\end{pmatrix},
		\end{equation}
		and let 
		$\Phi\in\bN$ 
		satisfy 
		$\Phi
		=\phi_{\card\cM}\bullet ( (\bW_1, \bB_1), (\bW_2, \bB_2) )$
		(cf.\ \cref{def:ANN,def:ReLu:identity,def:simpleParallelization,def:ANNcomposition} and \cref{lem:max_net}). 
		Then
		\begin{enumerate}[label=(\roman{*})]
			\item \label{neural_net_representation:item1}
			it holds that 
			$\cD(\Phi) = (d,2d\card\cM,2\card\cM-1,2\card\cM-3,\ldots,3,1)\in\N^{\card\cM+2}$, 
			\item 
			\label{neural_net_representation:item2}
			it holds that 
			$\cL(\Phi) = \card\cM+1$, 
			\item 
			\label{neural_net_representation:item4}
			it holds that 
			$\infnorm{\MappingStructuralToVectorized(\Phi)} \leq \max\{1,L, \sup_{z\in\cM} \infnorm{z} , 2[\sup_{z\in\cM} |f(z)|] \}$, 	
			and 
			\item 
			\label{neural_net_representation:item5}
			it holds that 
			$F = \cR_{\mf r}(\Phi)$
		\end{enumerate}
    (cf.\ \cref{def:TranslateStructuredIntoVectorizedDescription,def:infnorm,def:relu1,def:ANNrealization}%
    ).
	\end{lemma} 
	\begin{proof}[Proof of \cref{lem:neural_net_representation}]
		Throughout this proof let 
		$\Psi\in\bN$ 
		satisfy 
		$\Psi=((\bW_1,\bB_1),(\bW_2,\bB_2))$
    and let
    \newcommand{\lm}[2]{\mf m_{#1,#2}}%
    $\lm ij\in\R$, $i\in\{1,2,\dots,\card\cM\}$, $j\in\{1,2\dots,d\}$,
    satisfy for all $i\in\{1,2,\dots,\card\cM\}$, $j\in\{1,2\dots,d\}$
    that $m(i)=(\lm i1,\lm i2,\dots,\lm id)$.
    Note that 
      \cref{lem:max_net}
    establishes that there exist
      \newcommand{\WW}{\mf W}%
      \newcommand{\BB}{\mf B}%
      $\WW_1\in\R^{(2\card\cM-1)\times \card\cM}$,
      $\BB_1\in\R^{2\card\cM-1}$,
      $\WW_2\in\R^{(2\card\cM-3)\times(2\card\cM-1)}$,
      $\BB_2\in\R^{2\card\cM-3}$,
      $\dots$,
      $\WW_{\card\cM-1}\in\R^{3\times 5}$,
      $\BB_{\card\cM-1}\in\R^3$,
      $\WW_{\card\cM}\in\R^{1\times 3}$,
      $\BB_{\card\cM}\in\R$
      such that
      \begin{equation}\label{eq:phiM}
        \phi_{\card\cM}=( (\WW_{1}, \BB_{1}), (\WW_{2}, \BB_{2}), \ldots, (\WW_{\card\cM}, \BB_{\card\cM}) ).
      \end{equation}
		Next observe that \enum{
      \eqref{neural_net_representation:weights_and_biases_Phi_net}
    }[establish] that
		$\lengthANN(\Psi)=2$ 
		and 
		\begin{equation} \label{neural_net_representation:architectureOfPsi}
		\dims(\Psi) = (d,2d\card\cM,\card\cM). 
		\end{equation} 
		Moreover, note that 
		item~\ref{max_net:item2} 
		in 
		\cref{lem:max_net}
		ensures that
		\begin{equation} 
    \cD(\phi_{\card{\mc M}}) = (\abs{\mc M}, 2\abs{\mc M}-1, 2\abs{\mc M}-3, \ldots, 3, 1)
    \in\N^{\card\cM+1}
    . 
		\end{equation} 
      This, the fact that 
      $\Phi=\phi_{\card\cM}\bullet\Psi$, 
      \eqref{neural_net_representation:architectureOfPsi}, 
      and item (i) in~\cite[Proposition~2.6]{Grohs2019ANNCalculus}
	  show that 
		$\lengthANN(\Phi) = \card\cM + 1$ 
		and 
		\begin{equation} 
		\dims(\Phi) 
		= 
    (d,2d\card\cM,2\card\cM-1,2\card\cM-3,\ldots,3,1)
    \in\N^{\card\cM+2}
    . 
		\end{equation} 
		This establishes items~\ref{neural_net_representation:item1} and \ref{neural_net_representation:item2}. 
		In the next step we note that 
		the hypothesis that 
		$\Phi=\phi_{\card\cM}\bullet((\bW_1,\bB_1),(\bW_2,\bB_2))$ 
		and 
		\eqref{eq:phiM}
		ensure that 
		\begin{equation} 
		\begin{split}
		\Phi 
		& = 
		( 
		(\WW_{1}, \BB_{1}), 
		(\WW_{2}, \BB_{2}), 
		\ldots, 
		(\WW_{\card\cM}, \BB_{\card\cM}) 
		) 
		\bullet 
		( (\bW_1, \bB_1), (\bW_2, \bB_2) )
		\\
		& = 
		( 
		( \cW_1, \cB_1), 
		( \WW_{1}\cW_2,   
		\WW_{1}\cB_2+\BB_{1} ), 
		( \WW_{2}, \BB_{2}), 
		\ldots, 
		( \WW_{\card\cM}, \BB_{\card\cM}) 
		). 
		\end{split}
		\end{equation}  
		\cref{lem:structtovect} hence implies that 
		\begin{equation} \label{neural_net_representation:Theta_vector}
		\MappingStructuralToVectorized(\Phi) 
		= 
		\bigl( 
		\MappingStructuralToVectorized\bigl(((\bW_1,\bB_1))\bigr)
		,
		\MappingStructuralToVectorized\bigl(((\WW_{1}\cW_2, \WW_{1}\cB_2+\BB_{1}))\bigr)
		,
		\MappingStructuralToVectorized\bigl(((\WW_{2},\BB_{2}))\bigr)
		,
		\ldots
		,
		\MappingStructuralToVectorized\bigl(((\WW_{\card\cM}, \BB_{\card\cM}))\bigr)
    \bigr)
		\end{equation} 
	(cf.\ \cref{def:TranslateStructuredIntoVectorizedDescription}).
    Moreover, note that \eqref{neural_net_representation:weights_and_biases_Phi_net} and
    item~\ref{max_net2:item1} in \cref{lem:max_net2} imply that 
		\begin{equation} \label{eq:W1W2}
		\begin{split}
		\WW_{1}\bW_2 
		= 
		\underbrace{
			\begin{pmatrix} 
			1 		& -1 	 &  0 	  & \cdots &  0 	 \\
			0 		&  1 	 &  0 	  & \cdots &  0 	 \\
			0 		& -1 	 &  0 	  & \cdots &  0 	 \\
			0 		&  0 	 &  1 	  & \cdots &  0 	 \\
			0 		&  0 	 & -1 	  & \cdots &  0 	 \\
			\vdots & \vdots & \vdots & \ddots &  \vdots \\
			0 		&  0 	 & 0 	  & \cdots &  1 	 \\
			0 		&  0 	 & 0 	  & \cdots & -1
			\end{pmatrix} }_{\in\R^{(2\card\cM-1)\times\card\cM}}
		\bW_2 
    = 
    \underbrace{
		\begin{pmatrix} 
      \sW{m(1)}_2 & -\sW{m(2)}_2 & 0 & \cdots & 0 \\
      0 & \sW{m(2)}_2 & 0 & \cdots & 0 \\
      0 & -\sW{m(2)}_2 & 0 & \cdots & 0 \\
      0 & 0 & \sW{m(3)}_2 & \cdots & 0 \\
      0 & 0 & -\sW{m(3)}_2 & \cdots & 0 \\
      \vdots & \vdots & \vdots & \ddots & \vdots \\
      0 & 0 & 0 & \cdots & \sW{m(\card\cM)}_2 \\
      0 & 0 & 0 & \cdots & -\sW{m(\card\cM)}_2
      \end{pmatrix}
    }_{\in\R^{(2\card\cM-1)\times(2d\card\cM)}}
    .
      \end{split}  
		\end{equation}
		In addition, observe that \eqref{neural_net_representation:weights_and_biases_Phi_net} and
    items~\ref{max_net2:item1} and \ref{max_net2:item2} in \cref{lem:max_net2} show that
		\begin{equation} 
		\begin{split}
		\WW_{1}\bB_2 + \BB_{1} 
    & = 
    \underbrace{
		\begin{pmatrix} 
		1 		& -1 	 &  0 	  & \cdots &  0 	 \\
		0 		&  1 	 &  0 	  & \cdots &  0 	 \\
		0 		& -1 	 &  0 	  & \cdots &  0 	 \\
		0 		&  0 	 &  1 	  & \cdots &  0 	 \\
		0 		&  0 	 & -1 	  & \cdots &  0 	 \\
		\vdots & \vdots & \vdots & \ddots &  \vdots \\
		0 		&  0 	 & 0 	  & \cdots &  1 	 \\
		0 		&  0 	 & 0 	  & \cdots & -1
    \end{pmatrix} 
    }_{\in\R^{(2\card\cM-1)\times\card\cM}}
		\bB_2 
    +
    \underbrace{ 
		\begin{pmatrix} 
		0 \\ 0 \\ 0 \\ 0 \\ 0 \\ \vdots \\ 0 \\ 0 
    \end{pmatrix} 
    }_{\in\R^{2\card\cM-1}}
		\\
    & = 
    \underbrace{
		\begin{pmatrix} 
		1 		& -1 	 &  0 	  & \cdots &  0 	 \\
		0 		&  1 	 &  0 	  & \cdots &  0 	 \\
		0 		& -1 	 &  0 	  & \cdots &  0 	 \\
		0 		&  0 	 &  1 	  & \cdots &  0 	 \\
		0 		&  0 	 & -1 	  & \cdots &  0 	 \\
		\vdots & \vdots & \vdots & \ddots &  \vdots \\
		0 		&  0 	 & 0 	  & \cdots &  1 	 \\
		0 		&  0 	 & 0 	  & \cdots & -1
    \end{pmatrix} 
    }_{\in\R^{(2\card\cM-1)\times\card\cM}}
    \underbrace{
		\begin{pmatrix} 
		f(m(1)) \\
		f(m(2)) \\
		\vdots \\
		f(m(\card\cM))
    \end{pmatrix}  
    }_{\in\R^{\card\cM}}
    = 
    \underbrace{
		\begin{pmatrix} 
		f(m(1)) - f(m(2)) \\
		f(m(2)) \\
		- f(m(2)) \\
		f(m(3)) \\
		- f(m(3)) \\
		\vdots \\
		f(m(\card\cM)) \\
		-f(m(\card\cM))
    \end{pmatrix}
    }_{\in\R^{2\card\cM-1}}
    .  
		\end{split}
		\end{equation} 
		This and \eqref{eq:W1W2} demonstrate that
		\begin{equation} 
		\begin{split} 
		& 
		\asinfnorm{\MappingStructuralToVectorized\bigl(((\WW_{1}\cW_2, \WW_{1}\bB_2+\BB_{1}))\bigr)}
		\\
		& = 
		\max\{ L, |f(m(1))-f(m(2))|, |f(m(2))|, |f(m(3))|, \ldots, |f(m(\card\cM))| \} 
		\leq 
		\max\biggl\{ L, 2\!\left[\sup_{z\in\cM} |f(z)| \right]\biggr\}
		\end{split}
    \end{equation} 
    (cf.\ \cref{def:infnorm}).
		Combining this,
    \eqref{neural_net_representation:weights_and_biases_Phi_net},  
    and item~\ref{max_net2:item5} in \cref{lem:max_net2} 
    with \eqref{neural_net_representation:Theta_vector}
    proves that 
		\begin{equation} 
		\begin{split}
		\infnorm{\MappingStructuralToVectorized(\Phi)} 
		& 
		\leq 
    \max\bigl\{ \asinfnorm{\MappingStructuralToVectorized\bigl(((\bW_1,\bB_1))\bigr)}, 
    \asinfnorm{\MappingStructuralToVectorized\bigl(((\WW_{1}\bW_2, \WW_{1}\bB_2+\BB_{1}))\bigr) }, 
		\infnorm{\MappingStructuralToVectorized( \phi_{\card\cM})} \bigr\} 
		\\
		& \leq
		\max\biggl\{1,\sup_{z\in\cM} \infnorm{z}, L, 2\bigg[ \sup_{z\in\cM} |f(z)| \bigg]\biggr\}. 
		\end{split}
		\end{equation} 
		This establishes item~\ref{neural_net_representation:item4}. 
		Observe that \eqref{neural_net_representation:subnets} ensures that for all 
		$x=(x_1,x_2,\ldots,x_d)\in\R^d$, 
		$z=(z_1,z_2,\ldots,z_d)\in\cM$ 
		it holds that 
		\begin{equation} 
		\sW{z}_1 x + \sB{z}
		= 
		\underbrace{
			\begin{pmatrix} 
			1 &  0 & \cdots &  0 \\
			-1 &  0 & \cdots &  0 \\
			0 &  1 & \cdots &  0 \\
			0 & -1 & \cdots &  0 \\
			\vdots  & \vdots &  \ddots & \vdots \\
			0 &  0 & \cdots &  1 \\ 
			0 &  0 & \cdots & -1
			\end{pmatrix}}_{\in\R^{(2d)\times d}}
		\begin{pmatrix} 
		x_1 \\
		x_2 \\
		\vdots \\
		x_d 
		\end{pmatrix}
		+ 
		\begin{pmatrix} 
		-z_1 \\
		z_1 \\
		-z_2 \\
		z_2 \\
		\vdots \\
		-z_d \\
		z_d 
		\end{pmatrix}  
		=
		\begin{pmatrix} 
		x_1 \\
		-x_1 \\
		x_2 \\
		-x_2 \\
		\vdots \\
		x_d \\
		-x_d  
		\end{pmatrix} 
		+ 
		\begin{pmatrix} 
		-z_1 \\
		z_1 \\
		-z_2 \\
		z_2 \\
		\vdots \\
		-z_d \\
		z_d 
		\end{pmatrix} 
		= 
		\begin{pmatrix} 
		x_1 - z_1 \\
		-(x_1 - z_1) \\
		x_2 - z_2 \\
		-(x_2 - z_2) \\
		\vdots \\
		x_d - z_d \\
		-(x_d - z_d) 
		\end{pmatrix}.  
		\end{equation} 
		This and \eqref{neural_net_representation:weights_and_biases_Phi_net} prove that for all 
		$x=(x_1,x_2,\ldots,x_d)\in\R^d$, 
		$z=(z_1,z_2,\ldots,z_d)\in\cM$
		it holds that 
		\begin{equation} \label{neural_net_representation:zwischenschritt}
		\begin{split}
		\sW{z}_2 \bigl( \mathfrak{R}_{2d}(\sW{z}_1 x + \sB{z} ) \bigr)
		& = 
		\underbrace{
			\begin{pmatrix} 
			-L & -L & \cdots & -L 
			\end{pmatrix}}_{\in\R^{1\times(2d)}} 
		\begin{pmatrix} 
		\max\{x_1-z_1,0\} \\
		\max\{z_1-x_1,0\} \\
		\max\{x_2-z_2,0\} \\
		\max\{z_2-x_2,0\} \\
		\vdots \\
		\max\{x_d-z_d,0\} \\
		\max\{z_d-x_d,0\}
		\end{pmatrix}
		\\
		& = 
		-L \!\left[\sum_{i=1}^d \left( \max\{x_i-z_i,0\} + \max\{z_i-x_i,0\} \right) \right]
		= 
		-L \!\left[\sum_{i=1}^d |x_i-z_i|\right]
		\end{split}
		\end{equation} 
		(cf.\ Definition~\ref{def:relu}).
		Moreover, note that \eqref{neural_net_representation:weights_and_biases_Phi_net} implies that for all 
		$x\in\R^d$ 
		it holds that 
		\begin{equation}
		\bW_1 x + \bB_1
		= 
		\begin{pmatrix} 
		\sW{m(1)}_1 x + \sB{m(1)} \\
		\sW{m(2)}_1 x + \sB{m(2)} \\
		\vdots \\
		\sW{m(\card\cM)}_1 x + \sB{m(\card\cM)}
		\end{pmatrix}. 
		\end{equation}
		Therefore, we obtain that for all 
		$x\in\R^d$ 
		it holds that 
		\begin{equation} 
		\mathfrak{R}_{2d\card\cM}( 
		\bW_1 x + \bB_1 )
		= 
		\begin{pmatrix} 
		\mathfrak{R}_{2d}(\sW{m(1)}_1 x + \sB{m(1)}) \\
		\mathfrak{R}_{2d}(\sW{m(2)}_1 x + \sB{m(2)}) \\
		\vdots \\
		\mathfrak{R}_{2d}(\sW{m(\card\cM)}_1 x + \sB{m(\card\cM)})
		\end{pmatrix}.  
		\end{equation} 
		This, \eqref{neural_net_representation:weights_and_biases_Phi_net}, 
		and 
		\eqref{neural_net_representation:zwischenschritt} imply that for all 
		$x=(x_1,x_2,\dots,x_d)\in\R^d$ 
		it holds that 
		\begin{equation} 
		\begin{split}
		\bigl(\cR_{\mf r}(\Psi)\bigr)(x) 
		&= 
		\bW_2 
		\bigl(\mathfrak{R}_{2d\card\cM}( \bW_1 x + \bB_1 )\bigr)
		+ 
		\bB_2 
		\\
		& = 
		\begin{pmatrix} 
    \sW{m(1)}_2 & 0 & \cdots & 0 \\  
    0 & \sW{m(2)}_2 & \cdots & 0  \\
    \vdots & \vdots & \ddots & \vdots \\
    0 & 0 & \cdots &   
    \sW{m(\card\cM)}_2
    \end{pmatrix} 
    \begin{pmatrix} 
		\mathfrak{R}_{2d}(\sW{m(1)}_1 x + \sB{m(1)}) \\
		\mathfrak{R}_{2d}(\sW{m(2)}_1 x + \sB{m(2)}) \\
		\vdots \\
		\mathfrak{R}_{2d}(\sW{m(\card\cM)}_1 x + \sB{m(\card\cM)})
		\end{pmatrix}
		+ 
		\begin{pmatrix} 
		f(m(1)) \\
		f(m(2)) \\
		\vdots \\
		f(m(\card\cM)) 
		\end{pmatrix}
		\\
		& = 
		\begin{pmatrix} 
		\sW{m(1)}_2 
		\bigl(\mathfrak{R}_{2d}(\sW{m(1)}_1 x + \sB{m(1)})  \bigr)
		\\
		\sW{m(2)}_2 
		\bigl(\mathfrak{R}_{2d}(\sW{m(2)}_1 x + \sB{m(2)})  \bigr)
		\\
		\vdots \\
		\sW{m(\card\cM)}_2 
		\bigl(\mathfrak{R}_{2d}(\sW{m(\card\cM)}_1 x + \sB{m(\card\cM)})  \bigr)
		\end{pmatrix}
		+ 
		\begin{pmatrix} 
		f(m(1)) \\
		f(m(2)) \\
		\vdots \\
		f(m(\card\cM)) 
		\end{pmatrix}
		\\
		& = 
		\begin{pmatrix} 
		f(m(1)) - L \bigl[\sum_{i=1}^d |x_i-\lm 1i|\bigr] \\
		f(m(2)) - L \bigl[\sum_{i=1}^d |x_i-\lm 2i|\bigr] \\
		\vdots \\
		f(m(\card\cM)) - L \bigl[\sum_{i=1}^d |x_i-\lm{\card\cM}i| \bigr]
		\end{pmatrix}
		\end{split}
		\end{equation} 
		(cf.\ \cref{def:relu1,def:ANNrealization}).
    This,
    the fact that 
    $\Phi=\phi_{\card\cM}\bullet\Psi$,
    item~\ref{max_net:item3} in \cref{lem:max_net},
    and item (v) in \cite[Proposition~2.6]{Grohs2019ANNCalculus} ensure that for all 
		$x=(x_1,x_2,\dots,x_d)\in\R^d$ 
		it holds that 
		\begin{equation} 
		\begin{split}
		\bigl(\functionANN{\mf r}(\Phi)\bigr)(x) 
		& = 
		\left(\bigl[\functionANN{\mf r}(\phi_{\card\cM})\bigr]\circ\bigl[\functionANN{\mf r}(\Psi)\bigr]\right)\!(x)
		= 
		\max_{i\in\{1,2,\ldots,\card\cM\}} 
		\left[ 
		f(m(i)) - L\!\left(\sum_{j=1}^d |x_j-\lm ij|\right) 
		\right]
		\\
		& = 
		\max_{y=(y_1,y_2\ldots,y_d)\in\cM}
		\left[ 
		f(z) - L\!\left(\sum_{i=1}^d |x_i-y_i|\right)
		\right]\!.
		\end{split}
		\end{equation} 
		This establishes item~\ref{neural_net_representation:item5}. The proof of \cref{lem:neural_net_representation} is thus completed. 
	\end{proof}

	\subsubsection{Explicit approximations through DNNs}
	
	\begin{prop} \label{prop:neural_net_approximation}
		Let 
		$\phi_k\in\bN$, $k\in\{2,3,\ldots\}$, 
		satisfy for all 
		$k\in\{2,3,\ldots\}$ 
    that 
    $\inDimANN(\phi_k)=\outDimANN(\bP_2(\phi_2,\mf I_{k-1}))$,
		$\phi_{k+1} = \phi_{k} \bullet (\bP_{2}(\phi_2, \mathfrak{I}_{k-1}))$,
		and 
		\begin{equation} 
		\phi_2 
		= 
		\left( \rule{0cm}{1cm}
		\left(
		\begin{pmatrix} 
		1 & -1 \\
		0 & 1 \\
		0 & -1 
		\end{pmatrix}, 
		\begin{pmatrix} 
		0 \\ 0 \\ 0 	
		\end{pmatrix} 
		\right), 
		\left(
		\begin{pmatrix} 
		1 & 1 & -1
		\end{pmatrix}, 
		0  
		\right) 
    \right)
    \in
    \bigl((\R^{3\times 2}\times\R^{3})\times(\R^{1\times 3}\times\R)\bigr),
		\end{equation} 
    let
		$d\in\N$, 
		$L\in\R$, 
		let 
		$D\subseteq\R^d$ be a set, 
		let 
		$f\colon D \to \R$ 
		satisfy for all 
		$x=(x_1,x_2,\ldots,x_d)$,
		$y=(y_1,y_2,\ldots,y_d)\in D$ 
		that 
		$|f(x)-f(y)|\leq L\bigl[\sum_{i=1}^d |x_i-y_i|\bigr]$, 
		let 
    $\cM\subseteq D$
    satisfy
		$\card\cM \in \{2,3,\ldots\}$, 
		let 
		$m\colon\{1,2,\ldots,\card\cM\} \to \cM$ 
		be bijective,
		let
		$ \sW{z}_1 \in \R^{(2d)\times d} $, 
		$ \sW{z}_2 \in \R^{1\times (2d)} $,
		and $ \sB{z}\in\R^{2d}$, $z\in\cM$, 
		satisfy for all $z=(z_1,z_2,\dots,z_d)\in\cM$ that 
		\begin{equation} \label{neural_net_approximation:subnets}
		\sW{z}_1 = 
		\begin{pmatrix} 
      1 &  0 & \cdots &  0 \\
      -1 &  0 & \cdots &  0 \\
      0 &  1 & \cdots &  0 \\
      0 & -1 & \cdots &  0 \\
      \vdots  & \vdots &  \ddots & \vdots \\
      0 &  0 & \cdots &  1 \\ 
      0 &  0 & \cdots & -1
    \end{pmatrix}, 
    \qquad 
		\sB{z}
		= 
		\begin{pmatrix} 
		-z_1 \\
		z_1 \\
		-z_2 \\
		z_2 \\
		\vdots \\
		-z_d \\
		z_d
		\end{pmatrix}, 
		\qquad\text{and}\qquad
		\sW{z}_2 
		= 
		\begin{pmatrix} 
		-L & -L & \cdots & -L 
		\end{pmatrix}, 
		\end{equation} 
		let 
		$\bW_1\in \R^{(2d\card\cM)\times d}$, 
		$\bB_1\in \R^{2d\card\cM}$, 
		$\bW_2\in \R^{\card\cM \times (2d\card\cM)}$, 
		$\bB_2\in \R^{\card\cM}$ 
		satisfy 
		\begin{equation} 
		\bW_1 = \begin{pmatrix} 
		\sW{m(1)}_1 \\
		\sW{m(2)}_1 \\
		\vdots \\
		\sW{m(\card\cM)}_1
		\end{pmatrix}, 
		\;\;
		\bB_1 = \begin{pmatrix} 
		\sB{m(1)} \\
		\sB{m(2)} \\
		\vdots \\
		\sB{m(\card\cM)}
		\end{pmatrix}, 
		\;\;
		\bW_2 = \begin{pmatrix} 
    \sW{m(1)}_2 & 0 & \cdots & 0 \\  
    0 & \sW{m(2)}_2 & \cdots & 0  \\
    \vdots & \vdots & \ddots & \vdots \\
    0 & 0 & \cdots &   
    \sW{m(\card\cM)}_2
    \end{pmatrix},  
		\;\;and \;\;
		B_2 = \begin{pmatrix} 
		f(m(1)) \\
		f(m(2)) \\
		\vdots \\
		f(m(\card\cM))
		\end{pmatrix},  
		\end{equation}
		and let 
		$\Phi\in\bN$ 
		satisfy 
		$\Phi=\phi_{\card\cM}\bullet ( (\bW_1,\bB_1) , (\bW_2,\bB_2) )$
		(cf.\ \cref{def:ANN,def:ReLu:identity,def:simpleParallelization,def:ANNcomposition} and \cref{lem:max_net}).
		Then 
		\begin{enumerate}[label=(\roman{*})]
      \item \label{nna:it1}
        it holds that $\dims(\Phi)=(d,2d\card{\mc M},2\card{\mc M}-1,2\card{\mc M}-3,\dots,3,1)\in\N^{\card\cM+2}$,
      \item \label{nna:it2}
        it holds that $\infnorm{\MappingStructuralToVectorized(\Phi)}\leq\max\{1,L, \sup_{z\in\cM} \infnorm{z} , 2[\sup_{z\in\cM} |f(z)|]\}$, and
      \item \label{nna:it3}
        it holds that
        \begin{equation} \label{neural_net_approximation:claim}
          \sup_{x\in D} \left| f(x) - \bigl(\functionANN{\mf r}(\Phi)\bigr)(x) \right| 
          \leq 
          2L\!
          \left[ 
          \sup_{x=(x_1,x_2,\ldots,x_d)\in D} 
          \left( 
          \inf_{y=(y_1,y_2,\ldots,y_d)\in\cM}  \sum_{i=1}^{d} |x_i-y_i|
          \right)
          \right]
        \end{equation}
		\end{enumerate}
    (cf.\ \cref{def:TranslateStructuredIntoVectorizedDescription,def:infnorm,def:relu1,def:ANNrealization}%
    ).
	\end{prop}
	\begin{proof}[Proof of \cref{prop:neural_net_approximation}]
		Throughout this proof let 
		$F\colon \R^d\to\R$ 
		satisfy for all 
		$x=(x_1,x_2,\dots,x_d)\in \R^d$ 
		that 
		\begin{equation} 
		F(x) = \max_{y=(y_1,y_2\ldots,y_d)\in\cM} 
		\left[ 
		f(y) - L\!\left(\sum_{i=1}^d |x_i-y_i|\right)
		\right]\!.
		\end{equation} 
		Observe that \cref{lem:neural_net_representation} establishes that
		\begin{enumerate}[label=(\Alph{*})]
      \item \label{nnr:it1}
        it holds that $\dims(\Phi)=(d,2d\card{\mc M},2\card{\mc M}-1,2\card{\mc M}-3,\dots,3,1)\in\N^{\card\cM+2}$,
      \item \label{nnr:it2}
        it holds that $\infnorm{\MappingStructuralToVectorized(\Phi)}\leq\max\{1,L, \sup_{z\in\cM} \infnorm{z} , 2[\sup_{z\in\cM} |f(z)|]\}$, and
      \item \label{nnr:it3}
        it holds for all $x\in D$ that
      $\bigl(\functionANN{\mf r}(\Phi)\bigr)(x)=F(x)$
    \end{enumerate}
    (cf.\ \cref{def:TranslateStructuredIntoVectorizedDescription,def:infnorm,def:relu1,def:ANNrealization}).
    Note that 
      \cref{nnr:it1,nnr:it2} 
    prove
      \cref{nna:it1,nna:it2}.
    Next observe that
      item~\ref{nnr:it3}
      and \cref{lem:lipschitz_extension} 
      (with
        $E\is D$,
        $\delta\is (D\times D\ni((x_1,x_2,\dots,x_d),(y_1,y_2,\dots,y_d))\mapsto \sum_{i=1}^d\abs{x_i-y_i}\in[0,\infty))$,
        $\mc M\is\mc M$,
        $L\is L$,
        $f\is f$,
        $F\is (D\ni x\mapsto F(x)\in\R\cup\{\infty\})$
      in the notation of Lemma~\ref{lem:lipschitz_extension})
    ensure that 
		\begin{equation} 
		\begin{split}
		&\sup_{x\in D} \left|f(x)-\bigl(\functionANN{\mf r}(\Phi)\bigr)(x)\right| 
		= 
		\sup_{x\in D} 
		\left| f(x) - F(x) \right| 
		\\&\leq 
		2L\!
		\left[ 
		\sup_{x=(x_1,x_2,\ldots,x_d)\in D} 
		\left( 
		\inf_{y=(y_1,y_2,\ldots,y_d)\in\cM}  \sum_{i=1}^d |x_i-y_i| 
		\right) 
		\right]\!.
		\end{split}
		\end{equation} 
    The proof of \cref{prop:neural_net_approximation} is thus completed. 
	\end{proof}
	
	\subsubsection{Implicit approximations through DNNs}
	
	\begin{cor} \label{cor:existence_of_neural_net_approximation_vectorized_description} 
		Let 
		$d,\mathfrak{d}\in\N$, 
		$L\in \R$, 
		let
		$D\subseteq\R^d$ be a set, 
		let 
		$f\colon D \to \R$ 
		satisfy for all 
		$x=(x_1,x_2,\ldots,x_d)$,
		$y=(y_1,y_2,\ldots,y_d)\in D$ 
		that 
		$|f(x)-f(y)|\leq L\bigl[\sum_{i=1}^d |x_i-y_i|\bigr]$, 
		let 
    $\cM\subseteq D$
    satisfy
    $\card\cM \in \{2,3,\ldots\}$, 
    and let $l=(l_0,l_1,\dots,l_{\card\cM+1})\in\N^{\card\cM+2}$
		  satisfy $l=(d,2d\card\cM,2\card\cM-1,2\card\cM-3,\ldots,3,1)$
		  and	$\sum_{k=1}^{\card\cM+1}l_k(l_{k-1}+1)\leq \mathfrak{d}$. 
Then there exists 
	$\theta\in\R^{\mathfrak{d}}$
	such that
$\infnorm{\theta} \leq \max\{1,L, \sup_{z\in\cM} \infnorm{z} , 2[\sup_{z\in\cM} |f(z)|] \}$ 
and
\begin{equation} \label{existence_of_neural_net_approximation_vectorized_description:claim}
	\sup_{x\in D} 
	\,\bigl\lvert f(x) - (\UnclippedRealV{\theta}{l})(x) \bigr\rvert
	\leq 
	2L\! \left[ 
	\sup_{x=(x_1,x_2,\ldots,x_d)\in D} 
	\left( 
	\inf_{y=(y_1,y_2,\ldots,y_d)\in\cM}  \sum_{i=1}^d |x_i-y_i| 
	\right) 
	\right]
  \end{equation}
  (cf.\ \cref{def:infnorm,def:rectclippedFFANN}).
\end{cor} 
	\begin{proof}[Proof of \cref{cor:existence_of_neural_net_approximation_vectorized_description}] 
		Observe that \enum{
		  \cref{prop:neural_net_approximation} ;
		  item~\ref{max_net:item0} in \cref{lem:max_net} ;
    }[ensure] that there exists 
		$\Phi\in\ANNs$ 
		such that
    \begin{enumerate}[label=(\Alph{*})]
      \item 
        it holds that $\dims(\Phi)=l$,
      \item 
        it holds that $\infnorm{\MappingStructuralToVectorized(\Phi)}\leq\max\{1,L, \sup_{z\in\cM} \infnorm{z} , 2[\sup_{z\in\cM} |f(z)|]\}$, and
      \item 
        it holds that
        \begin{equation} 
          \sup_{x\in D} \left| f(x) - \bigl(\functionANN{\mf r}(\Phi)\bigr)(x) \right|
          \leq 
          2L\!
          \left[ 
          \sup_{x=(x_1,x_2,\ldots,x_d)\in D} 
          \left( 
          \inf_{y=(y_1,y_2,\ldots,y_d)\in\cM}  \sum_{i=1}^{d} |x_i-y_i|
          \right)
          \right]
        \end{equation}
    \end{enumerate}
    (cf.\ \cref{def:ANN,def:TranslateStructuredIntoVectorizedDescription,def:infnorm,def:relu1,def:ANNrealization}).
    Combining this with
    \cref{cor:structvsvect}
		establishes \eqref{existence_of_neural_net_approximation_vectorized_description:claim}. 
		The proof of \cref{cor:existence_of_neural_net_approximation_vectorized_description} is thus completed. 
	\end{proof} 
\begin{cor} \label{cor:existence_of_clipped_neural_net_approximation_vectorized_description} 
Let 
	$d,\mathfrak{d}\in\N$, 
	$L\in \R$, 
	$u\in[-\infty,\infty)$,
	$v\in(u,\infty]$,
let
	$D\subseteq\R^d$ be a set, 
let 
	$f\colon D \to [u,v]$ 
satisfy for all 
	$x=(x_1,x_2,\ldots,x_d)$,
	$y=(y_1,y_2,\ldots,y_d)\in D$ 
that 
	$|f(x)-f(y)|\leq L\bigl[\sum_{i=1}^d |x_i-y_i|\bigr]$, 
let 
  $\cM\subseteq D$
  satisfy
	$\card\cM \in \{2,3,\ldots\}$, 
let $l=(l_0,l_1,\dots,l_{\card\cM+1})\in\N^{\card\cM+2}$
  satisfy $l=(d,2d\card\cM,2\card\cM-1,2\card\cM-3,\ldots,3,1)$
  and $\mathfrak{d}\geq\sum_{k=1}^{\card\cM+1}l_k(l_{k-1}+1)$. 
Then there exists 
	$\theta\in\R^{\mathfrak{d}}$ 
such that 
$\infnorm{\theta} \leq \max\{1,L, \sup_{z\in\cM} \infnorm{z} ,\allowbreak 2[\sup_{z\in\cM} |f(z)|] \}$
and
	\begin{equation} \label{existence_of_clipped_neural_net_approximation_vectorized_description:claim}
	\sup_{x\in D} 
	\left| f(x) - \ClippedRealV{\theta}{l}uv(x) \right| 
	\leq 
	2L\!\left[ 
	\sup_{x=(x_1,x_2,\ldots,x_d)\in D} 
	\left( 
	\inf_{y=(y_1,y_2,\ldots,y_d)\in\cM}  \sum_{i=1}^d |x_i-y_i| 
	\right) 
	\right]
	\end{equation} 
(cf.\ \cref{def:infnorm,def:rectclippedFFANN}). 	
\end{cor} 
\begin{proof}[Proof of \cref{cor:existence_of_clipped_neural_net_approximation_vectorized_description}]
First, observe that \cref{cor:existence_of_neural_net_approximation_vectorized_description} (with 
	$ d \is d $, 
	$ \mf d \is \mf d $, 
	$ L \is L $, 
	$ D \is D $, 
	$ f \is (D\ni x\mapsto f(x)\in\R) $, 
	$ \mc M \is \mc M $,
	$l\is l$
in the notation of \cref{cor:existence_of_neural_net_approximation_vectorized_description}) ensures that there exists 
	$\theta\in\R^{\mathfrak{d}}$
which satisfies 
$\infnorm{\theta} \leq \max\{1,L, \sup_{z\in\cM} \infnorm{z} , 2[\sup_{z\in\cM} |f(z)|] \}$  
and
	\begin{equation} 
	\sup_{x\in D} 
	\,\bigl\lvert f(x) - (\UnclippedRealV{\theta}{l})(x) \bigr\rvert 
	\leq 
	2L\!\left[ 
	\sup_{x=(x_1,x_2,\ldots,x_d)\in D} 
	\left( 
	\inf_{y=(y_1,y_2,\ldots,y_d)\in\cM}  \sum_{i=1}^d |x_i-y_i| 
	\right) 
  \right]
	\end{equation} 
(cf.\ \cref{def:infnorm,def:rectclippedFFANN}).
The assumption that for all 
	$x\in D$ 
it holds that 
	$u \leq f(x) \leq v$
and \cref{lem:clipcontr}
hence imply that 
	\begin{equation} 
	\begin{split}
	& 
	\sup_{x\in D} 
	\,\babs{ f(x) - \ClippedRealV{\theta}{l}uv(x) }
    = 
	\sup_{x\in D} 
	\,\babs{ \Clip uv1(f(x)) -  \Clip uv1\bigl((\UnclippedRealV{\theta}{l})(x)\bigr) }
	\\
	& \leq 
	\sup_{x\in D} 
	\,\babs{ f(x) - (\UnclippedRealV{\theta}{l})(x) }
	\leq 
	2L\!\left[ 
	\sup_{x=(x_1,x_2,\ldots,x_d)\in D} 
	\left( 
	\inf_{y=(y_1,y_2,\ldots,y_d)\in\cM}  \sum_{i=1}^d |x_i-y_i| 
	\right) 
  \right]
	\end{split}
	\end{equation} 	
	(cf.\ \cref{def:clip}).
The proof of \cref{cor:existence_of_clipped_neural_net_approximation_vectorized_description}
is thus completed.
\end{proof} 

\begin{cor} \label{cor:existence_of_clipped_neural_net_approximation_vectorized_description2} 
Let 
  $d,\mathfrak{d},\mf L\in\N$, 
  $L\in \R$, 
  $u\in[-\infty,\infty)$,
  $v\in(u,\infty]$,
let
  $D\subseteq\R^d$ be a set, 
let 
  $f\colon D \to ([u,v]\cap\R)$ 
satisfy for all 
  $x=(x_1,x_2,\ldots,x_d)$,
  $y=(y_1,y_2,\ldots,y_d)\in D$ 
that 
  $|f(x)-f(y)|\leq L\bigl[\sum_{i=1}^d |x_i-y_i|\bigr]$, 
let 
  $\cM\subseteq D$
  satisfy
  $\card\cM \in \{2,3,\ldots\}$, 
let $l=(l_0,l_1,\dots,l_{\mf L})\in\N^{\mf L+1}$
  satisfy for all
    $k\in\{2,3,\dots,\card\cM\}$
  that
    $\mf L\geq \card\cM+1$,
    $\sum_{i=1}^{\mf L}l_i(l_{i-1}+1)\leq\mf d$,
    $l_0=d$,
    $l_{\mf L}=1$,
    $l_1\geq 2d\card\cM$,
    and $l_k\geq 2\card\cM-2k+3$,
and assume for all
  $i\in\N\cap(\card\cM,\mf L)$
that 
  $l_i\geq 2$. 
Then there exists 
  $\theta\in\R^{\mathfrak{d}}$
such that 
  $\infnorm{\theta} \leq \max\{1,L, \sup_{z\in\cM} \infnorm{z} , 2[\sup_{z\in\cM} |f(z)|] \}$
  and
  \begin{equation} \label{existence_of_clipped_neural_net_approximation_vectorized_description2:claim}
  \sup_{x\in D} 
  \,\babs{ f(x) - \ClippedRealV{\theta}{l}uv(x) }
  \leq 
  2L\!\left[ 
  \sup_{x=(x_1,x_2,\ldots,x_d)\in D} 
  \left( 
  \inf_{y=(y_1,y_2,\ldots,y_d)\in\cM}  \sum_{i=1}^d |x_i-y_i| 
  \right) 
  \right]
  \end{equation} 
(cf.\ \cref{def:infnorm,def:rectclippedFFANN}).   
\end{cor}
\begin{proof}[Proof of \cref{cor:existence_of_clipped_neural_net_approximation_vectorized_description2}]
  Throughout this proof let
    $\mf l=(\mf l_0,\mf l_1,\dots,\mf l_{\card\cM+1})\in\N^{\card\cM+2}$
  satisfy $\mf l=(d,2d\card\cM,\allowbreak 2\card\cM-1,2\card\cM-3,\ldots,3,1)$.
  First, note that
    \cref{cor:existence_of_clipped_neural_net_approximation_vectorized_description}
    (with
      $d\is d$,
      $\mf d\is\sum_{k=1}^{\card\cM+1}\mf l_k(\mf l_{k-1}+1)$,
      $L\is L$,
      $u\is u$,
      $v\is v$,
      $D\is D$,
      $f\is f$,
      $\cM\is\cM$,
      $l\is\mf l$
    in the notation of \cref{cor:existence_of_clipped_neural_net_approximation_vectorized_description})
  establishes that there exists
    $\eta\in\R^{\sum_{k=1}^{\card\cM+1}\mf l_k(\mf l_{k-1}+1)}$
  which satisfies
  $\infnorm{\eta} \leq \max\{1,L, \sup_{z\in\cM} \infnorm{z} , 2[\sup_{z\in\cM} |f(z)|] \}$
  and
  \begin{equation}
    \label{eq:ecnnav.1}
    \sup_{x\in D} 
    \left| f(x) - \ClippedRealV{\eta}{\mf l}uv(x) \right| 
    \leq 
    2L\!\left[ 
    \sup_{x=(x_1,x_2,\ldots,x_d)\in D} 
    \left( 
    \inf_{y=(y_1,y_2,\ldots,y_d)\in\cM}  \sum_{i=1}^d |x_i-y_i| 
    \right) 
    \right]
  \end{equation} 
  (cf.\ \cref{def:infnorm,def:rectclippedFFANN}). 	
  Next observe that
    \cref{lem:embednet_vectorized}
    (with
      $u\is u$,
      $v\is v$,
      $L\is\card\cM+1$,
      $\mf L\is \mf L$,
      $d\is \sum_{k=1}^{\card\cM+1}\mf l_k(\mf l_{k-1}+1)$,
      $\mf d\is \mf d$,
      $\theta\is\eta$,
      $(l_0,l_1,\dots,l_L)\is (\mf l_0,\mf l_1,\dots,\mf l_{\card\cM+1})$,
      $(\mf l_0,\mf l_1,\dots,\mf l_{\mf L})\is (l_0,l_1,\dots,l_{\mf L})$,
    in the notation of \cref{lem:embednet_vectorized})
  shows that there exists $\theta\in\R^d$ such that
  \begin{equation}
    \infnorm{\theta}\leq\max\{1,\infnorm{\eta}\}
    \qandq
    \ClippedRealV \theta{l}uv=\ClippedRealV \eta{\mf l}uv.
  \end{equation}
  Combining 
    this 
  with 
    \eqref{eq:ecnnav.1} 
  proves 
    \eqref{existence_of_clipped_neural_net_approximation_vectorized_description2:claim}.
  The proof of \cref{cor:existence_of_clipped_neural_net_approximation_vectorized_description2} is thus completed.
\end{proof}

\begin{cor} \label{cor:concrete_clipped_neural_net_approximation}
Let 
	$d,\mathfrak{d},N\in\N$, 	
	$L\in \R$,
	$u\in[-\infty,\infty)$,
	$v\in(u,\infty]$
satisfy 
	$\mathfrak{d}\geq 2d^2(N+1)^d + 5d(N+1)^{2d} + \frac43 (N+1)^{3d}$, 
let 
	$\norm{\cdot}\colon\R^d\to [0,\infty)$ 
be the standard norm, 
let 
	$p=(p_1,p_2,\ldots,p_d)$,
	$q=(q_1,q_2,\ldots,q_d)\in \R^d$ 
satisfy for all 
	$i\in\{1,2,\ldots,d\}$ 
that 
	$p_i \leq q_i$
	and $\max_{j\in\{1,2,\dots,d\}}(q_j-p_j)>0$,
let 	
	$D = \prod_{i=1}^d [p_i,q_i]$, 
let 
	$\cM\subseteq D$ 
satisfy
	\begin{equation} 
	\cM = \left\{
	y=(y_1,y_2,\ldots,y_d)\in\R^d\colon
	\begin{pmatrix} 
	\Exists k_1,k_2,\ldots,k_d\in\{0,1,\ldots,N\}\colon
	\\
	\Forall i\in\{1,2,\ldots,d\}\colon y_i=p_i + \frac{k_i}{N}(q_i-p_i) 
	\end{pmatrix}
	\right\}\!, 
	\end{equation}  
and let 
	$f\colon D \to ([u,v]\cap\R)$ 
satisfy for all 
	$x,y\in D$ 
that 
	$|f(x)-f(y)|\leq L\norm{x-y}$. 
Then there exist 
  $\theta\in\R^{\mf d}$,
	$\mf L\in\N$, 
	$l=(l_0,l_1,\ldots,l_{\mf L})\in\N^{\mf L+1}$ 
such that
  $\infnorm{\theta} \leq \max\{1,L, \infnorm{p}, \infnorm{q} , 2[\sup_{z\in D} |f(z)|] \}$,
	$\sum_{k=1}^{\mf L} l_k(l_{k-1}+1)\leq \mf d$,
and  
	\begin{equation} \label{concrete_clipped_neural_net_approximation:claim}
	\sup_{x\in D} 
	\left| f(x) - \ClippedRealV{\theta}{l}uv(x) \right| 
	\leq 
	\frac{L}{N}\Biggl[ \sum_{i=1}^d |q_i-p_i| \Biggr]
	\end{equation}
	(cf.\ \cref{def:infnorm,def:rectclippedFFANN}).
\end{cor}
\begin{proof}[Proof of \cref{cor:concrete_clipped_neural_net_approximation}]
  Throughout this proof
    let $l=(l_0,l_1,\dots,l_{\card\cM+1})\in\N^{\card\cM+2}$
      satisfy $l\allowbreak=(d,\allowbreak 2d\card\cM,\allowbreak2\card\cM-1,2\card\cM-3,\ldots,3,1)$.
  Observe that \enum{
    the fact that $\card\cM\leq(N+1)^d$ ;
    the fact that for all $n\in\N$ it holds that
      $\sum_{i=1}^n (2i-1)=n^2$ ;
    the fact that for all $n\in\N$ it holds that 
      $\sum_{i=1}^ni^2=\frac{n(n+1)(2n+1)}{6}\leq \frac{(n+1)^3}{3}$ ;
  }[ensure] that
  \begin{equation}
  \label{eq:ccnna3}
  \begin{split}
    &\sum_{k=1}^{\card\cM+1}l_k(l_{k-1}+1)
    \\&= 
    \underbrace{d(2d\card\cM)+2d\card\cM(2\card\cM-1)
      + 
      \Biggl[\sum_{i=1}^{\card\cM-1} (2i+1)(2i-1) \Biggr]}_{\text{number of weights}}
    +
    \underbrace{2d\card\cM
      + 
      \Biggl[\sum_{i=1}^{\card\cM} (2i-1) \Biggr]}_{\text{number of biases}}
    \\
    & = 
    2d^2\card\cM+4d\card\cM^2-2d\card\cM
    +\Biggl[\sum_{i=1}^{\card\cM-1} (4i^2-1)\Biggr]
    +2d\card\cM
    +\card\cM^2
    \\ &=
    2d^2\card\cM 
    + 
    (4d+1)\card\cM^2
    + 
    4 \Biggl[\sum_{i=1}^{\card\cM-1} i^2 \Biggr]
    -
    (\card\cM-1)
    \\&\leq 
    2d^2\card\cM 
    + 
    5d\card\cM^2 
    + 
    \tfrac43 \card\cM^3
    \leq
    2d^2(N+1)^d+5d(N+1)^{2d}+\tfrac43(N+1)^{3d}
    \leq
    \mf d
    .
  \end{split}
  \end{equation}
  In addition, note that
    the hypothesis that for all
      $x,y\in D$
    it holds that
      $\abs{f(x)-f(y)}\leq L\norm{x-y}$
  implies that for all
    $x=(x_1,x_2,\dots,x_d),\, y=(y_1,y_2,\dots,y_d)\in D$
  it holds that
  \begin{equation}
  \label{eq:ccnna2}
    \abs{f(x)-f(y)}
    \leq 
    L\!\left[\sum_{i=1}^d\abs{x_i-y_i}\right]
    \!.
  \end{equation}
  Furthermore, observe that
    the hypothesis that $\max_{j\in\{1,2,\dots,d\}}(q_j-p_j)>0$
  ensures that
    $\card\cM\geq 2$.
  Combining
    this,
    \eqref{eq:ccnna3},
    and \eqref{eq:ccnna2}
  with
    \cref{cor:existence_of_clipped_neural_net_approximation_vectorized_description}
  establishes that there exists 
    $\theta\in\R^{\mf d}$ 
  such that 
  $\infnorm{\theta} \leq \max\{1,L, \sup_{z\in\cM} \infnorm{z} , 2[\sup_{z\in\cM} |f(z)|] \}$
  and
  \begin{equation}
    \label{eq:ccnna7}
    \sup_{x\in D} 
    \left| f(x) - \ClippedRealV{\theta}{l}uv(x) \right| 
    \leq 
    2L\!\left[ 
    \sup_{x=(x_1,x_2,\ldots,x_d)\in D} 
    \left( 
    \inf_{y=(y_1,y_2,\ldots,y_d)\in\cM}  \sum_{i=1}^d |x_i-y_i| 
    \right) 
    \right]
  \end{equation} 
  (cf.\ \cref{def:infnorm,def:rectclippedFFANN}).
  Next note that 
    the hypothesis that $\cM\subseteq D=\prod_{i=1}^d[p_i,q_i]$
  implies that for all 
    $z\in\cM$
  it holds that
  \begin{equation}
  \label{eq:ccnna1}
    \infnorm{z}
    \leq
    \max\{\infnorm p,\infnorm q\}
    .
  \end{equation}
  Therefore, we obtain that
  \begin{equation}
    \label{eq:ccnna6}
    \infnorm{\theta} 
    \leq
    \max \left\{1,L, \infnorm{p}, \infnorm q , 2\!\left[\sup_{z\in\cM} |f(z)|\right] \right\}
    \leq
    \max \left\{1,L, \infnorm{p}, \infnorm q , 2\!\left[\sup_{z\in D} |f(z)|\right] \right\}
    \!.
  \end{equation}
  In the next step we note that
	the fact that for all
	  $N\in\N$,
	  $r\in\R$,
	  $s\in[r,\infty)$,
	  $x\in[r,s]$
	there exists
	  $k\in\{0,1,\dots,N\}$
	such that
	  $\abs{x-(r+\tfrac kN(s-r))}\leq \tfrac{s-r}{2N}$
  ensures that for all
	$x=(x_1,x_2,\dots,x_d)\in D$
  there exists
	$y=(y_1,y_2,\dots,y_d)\in\mc M$
  such that
  \begin{equation}
    \label{eq:ccnnaa5}
    \sum_{i=1}^d\abs{x_i-y_i}
    \leq
    \frac{1}{2N}\!\left[\sum_{i=1}^d \abs{q_i-p_i}\right]\!
    .
  \end{equation}
  Combining
    this,
    \eqref{eq:ccnna3},
    \eqref{eq:ccnna7},
    and \eqref{eq:ccnna6}
  establishes \eqref{concrete_clipped_neural_net_approximation:claim}.
	The proof of \cref{cor:concrete_clipped_neural_net_approximation} is thus completed. 
\end{proof}
	
\subsection{Analysis of the generalization error}
\label{subsec:generalization_error}

\subsubsection{Hoeffding's concentration inequality}
\label{subsec:hoeffding}

\begingroup
\newcommand{\vmcF}{\mc F}
\begin{prop}
\label{prop:Hoeffding2}
Let $ ( \Omega, \vmcF, \P ) $ be a probability space, 
let $ N \in \N $, 
$\eps\in[0,\infty)$,
$ a_1, a_2, \dots, a_N \in \R $, 
$ b_1 \in [a_1,\infty) $, 
$ b_2 \in [a_2, \infty) $, 
$ \dots $, 
$ b_N \in [ a_N, \infty) $,
assume $\sum_{n=1}^N(b_n-a_n)^2\neq 0$,
and let $ X_n \colon \Omega \to [ a_n, b_n ] $, $ n \in \{ 1, 2, \dots, N \} $, 
be independent random variables. 
Then
\begin{equation}
\label{eq:Hoeffding_inequality3}
  \P\!\left(
    \frac{ 1 }{ N }
    \left|
        \sum_{ n = 1 }^N
        \big(
          X_n
          -
          \E[ X_n ]
        \big)
    \right|
    \geq 
    \varepsilon
  \right)
  \leq
  2 
  \exp\!\left(
    \frac{
      - 2 \varepsilon^2 
      N^2
    }{
      \sum_{ n = 1 }^N
      ( b_n - a_n )^2
    }
  \right)
  \!.
\end{equation}
\end{prop}
\endgroup

\subsubsection{Covering number estimates}

\begingroup
\newcommand{\vX}{E}
\newcommand{\vd}{\delta}
\begin{definition}[Covering number]
  \label{def:coveringnumber}
  Let $(\vX,\vd)$ be a metric space
  and let $r\in[0,\infty]$.
  Then we denote by $\CovNum{(\vX,\vd),r}\in\N_0\cup\{\infty\}$ 
  (we denote by $\CovNum{\vX,r}\in\N_0\cup\{\infty\}$)
  the extended real number given by
  \begin{equation}
    \CovNum{(\vX,\vd),r}
    =
    \inf\Bigl(
      \Bigl\{ 
        n \in \N_0 \colon
        \left(
          \Exists A\subseteq \vX \colon \bigl[
          (\card A\leq n)
          \land
          (
            \Forall x\in \vX\colon
            \Exists a\in A\colon
            \vd(a,x)\leq r
          )
          \bigr]
        \right)
      \Bigr\}
      \cup \{ \infty \}
    \Bigr)
    .
  \end{equation}
\end{definition}
\endgroup

\begingroup
\newcommand{\vX}{X}
\newcommand{\vd}{\delta}
\begin{prop}
  \label{prop:covnum}
  Let
    $(\vX,\norm\cdot)$ be a finite-dimensional Banach space,
  let
    $R,r\in(0,\infty)$, 
    $B=\{\theta\in \vX\colon \norm{\theta}\leq R\}$,
  and let
    $\vd\colon B\times B\to [0,\infty)$ satisfy for all 
      $\theta,\vartheta\in B$
    that
      $\vd(\theta,\vartheta)=\norm{\theta-\vartheta}$.
  Then
  \begin{equation}
    \CovNum{(B,\vd),r}
    \leq 
    \begin{cases}
      1 & \colon r\geq R \\
    \left[
      \frac{4R}r
    \right]^{\dim(\vX)} & \colon r<R
    \end{cases}
  \end{equation}
  (cf.\ Definition~\ref{def:coveringnumber}).
\end{prop}
\endgroup


\subsubsection{Measurability properties for suprema}

\begingroup
\newcommand{\vX}{E}
\newcommand{\vY}{\mathbf E}
\newcommand{\vmcF}{\mc F}
\begin{lemma}
  \label{lem:meas}
  Let $(\vX,\mathscr E)$ be a topological space,
  assume $\vX\neq\emptyset$,
  let $\vY\subseteq \vX$ be an at most countable set,
  assume that $\vY$ is dense in $\vX$,
  let $(\Omega,\vmcF)$ be a measurable space,
  let $f_x\colon\Omega\to\R$, $x\in \vX$, be $\vmcF$/$\mc B(\R)$-measurable functions,
  assume 
    for all 
      $\omega\in\Omega$ 
    that 
      $\vX\ni x\mapsto f_x(\omega)\in\R$ is a continuous function,
  and let $F\colon\Omega\to\R\cup\{\infty\}$ satisfy
    for all
      $\omega\in\Omega$
    that
      $F(\omega)=\sup_{x\in \vX} f_x(\omega)$.
  Then
  \begin{enumerate}[label=(\roman{*})]
    \item \label{it:meas1}
      it holds for all 
        $\omega\in\Omega$ 
      that 
        $F(\omega)=\sup_{x\in \vY} f_x(\omega)$ 
      and
    \item \label{it:meas2}
      it holds that 
        $F$ is an $\vmcF$/$\mc B(\R\cup\{\infty\})$-measurable function.
  \end{enumerate}
\end{lemma}
\begin{proof}[Proof of Lemma~\ref{lem:meas}]
  Note that
    the hypothesis that $\vY$ is dense in $\vX$ 
  implies that for all
    $g\in C(\vX,\R)$
  it holds that
  \begin{equation}
    \sup_{x\in \vX} g(x)
    =
    \sup_{x\in \vY} g(x)
    .
  \end{equation}
    This
    and the hypothesis that
      for all
        $\omega\in\Omega$
      it holds that
        $\vX\ni x\mapsto f_x(\omega)\in\R$ is a continuous function
  show that for all
    $\omega\in\Omega$
  it holds that
  \begin{equation}
    \label{eq:meas.1}
    F(\omega)
    =
    \sup_{x\in \vX}f_x(\omega)
    =
    \sup_{x\in \vY}f_x(\omega)
    .
  \end{equation}
  This establishes item~\ref{it:meas1}.
  Next note that \enum{
    item~\ref{it:meas1} ;
    the hypothesis that for all
      $x\in \vX$
    it holds that
      $f_x\colon\Omega\to\R$ is an $\vmcF$/$\mc B(\R)$-measurable function
  }[demonstrate] item~\ref{it:meas2}.
  The proof of Lemma~\ref{lem:meas} is thus completed.
\end{proof}
\endgroup

\begingroup
\newcommand{\ZZ}[1]{Z_{#1}}
\newcommand{\vX}{E}
\newcommand{\vd}{\delta}
\newcommand{\vmcF}{\mc F}
\begin{lemma}
  \label{lem:meas2}
  Let $(\vX,\vd)$ be a separable metric space,
  assume $\vX\neq\emptyset$,
  let $(\Omega,\vmcF,\P)$ be a probability space,
  let $L\in\R$,
  and let $\ZZ x\colon\Omega\to\R$, $x\in \vX$, be random variables which satisfy
    for all $x,y\in \vX$ that
      $\E[\abs{\ZZ x}]<\infty$
      and $\abs{\ZZ x-\ZZ y}\leq L\vd(x,y)$.
  Then
  \begin{enumerate}[label=(\roman{*})]
    \item \label{it:meas2.1}
      it holds for all 
        $x,y\in \vX$,
        $\eta\in\Omega$
      that 
        $\abs{(\ZZ x(\eta)-\E[\ZZ x])-(\ZZ y(\eta)-\E[\ZZ y])}\leq 2L\vd(x,y)$
      and
    \item \label{it:meas2.2}
      it holds that 
        $\Omega\ni\eta\mapsto\sup_{x\in \vX}\abs{\ZZ x(\eta)-\E[\ZZ x]}\in[0,\infty]$
      is an $\vmcF$/$\Borel([0,\infty])$-measurable function.
  \end{enumerate}
\end{lemma}
\begin{proof}[Proof of \cref{lem:meas2}]
  Note that \enum{
    the hypothesis that for all
      $x,y\in \vX$
    it holds that
      $\abs{\ZZ x -\ZZ y}\leq L\vd(x,y)$
  }[show] that for all 
    $x,y\in \vX$,
    $\eta\in\Omega$
  it holds that
  \begin{equation}
  \begin{split}
      &\abs{(\ZZ x(\eta)-\E[\ZZ x])-(\ZZ y(\eta)-\E[\ZZ y])}
      =
      \abs{(\ZZ x(\eta)-\ZZ y(\eta))+(\E[\ZZ y]-\E[\ZZ x])}
      \\&\leq
      \abs{\ZZ x(\eta)-\ZZ y(\eta)}+\abs{\E[\ZZ x]-\E[\ZZ y]}
      \leq
      L\vd(x,y)+\abs{\E[\ZZ x]-\E[\ZZ y]}
      \\&=
      L\vd(x,y)+\abs{\E[\ZZ x-\ZZ y]}
      \leq
      L\vd(x,y)+\E[\abs{\ZZ x-\ZZ y}]
      \leq
      L\vd(x,y)+L\vd(x,y)
      =
      2L\vd(x,y)
      .
  \end{split}  
  \end{equation}
  This proves item~\ref{it:meas2.1}.
  Next observe that
    item~\ref{it:meas2.1} 
  implies that
    for all $\eta\in\Omega$ it holds that
      $\vX\ni x\mapsto \abs{\ZZ x(\eta)-\E[\ZZ x]}\in\R$ is a continuous function.
  Combining \enum{
    this ;
    the hypothesis that $\vX$ is separable
  } with
    \cref{lem:meas}
  establishes
    item~\ref{it:meas2.2}.
  The proof of \cref{lem:meas2} is thus completed.
\end{proof}

\endgroup

\subsubsection{Concentration inequalities for random fields}

\begingroup
\newcommand{\ZZ}[1]{Z_{#1}}
\newcommand{\vX}{E}
\newcommand{\vd}{\delta}
\newcommand{\vmcF}{\mc F}
\begin{lemma}
  \label{lem:cov0}
  Let $(\vX,\vd)$ be a separable metric space,
  let 
    $\eps,L\in\R$,
    $N\in\N$,
    $z_1,z_2,\dots,z_N\in \vX$
    satisfy
    $\vX\subseteq\bigcup_{i=1}^N\{x\in \vX\colon 2L\vd(x,z_i)\leq \eps\}$,
  let $(\Omega,\vmcF,\P)$ be a probability space,
  and let $\ZZ x\colon\Omega\to\R$, $x\in \vX$, be random variables 
    which satisfy for all
      $x,y\in \vX$
    that
      $\abs{\ZZ x-\ZZ y}\leq L\vd(x,y)$.
  Then
  \begin{equation}
    \P\!\left(\textstyle \sup_{x\in \vX}\abs{\ZZ x}\geq\eps\right)
    \leq
    \sum_{i=1}^N\P\!\left(\abs{\ZZ{z_i}}\geq\tfrac\eps2\right)
  \end{equation}
  (cf.\ Lemma~\ref{lem:meas}).
\end{lemma}
\begin{proof}[Proof of Lemma~\ref{lem:cov0}]
  Throughout this proof
  let $B_1,B_2,\dots,B_N\subseteq \vX$ satisfy
    for all $i\in\{1,2,\dots,N\}$ that
      $B_i=\{x\in \vX\colon 2L\vd(x,z_i)\leq \eps\}$.
  Observe that \enum{
    the triangle inequality ;
    the hypothesis that for all
      $x,y\in \vX$
    it holds that
      $\abs{\ZZ x-\ZZ y}\leq L\vd(x,y)$ ;
  }[show]
  that for all
    $i\in\{1,2,\dots,N\}$,
    $x\in B_i$
  it holds that
  \begin{equation}
  \begin{split}
    \abs{\ZZ x}
    =
    \abs{\ZZ x-\ZZ{z_i}+\ZZ{z_i}}
    \leq
    \abs{\ZZ x-\ZZ{z_i}} + \abs{\ZZ{z_i}}
    \leq
    L\vd(x,z_i) + \abs{\ZZ{z_i}}
    \leq \tfrac\eps2+\abs{\ZZ{z_i}}
    .
  \end{split}
  \end{equation}
  Combining
    this
  with 
    \cref{lem:meas}
  proves that for all
    $i\in\{1,2,\dots,N\}$
  it holds that
  \begin{equation}
  \label{eq:cov0.1}
    \P\!\left(\sup\nolimits_{x\in B_i}\abs{\ZZ x}\geq\eps\right)
    \leq
    \P\!\left(\tfrac\eps2+\abs{\ZZ{z_i}}\geq\eps\right)
    =
    \P\bigl(\abs{\ZZ{z_i}}\geq\tfrac\eps2\bigr)
    .
  \end{equation}
  \enum{
    This ;
    \cref{lem:meas}
  }[establish] that
  \begin{equation}
  \begin{split}
    \textstyle \P\!\left(\sup_{x\in \vX}\abs{\ZZ x}\geq\eps\right)
    &=
    \textstyle \P\!\left(\sup_{x\in\left(\bigcup_{i=1}^N B_i\right)}\abs{\ZZ x}\geq\eps\right)
    =
    \P\!\left(\textstyle\bigcup_{i=1}^N\left\{\sup_{x\in B_i}\abs{\ZZ x}\geq\eps\right\}\right)
    \\&\leq
    \sum_{i=1}^N\P\!\left(\sup\nolimits_{x\in B_i}\abs{\ZZ x}\geq\eps\right)
    \leq
    \sum_{i=1}^N\P\bigl(\abs{\ZZ{z_i}}\geq\tfrac\eps2\bigr)
    .
  \end{split}  
  \end{equation}
  This completes the proof of Lemma~\ref{lem:cov0}.
\end{proof}
\endgroup

\begingroup
\newcommand{\ZZ}[1]{Z_{#1}}
\newcommand{\vX}{E}
\newcommand{\vd}{\delta}
\newcommand{\vmcF}{\mc F}
\begin{lemma}
  \label{lem:cov02}
  Let $(\vX,\vd)$ be a separable metric space,
  assume $\vX\neq\emptyset$,
  let 
    $\eps,L\in(0,\infty)$,
  let $(\Omega,\vmcF,\P)$ be a probability space,
  and let $\ZZ x\colon\Omega\to\R$, $x\in \vX$, be random variables 
    which satisfy for all
      $x,y\in \vX$
    that \enum{
      $\abs{\ZZ x-\ZZ y}\leq L\vd(x,y)$
    }.
  Then
  \begin{equation}
  \label{eq:cov02con}
    \bigl[\CovNum{(\vX,\vd),\frac\eps{2L}}\bigr]^{-1}\textstyle \P\!\left(\sup_{x\in \vX}\abs{\ZZ x}\geq\eps\right)
    \leq
    \sup_{x\in \vX}\P\bigl(\abs{\ZZ x}\geq\tfrac\eps2\bigr)
    .
  \end{equation}
  (cf.\ \cref{def:coveringnumber,lem:meas}).
\end{lemma}
\begin{proof}[Proof of Lemma~\ref{lem:cov02}]
  Throughout this proof
    let $N\in\N\cup\{\infty\}$ satisfy $N=\CovNum{(\vX,\vd),\frac\eps{2L}}$,
    assume without loss of generality that $N<\infty$,
    and let $z_1,z_2,\dots,z_N\in \vX$ satisfy
      $\vX\subseteq \bigcup_{i=1}^N\{x\in \vX\colon \vd(x,z_i)\leq \frac\eps{2L}\}$
  (cf.\ \cref{def:coveringnumber}).
  Observe that \enum{
    \cref{lem:meas}
    ;
    \cref{lem:cov0}
  }[establish] that
  \begin{equation}
  \begin{split}
    \P\!\left(\textstyle \sup_{x\in \vX}\abs{\ZZ x}\geq\eps\right)
    \leq
    \sum_{i=1}^N\P\!\left(\abs{\ZZ{z_i}}\geq\tfrac\eps2\right)
    \leq
    N\bigl[\sup\nolimits_{x\in \vX}\P\big(\abs{\ZZ x}\geq\tfrac\eps2\big)\bigr]
    .
  \end{split}
  \end{equation}
  This completes the proof of Lemma~\ref{lem:cov02}.
\end{proof}
\endgroup

\begingroup
\newcommand{\ZZ}[1]{Z_{#1}}
\newcommand{\vX}{E}
\newcommand{\vd}{\delta}
\newcommand{\vmcF}{\mc F}
\begin{lemma}
  \label{lem:cov2}
  Let $(\vX,\vd)$ be a separable metric space,
  assume $\vX\neq\emptyset$,
  let 
    $\eps,L\in(0,\infty)$,
  let $(\Omega,\vmcF,\P)$ be a probability space,
  and let $\ZZ x\colon\Omega\to\R$, $x\in \vX$, be random variables 
    which satisfy for all
      $x,y\in \vX$
    that
      $\E[\abs{\ZZ x}]<\infty$
      and $\abs{\ZZ x-\ZZ y}\leq L\vd(x,y)$.
  Then
  \begin{equation}
  \label{eq:cov2con}
    \bigl[\CovNum{(\vX,\vd),\frac\eps{4L}}\bigr]^{-1}\textstyle \P\!\left(\sup_{x\in \vX}\abs{\ZZ x-\E[\ZZ x]}\geq\eps\right)
    \leq
    \sup_{x\in \vX}\P\bigl(\abs{\ZZ x-\E[\ZZ x]}\geq\tfrac\eps2\bigr)
    .
  \end{equation}
  (cf.\ \cref{def:coveringnumber,lem:meas2}).
\end{lemma}
\begin{proof}[Proof of Lemma~\ref{lem:cov2}]
  \newcommand{\YY}[1]{Y_{#1}}
  Throughout this proof
    let $\YY x\colon\Omega\to\R$, $x\in \vX$, satisfy for all
      $x\in \vX$,
      $\eta\in\Omega$
    that
      $\YY x(\eta)=\ZZ x(\eta)-\E[\ZZ x]$.
  Observe that \enum{
    \cref{lem:meas2}
  }[ensure] that for all
    $x,y\in \vX$
  it holds that
  \begin{equation}
    \abs{\YY x-\YY y}\leq 2L\vd(x,y).
  \end{equation}
  \enum{
    This ;
    \cref{lem:cov02} (with
      $(\vX,\vd)\is(\vX,\vd)$,
      $\eps\is\eps$,
      $L\is 2L$,
      $(\Omega,\vmcF,\P)\is(\Omega,\vmcF,\P)$,
      $(\ZZ x)_{x\in \vX}\is(\YY x)_{x\in \vX}$
    in the notation of \cref{lem:cov02})
  }[establish] \eqref{eq:cov2con}.
  The proof of Lemma~\ref{lem:cov2} is thus completed.
\end{proof}
\endgroup

\newcommand{\lemIcovthreeIvE}[2]{Y_{#1,#2}}
\begingroup
\newcommand{\vE}[2]{\lemIcovthreeIvE{#1}{#2}}
\newcommand{\ZZ}[1]{Z_{#1}}
\newcommand{\vX}{E}
\newcommand{\vd}{\delta}
\newcommand{\vmcF}{\mc F}
\begin{lemma}
  \label{lem:cov3}
  Let $(\vX,\vd)$ be a separable metric space,
  assume $\vX\neq\emptyset$,
  let 
    $M\in\N$,
    $\eps,L,D\in(0,\infty)$,
  let $(\Omega,\vmcF,\P)$ be a probability space,
  for every 
    $x\in \vX$ 
  let 
    $\vE{x}{1},\vE{x}{2},\dots,\vE{x}{M}\colon\Omega\to[0,D]$
  be independent random variables,
  assume
    for all
      $x,y\in \vX$,
      $m\in\{1,2,\dots,M\}$
    that
      $\abs{\vE{x}{m}-\vE{y}{m}}\leq L\vd(x,y)$,
  and let $\ZZ{x}\colon \Omega\to[0,\infty)$, $x\in \vX$, satisfy for all 
    $x\in\vX$ 
  that
  \begin{equation}
    \label{eq:cov3.defZZ}
    \ZZ{x}=\frac1M\Biggl[\sum_{m=1}^M \vE{x}{m}\Biggr].
  \end{equation}
  Then 
  \begin{enumerate}[label=(\roman *)]
    \item \label{it:cov3.1}
      it holds for all 
        $x\in \vX$
      that 
        $\E[\abs{\ZZ{x}}]\leq D<\infty$,
    \item \label{it:cov3.2}
      it holds that 
        $\Omega\ni\eta\mapsto\sup_{x\in \vX} \abs{\ZZ{x}(\eta)-\E[\ZZ{x}]}\in[0,\infty]$
        is an $\vmcF$/$\Borel([0,\infty])$-measurable function, and
    \item \label{it:cov3.3}
      it holds that
      \begin{equation}
        \label{eq:cov3con}
        \P\!\left(\sup\nolimits_{x\in \vX}\abs{\ZZ{x}-\E[\ZZ{x}]}\geq\eps\right)
        \leq
        2\CovNum{(\vX,\vd),\frac\eps{4L}}\exp\!\left(\frac{-\eps^2M}{2D^2}\right)
      \end{equation}
  \end{enumerate}
  (cf.\ Definition~\ref{def:coveringnumber}).
\end{lemma}
\begin{proof}[Proof of Lemma~\ref{lem:cov3}]
  First, observe that \enum{
    the triangle inequality
    ;
    the hypothesis that for all
      $x,y\in \vX$,
      $m\in\{1,2,\dots,M\}$
    it holds that 
      $\abs{\vE{x}{m}-\vE{y}{m}}\leq L\vd(x,y)$
  }[imply] that for all
    $x,y\in \vX$
  it holds that
  \begin{equation}
    \label{eq:cov3lipsch}
  \begin{split}
    \abs{\ZZ{x}-\ZZ{y}}
    &=
    \left\lvert \frac1M\!\left[\sum_{m=1}^M \vE{x}{m}\right] - \frac1M\!\left[\sum_{m=1}^M \vE{y}{m}\right] \right\lvert
    =
    \frac1M\!\left\lvert\sum_{m=1}^M \bigl(\vE{x}{m} - \vE{y}{m}\bigr)\right\lvert
    \\&\leq
    \frac1M\!\left[\sum_{m=1}^M \babs{\vE{x}{m} - \vE{y}{m}}\right]
    \leq
    L\vd(x,y)
    .
  \end{split}
  \end{equation}
  Next note that
    the hypothesis that 
      for all
        $x\in \vX$,
        $m\in\{1,2,\dots,M\}$,
        $\omega\in\Omega$
      it holds that
        $\abs{\vE{x}{m}(\omega)}\in [0,D]$
  ensures that for all
    $x\in \vX$
  it holds that
  \begin{equation}
    \label{eq:Zxint}
    \E\bigl[\abs{\ZZ{x}}\bigr]
    =
    \E\!\left[\frac1M\!\left[\sum_{m=1}^M \vE{x}{m}\right]\right]
    =
    \frac1M\!\left[\sum_{m=1}^M \E\bigl[\vE{x}{m}\bigr]\right]
    \leq
    D
    <
    \infty
    .
  \end{equation}
  This proves \cref{it:cov3.1}.
  Furthermore, note that \enum{
    \cref{it:cov3.1} ;
    \eqref{eq:cov3lipsch} ;
    \cref{lem:meas2}
  }[establish]
  \cref{it:cov3.2}.
  Next observe that \enum{
    \eqref{eq:cov3.defZZ}
  }[show] that for all
    $x\in \vX$
  it holds that
  \begin{equation}
    \abs{\ZZ{x}-\E[\ZZ{x}]}
    =
    \left\lvert \frac1M\!\left[\sum_{m=1}^M \vE{x}{m}\right] - \E\!\left[\frac1M\!\left[\sum_{m=1}^M \vE{x}{m}\right]\right] \right\rvert
    =
    \frac1M\left\lvert \sum_{m=1}^M \bigl(\vE{x}{m} - \E\bigl[\vE{x}{m}\bigr]\bigr) \right\rvert
    .
  \end{equation}
  Combining
    this
  with
    Proposition~\ref{prop:Hoeffding2}
      (with 
        $(\Omega,\vmcF,\P)\is (\Omega,\vmcF,\P)$,
        $N\is M$,
        $\eps\is\frac\eps2$,
        $(a_1,a_2,\dots,a_N)\is (0,0,\dots,0)$,
        $(b_1,b_2,\dots,b_N)\is (D,D,\dots,D)$,
        $(X_n)_{n\in\{1,2,\dots,N\}}\is (\vE{x}{m})_{m\in\{1,2,\dots,M\}}$
        for $x\in \vX$
        in the notation of \cref{prop:Hoeffding2})
  ensures that for all 
    $x\in \vX$
  it holds that
  \begin{equation}
    \P\bigl(\abs{\ZZ{x}-\E[\ZZ{x}]}\geq\tfrac\eps2\bigr)
    \leq
    2\exp\!\left(\frac{-2\bigl[\frac{\eps}2\bigr]^2M^2}{MD^2}\right)
    =
    2\exp\!\left(\frac{-\eps^2M}{2D^2}\right)
    \!.
  \end{equation}
  Combining
    this,
    \eqref{eq:cov3lipsch},
    and~\eqref{eq:Zxint}
  with
    Lemma~\ref{lem:cov2}
  establishes~\cref{it:cov3.3}.
  The proof of Lemma~\ref{lem:cov3} is thus completed.
\end{proof}
\endgroup

\subsubsection{Uniform estimates for the statistical learning error}

\begingroup
\newcommand{\vA}{E}
\newcommand{\vd}{\delta}
\newcommand{\vX}[2]{X_{#1,#2}}
\newcommand{\vmfE}[1]{\mf E_{#1}}
\newcommand{\vmcE}[1]{\mc E_{#1}}
\newcommand{\vmcF}{\mc F}
\begin{lemma}
  \label{lem:cov4}
  Let $(\vA,\vd)$ be a separable metric space,
  assume $\vA\neq\emptyset$,
  let 
    $M\in\N$,
    $\eps,L,D\in(0,\infty)$,
  let $(\Omega,\vmcF,\P)$ be a probability space,
  let $\vX xm\colon \Omega\to\R$, $x\in \vA$, $m\in\{1,2,\dots,M\}$,
    and $Y_m\colon \Omega\to\R$, $m\in\{1,2,\dots,M\}$,
    be functions,
  assume for all $x\in \vA$ that
    $(\vX{x}{m},Y_m)$, $m\in\{1,2,\dots,M\}$, are i.i.d.\ random variables,
  assume for all
    $x,y\in \vA$,
    $m\in\{1,2,\dots,M\}$
  that
    $\abs{\vX{x}{m}-\vX{y}{m}}\leq L\vd(x,y)$
    and $\lvert \vX{x}{m}-Y_m\rvert\leq D$,
  let $\vmfE{x}\colon \Omega\to[0,\infty)$, $x\in \vA$,
    satisfy for all
      $x\in \vA$
    that
    \begin{equation}
      \label{eq:cov4.defmfE}
      \vmfE{x}=\frac1M\!\left[\sum_{m=1}^M \abs{\vX{x}{m}-Y_m}^2\right]\!,
    \end{equation}
  and let $\vmcE{x}\in[0,\infty)$, $x\in \vA$, 
    satisfy for all
      $x\in \vA$
    that
      $\vmcE{x}=\E[\abs{\vX{x}{1}-Y_1}^2]$.
  Then
    $\Omega\ni\omega\mapsto \sup_{x\in \vA}\abs{\vmfE{x}(\omega)-\vmcE{x}}\in [0,\infty]$
    is an $\vmcF$/$\Borel([0,\infty])$-measurable function
  and
  \begin{equation}
    \label{eq:cov4con}
    \P\!\left(\sup\nolimits_{x\in \vA}\abs{\vmfE{x}-\vmcE{x}}\geq\eps\right)
    \leq
    2\CovNum{(\vA,\vd),\frac\eps{8LD}}\exp\!\left(\frac{-\eps^2M}{2D^4}\right)
  \end{equation}
  (cf.\ \cref{def:coveringnumber}).
\end{lemma}
\begin{proof}[Proof of Lemma~\ref{lem:cov4}]
\begingroup
\newcommand{\vE}[2]{\mathscr{E}_{#1,#2}}
  Throughout this proof let 
    $\vE{x}{m}\colon\Omega\to[0,D^2]$, $x\in \vA$, $m\in\{1,2,\dots,M\}$,
  satisfy for all
    $x\in \vA$,
    $m\in\{1,2,\dots,M\}$
  that
  \begin{equation}
    \vE{x}{m}
    =
    \lvert \vX{x}{m}-Y_m\rvert^2
    .
  \end{equation}
  %
  Observe that \enum{
    the fact that for all 
      $x_1,x_2,y\in\R$ 
    it holds that 
      $(x_1-y)^2-(x_2-y)^2=(x_1-x_2)((x_1-y)+(x_2-y))$ ;
    the hypothesis that
      for all
        $x\in \vA$,
        $m\in\{1,2,\dots,M\}$
      it holds that
        $\abs{\vX{x}{m}-Y_m}\leq D$ ;
    the hypothesis that
      for all
        $x,y\in \vA$,
        $m\in\{1,2,\dots,M\}$
      it holds that
        $\abs{\vX{x}{m}-\vX{y}{m}}\leq L\vd(x,y)$
  }[imply] that for all
    $x,y\in \vA$,
    $m\in\{1,2,\dots,M\}$
  it holds that
  \begin{equation}
    \label{eq:cov4lipsch}
  \begin{split}
    &\abs{\vE{x}{m}-\vE ym}
    =
    \babs{(\vX{x}{m}-Y_m)^2-(\vX{y}{m}-Y_m)^2}
    =
    \abs{\vX{x}{m}-\vX{y}{m}}\babs{(\vX{x}{m}-Y_m)+(\vX{y}{m}-Y_m)}
    \\&\leq
    \abs{\vX{x}{m}-\vX{y}{m}}\bigl(\abs{\vX{x}{m}-Y_m}+\abs{\vX{y}{m}-Y_m}\bigr)
    \leq
    2D\abs{\vX{x}{m}-\vX{y}{m}}
    \leq
    2LD\vd(x,y)
    .
  \end{split}
  \end{equation}
  In addition, note that 
    \eqref{eq:cov4.defmfE}
    and the hypothesis that
      for all 
        $x\in\vA$ 
      it holds that
        $(X_{x,m},Y_m)$, $m\in\{1,2,\dots,M\}$, are i.i.d.\ random variables
  show that for all
    $x\in \vA$
  it holds that
  \begin{equation}
  \label{eq:cov4.1}
    \E\bigl[\vmfE{x}\bigr]
    =
    \frac1M\!\left[\sum_{m=1}^M\E\bigl[\abs{\vX{x}{m}-Y_m}^2\bigr]\right]\!
    =
    \frac1M\!\left[\sum_{m=1}^M\E\bigl[\abs{\vX{x}{1}-Y_1}^2\bigr]\right]\!
    =
    \frac1M\!\left[\sum_{m=1}^M\vmcE{x}\right]\!
    =
    \vmcE{x}
    .
  \end{equation}
  Furthermore, observe that
    the hypothesis that for all
      $x\in \vA$
    it holds that
      $(\vX{x}{m},Y_m)$, $m\in\{1,2,\dots,M\}$,
    are i.i.d.\ random variables
  ensures that for all
    $x\in \vA$
  it holds that
    $\vE xm$, $m\in\{1,2,\dots,M\}$,
  are i.i.d.\ random variables.
  Combining \enum{
    this ;
    \eqref{eq:cov4lipsch} ;
    \eqref{eq:cov4.1} ;
  } with
    Lemma~\ref{lem:cov3}
    (with
      $(E,\delta)\is(\vA,\vd)$,
      $M\is M$,
      $\eps\is\eps$,
      $L\is 2LD$,
      $D\is D^2$,
      $(\Omega,\vmcF,\P)\is(\Omega,\vmcF,\P)$,
      $(\lemIcovthreeIvE xm)_{x\in E,\,m\in\{1,2,\dots,M\}}\is (\vE xm)_{x\in \vA,\,m\in\{1,2,\dots,M\}}$,
      $(Z_x)_{x\in E}=(\vmfE{x})_{x\in \vA}$
    in the notation of Lemma~\ref{lem:cov3})
  establishes~\eqref{eq:cov4con}.
  The proof of Lemma~\ref{lem:cov4} is thus completed.
\endgroup
\end{proof}
\endgroup

\begingroup
\newcommand{\vmcF}{\mc F}
\begin{lemma}
  \label{lem:cov5}
  Let 
    $d,\mf d,M\in\N$,
    $R,L,\mc R,\eps\in(0,\infty)$,
  let $D\subseteq\R^d$ be a compact set,
  let $(\Omega,\vmcF,\P)$ be a probability space,
  let $X_m\colon\Omega\to D$, $m\in\{1,2,\dots,M\}$,
    and $Y_m\colon \Omega\to\R$, $m\in\{1,2,\dots,M\}$,
    be functions,
  assume that $(X_m,Y_m)
      $, $m\in\{1,2,\dots,M\}$,
    are i.i.d.\ random variables,
  let $H=(H_\theta)_{\theta\in [-R,R]^{\mf d}}\colon [-R,R]^{\mf d}\to C(D,\R)$ satisfy
    for all
      $\theta,\vartheta\in [-R,R]^{\mf d}$,
      $x\in D$
    that
      $\abs{H_\theta(x)-H_\vartheta(x)}\leq L\infnorm{\theta-\vartheta}$,
  assume for all 
      $\theta\in [-R,R]^{\mf d}$,
      $m\in\{1,2,\dots,M\}$
    that
      $\abs{H_\theta(X_m)-Y_m}\leq\mc R$
      and $\E[\abs{Y_1}^2]<\infty$,
  let $\mc E\colon C(D,\R)\to[0,\infty)$ satisfy 
    for all
      $f\in C(D,\R)$
    that
      $\mc E(f)=\E[\abs{f(X_1)-Y_1}^2]$,
  and let $\mf E\colon [-R,R]^{\mf d}\times\Omega\to[0,\infty)$ satisfy
    for all
      $\theta\in [-R,R]^{\mf d}$,
      $\omega\in\Omega$
    that
    \begin{equation}
      \mf E(\theta,\omega)
      =
      \frac1M\!\left[\sum_{m=1}^M\abs{H_\theta(X_m(\omega))-Y_m(\omega)}^2\right]
    \end{equation}
  (cf.\ \cref{def:infnorm}).
  Then
    $\Omega\ni\omega\mapsto \sup_{\theta\in[-R,R]^{\mf d}}\abs{\mf E(\theta,\omega)-\mc E(H_\theta)}\in [0,\infty]$ is an $\vmcF$/$\Borel([0,\infty])$-measurable function and
  \begin{equation}
    \label{eq:cov5con}
    \begin{split}
    \P\bigl(\sup\nolimits_{\theta\in [-R,R]^{\mf d}}\abs{\mf E(\theta)-\mc E(H_\theta)}\geq\eps\bigr)
    &\leq
    2\max\biggl\{1,\biggl[
    \frac{32LR\mc R}{\eps}
    \biggr]^{\mf d}\biggr\}
    \exp\!\left(\frac{-\eps^2M}{2\mc R^4}\right).
    \end{split}
  \end{equation}
\end{lemma}
\begin{proof}[Proof of Lemma~\ref{lem:cov5}]
  Throughout this proof 
  let $B\subseteq\R^{\mf d}$ satisfy $B=[-R,R]^{\mf d}=\{\theta\in\R^{\mf d}\colon \infnorm\theta\leq R\}$
  and let
    $\delta\colon B\times B\to[0,\infty)$
      satisfy for all
        $\theta,\vartheta\in B$
      that
      \begin{equation}
        \delta(\theta,\vartheta)
        =
        \infnorm{\theta-\vartheta}
        .
      \end{equation}
  Observe that \enum{
    the hypothesis that $(X_m,Y_m)$, $m\in\{1,2,\dots,M\}$, are i.i.d.\ random variables
    ;
    the hypothesis that for all $\theta\in [-R,R]^{\mf d}$ it holds that $H_\theta$ is a continuous function
  }[imply] that for all 
    $\theta\in B$
  it holds that
    $(H_\theta(X_m),Y_m)$, $m\in\{1,2,\dots,M\}$, are i.i.d.\ random variables.
  Combining \enum{
    this
    ;
    the hypothesis that
    for all
      $\theta,\vartheta\in B$,
      $x\in D$
    it holds that
      $\abs{H_\theta(x)-H_\vartheta(x)}\leq L\infnorm{\theta-\vartheta}$
    ;
    the hypothesis that for all 
      $\theta\in B$,
      $m\in\{1,2,\dots,M\}$
    it holds that
      $\abs{H_\theta(X_m)-Y_m}\leq\mc R$
  } with Lemma~\ref{lem:cov4} 
    (with
      $(E,\delta)\is (B,\delta)$,
      $M\is M$,
      $\eps\is\eps$,
      $L\is L$,
      $D\is \mc R$,
      $(\Omega,\vmcF,\P)\is(\Omega,\vmcF,\P)$,
      $(X_{x,m})_{x\in E,\,m\in\{1,2,\dots,M\}}\is (H_\theta(X_m))_{\theta\in B,\,m\in\{1,2,\dots,M\}}$,
      $(Y_m)_{m\in\{1,2,\dots,M\}}\is(Y_m)_{m\in\{1,2,\dots,M\}}$,
      $(\mf E_x)_{x\in E}\is\bigl((\Omega\ni\omega\mapsto\mf E(\theta,\omega)\in[0,\infty))\bigr)_{\theta\in B}$,
      $(\mc E_x)_{x\in E}\is(\mc E(H_\theta))_{\theta\in B}$
   in the notation of Lemma~\ref{lem:cov4}) establishes that
     $\Omega\ni\omega\mapsto \sup_{\theta\in B}\abs{\mf E(\theta,\omega)-\mc E(H_\theta)}\in [0,\infty]$ is an $\vmcF$/$\Borel([0,\infty])$-measurable function and
   \begin{equation}
    \label{eq:cov5bnd1}
    \P\bigl(\sup\nolimits_{\theta\in B}\abs{\mf E(\theta)-\mc E(H_\theta)}\geq\eps\bigr)
    \leq
    2\CovNum{(B,\delta),\frac\eps{8L\mc R}}\exp\!\left(\frac{-\eps^2M}{2\mc R^4}\right)
  \end{equation}
  (cf.\ \cref{def:coveringnumber}).
  Moreover, note that
    \cref{prop:covnum}
      (with
        $X\is\R^{\mf d}$,
        $\norm\cdot\is(\R^{\mf d}\ni x\mapsto \infnorm x\in[0,\infty))$,
        $R\is R$,
        $r\is\tfrac{\eps}{8L\mc R}$,
        $B\is B$,
        $\delta\is \delta$
      in the notation of \cref{prop:covnum})
  demonstrates that
  \begin{equation}
    \CovNum{(B,\delta),\frac{\eps}{8L\mc R}}
    \leq
    \max\biggl\{1,
    \left[\frac{32LR\mc R}{\eps}\right]^{\mf d}\biggr\}
    .
  \end{equation}
    This
    and \eqref{eq:cov5bnd1}
  prove
    \eqref{eq:cov5con}.
   The proof of Lemma~\ref{lem:cov5} is thus completed.
\end{proof}
\endgroup

\begingroup
\newcommand{\vmcF}{\mc F}
\begin{lemma}
  \label{lem:cov6}
  Let 
    $\mf d,M,L\in\N$,
    $u\in\R$,
    $v\in(u,\infty)$,
    $R\in[1,\infty)$,
    $\eps,b\in(0,\infty)$,
    $l=(l_0,l_1,\dots,l_L)\in\N^{L+1}$
    satisfy 
      $l_L=1$ 
      and $\sum_{k=1}^Ll_k(l_{k-1}+1) \leq \mf d$,
  let $D\subseteq[-b,b]^{l_0}$ be a compact set,
  let $(\Omega,\vmcF,\P)$ be a probability space,
  let $X_m\colon\Omega\to D$, $m\in\{1,2,\dots,M\}$,
    and $Y_m\colon\Omega\to [u,v]$, $m\in\{1,2,\dots,M\}$,
    be functions,
  assume that
  $(X_m,Y_m)
     $, $m\in\{1,2,\dots,M\}$,
    are i.i.d.\ random variables,
  let $\mc E\colon C(D,\R)\to[0,\infty)$ satisfy 
    for all
      $f\in C(D,\R)$
    that
      $\mc E(f)=\E[\abs{f(X_1)-Y_1}^2]$,
  and let $\mf E\colon [-R,R]^{\mf d}\times\Omega\to[0,\infty)$ satisfy
    for all
      $\theta\in [-R,R]^{\mf d}$,
      $\omega\in\Omega$
    that
    \begin{equation}
      \mf E(\theta,\omega)
      =
      \frac1M\!\left[\sum_{m=1}^M\abs{\ClippedRealV{\theta}{l} uv(X_m(\omega))-Y_m(\omega)}^2\right]
    \end{equation}
  (cf.\ \cref{def:infnorm,def:rectclippedFFANN}).
  Then
    $\Omega\ni\omega\mapsto \sup\nolimits_{\theta\in [-R,R]^{\mf d}}\babs{\mf E(\theta,\omega)-\mc E\bigl(\ClippedRealV\theta{l}uv|_{D}\bigr)}\in[0,\infty]$ is an $\vmcF$/$\Borel([0,\infty])$-measurable function and
  \begin{equation}
    \label{eq:cov6con}
    \begin{split}
    &\P\!\left(\sup\nolimits_{\theta\in [-R,R]^{\mf d}}\babs{\mf E(\theta)-\mc E\bigl(\ClippedRealV\theta{l}uv|_{D}\bigr)}\geq\eps\right)
    \\&\leq
    2\max\biggl\{1,\biggl[\frac{32L\max\{1,b\}(\infnorm l+1)^LR^{L}(v-u)}{\eps}\biggr]^{\mf d}\biggr\}
    \exp\!\left(\frac{-\eps^2M}{2(v-u)^4}\right)
    \!.
    \end{split}
  \end{equation}
\end{lemma}
\begin{proof}[Proof of Lemma~\ref{lem:cov6}]
  Throughout this proof
    let $\mf L\in(0,\infty)$ satisfy
    \begin{equation}
      \mf L
      =
      L 
      \max\{ 1, b \}
      \, 
      ( \infnorm{ l } + 1 )^L 
      R^{ L - 1 }
      .
    \end{equation}
  Observe that
    \cref{cor:ClippedRealNNLipsch}
      (with
        $a\is -b$,
        $b\is b$,
        $u\is u$,
        $v\is v$,
        $d\is \mf d$,
        $L\is L$,
        $l\is l$
      in the notation of \cref{cor:ClippedRealNNLipsch})
    and the hypothesis that $D\subseteq[-b,b]^{l_0}$
  show that for all
    $\theta,\vartheta\in [-R,R]^{\mf d}$
  it holds that
  \begin{equation}
  \begin{split}
    \label{eq:cov6lipsch}
    \sup_{x\in D}\,\abs{\ClippedRealV{\theta}{l}uv(x)-\ClippedRealV\vartheta{l}uv(x)}
    &\leq
    \sup_{x\in[-b,b]^{l_0}}\abs{\ClippedRealV{\theta}{l}uv(x)-\ClippedRealV\vartheta{l}uv(x)}
    \\&\leq
    L 
    \max\{ 1, b \}
    \, 
    ( \infnorm{ l } + 1 )^L 
    \,
    ( \max\{ 1, \infnorm\theta,\infnorm\vartheta\} )^{ L - 1 }
    \infnorm{ \theta - \vartheta }
    \\&\leq
    L 
    \max\{ 1, b \}
    \, 
    ( \infnorm{ l } + 1 )^L 
    R^{ L - 1 }
    \infnorm{ \theta - \vartheta }
    =
    \mf L\infnorm{ \theta - \vartheta }
    .
  \end{split}
  \end{equation}
  Furthermore, observe that \enum{
    the fact that for all 
      $\theta\in\R^{\mf d}$,
      $x\in\R^{l_0}$
    it holds that
      $\ClippedRealV{\theta}{l}uv(x)\in[u,v]$ ;
    the hypothesis that for all 
      $m\in\{1,2,\dots,M\}$,
      $\omega\in\Omega$
    it holds that
      $Y_m(\omega)\in[u,v]$
  }[demonstrate] that for all
    $\theta\in [-R,R]^{\mf d}$,
    $m\in\{1,2,\dots,M\}$
  it holds that
  \begin{equation}
    \abs{\ClippedRealV\theta{l}uv(X_m)-Y_m}
    \leq 
    v-u
    .
  \end{equation}
  Combining
    this
    and~\eqref{eq:cov6lipsch}
  with
    Lemma~\ref{lem:cov5}
      (with
        $d\is l_0$,
        $\mf d\is\mf d$,
        $M\is M$,
        $R\is R$,
        $L\is \mf L$,
        $\mc R\is v-u$,
        $\eps\is\eps$,
        $D\is D$,
        $(\Omega,\vmcF,\P)\is(\Omega,\vmcF,\P)$,
        $(X_m)_{m\in\{1,2,\dots,M\}}\is(X_m)_{m\in\{1,2,\dots,M\}}$,
        $(Y_m)_{m\in\{1,2,\dots,M\}}\is((\Omega\ni\omega\mapsto Y_m(\omega)\in\R))_{m\in\{1,2,\dots,M\}}$,
        $H\is([-R,R]^{\mf d}\ni\theta\mapsto \ClippedRealV\theta{l}uv|_{D}\in C(D,\R))$,
        $\mc E\is\mc E$,
        $\mf E\is\mf E$
      in the notation of Lemma~\ref{lem:cov5})
  establishes that
    $\Omega\ni\omega\mapsto \sup\nolimits_{\theta\in [-R,R]^{\mf d}}\babs{\mf E(\theta,\omega)-\mc E\bigl(\ClippedRealV\theta{l}uv|_{D}\bigr)}\in[0,\infty]$ is an $\vmcF$/$\Borel([0,\infty])$-measurable function and
  \begin{equation}
    \P\!\left(\sup\nolimits_{\theta\in [-R,R]^{\mf d}}\babs{\mf E(\theta)-\mc E\bigl(\ClippedRealV\theta{l}uv|_{D}\bigr)}\geq\eps\right)
    \leq
    2\max\biggl\{1,\biggl[\frac{32\mf LR(v-u)}{\eps}\biggr]^{\mf d}\biggr\}
    \exp\!\left(\frac{-\eps^2M}{2(v-u)^4}\right)
    \!.
  \end{equation}
  The proof of Lemma~\ref{lem:cov6} is thus completed.
\end{proof}
\endgroup


\subsection{Analysis of the optimization error}
\label{subsec:optimization_error}

\subsubsection{Convergence rates for the minimum Monte Carlo method}

\begingroup
\newcommand{\vmcA}{\mathcal{F}}
\begin{lemma}
\label{lem:estimate_opt_error0}
Let $ ( \Omega, \vmcA, \P ) $ be a probability space, 
let $ \mathfrak{d}, N \in \N $, 
let $ \left\| \cdot \right\| \colon \R^{ \mathfrak{d} } \to [0,\infty) $ be a norm, 
let $ \mathfrak{H} \subseteq \R^{ \mathfrak{d} } $ be a set, 
let 
$ \vartheta \in \mathfrak{H} $, 
$ L,\eps \in (0,\infty) $, 
let 
$
  \mathfrak{E} \colon \mathfrak{H} \times \Omega \to \R
$
be a $(\Borel(\mf H)\otimes \vmcA)$/$\Borel(\R)$-measurable function, 
assume for all $ x, y \in \mathfrak{H} $, $ \omega \in \Omega $ that
$
  | \mathfrak{E}(x, \omega) - \mathfrak{E}( y, \omega ) | \leq L \lVert x - y \rVert
$,
and let 
$ \Theta_n \colon \Omega \to \mathfrak{H} $, 
$ n \in \{ 1, 2, \dots, N \} $, 
be i.i.d.\ random variables.
Then 
\begin{equation}
\begin{split}
&
  \P\bigl(
    \bigl[
    \min\nolimits_{ n \in \{ 1, 2, \dots, N \} } 
      \mathfrak{E}( \Theta_n )
    \bigr]
    -
    \mathfrak{E}( \vartheta )
    > \varepsilon
  \bigr)
\leq 
  \left[
    \P\bigl(
      \|
        \Theta_1
        -
        \vartheta
      \|
      >
      \tfrac{ \varepsilon }{ L }
    \bigr)
  \right]^N
\leq 
  \exp\bigl(
    - N
    \,
    \P\bigl(
      \|
        \Theta_1
        -
        \vartheta
      \|
      \leq
      \tfrac{ \varepsilon }{ L }
    \bigr)
  \bigr)
  .
\end{split}
\end{equation}
\end{lemma}
\begin{proof}[Proof of Lemma~\ref{lem:estimate_opt_error0}]
Note that \enum{
  the hypothesis that for all
    $x,y\in\mf H$,
    $\omega\in\Omega$
  it holds that
    $\abs{\mf E(x,\omega)-\mf E(y,\omega)}\leq L\norm{x-y}$
}[imply] that
\begin{equation}
\begin{split}
&
  \left[
    \min\nolimits_{ n \in \{ 1, 2, \dots, N \} } \mathfrak{E}( \Theta_n )
  \right]
  -
  \mathfrak{E}( \vartheta )
=
  \min\nolimits_{ n \in \{ 1, 2, \dots, N \} } 
  \left[ 
    \mathfrak{E}( \Theta_n )
    -
    \mathfrak{E}( \vartheta )
  \right]
\\ &
\leq 
  \min\nolimits_{ n \in \{ 1, 2, \dots, N \} } 
  \left|
    \mathfrak{E}( \Theta_n )
    -
    \mathfrak{E}( \vartheta )
  \right|
\leq
  \min\nolimits_{ n \in \{ 1, 2, \dots, N \} } 
  \bigl[
    L
    \|
      \Theta_n 
      -
      \vartheta
    \|
  \bigr]
\\ & =
  L
  \bigl[
    \min\nolimits_{ n \in \{ 1, 2, \dots, N \} } 
    \|
      \Theta_n 
      -
      \vartheta
    \|
  \bigr]
  .
\end{split}
\end{equation} 
The hypothesis that $\Theta_n$, $n\in\{1,2,\dots,N\}$, are i.i.d.\ random variables
and the fact that $ \forall \, x \in \R \colon 1 - x \leq e^{-x} $ 
hence show that
\begin{equation}
\begin{split}
&
  \P\Big(
    \bigl[
    \min\nolimits_{ n \in \{ 1, 2, \dots, N \} } 
    \mathfrak{E}( \Theta_n )
    \bigr]
    -
    \mathfrak{E}( \vartheta )
    > \varepsilon
  \Big)
\leq 
  \P\Big(
    L
    \bigl[
      \min\nolimits_{ n \in \{ 1, 2, \dots, N \} } 
      \|
        \Theta_n 
        -
        \vartheta
      \|
    \bigr]
    >
    \varepsilon
  \Big)
\\ & 
=
  \P\big(
      \min\nolimits_{ n \in \{ 1, 2, \dots, N \} } 
      \|
        \Theta_n 
        -
        \vartheta
      \|
    >
    \tfrac{ \varepsilon }{ L }
  \big)
=
  \left[
    \P\big(
      \|
        \Theta_1
        -
        \vartheta
      \|
      >
      \tfrac{ \varepsilon }{ L }
    \big)
  \right]^N
\\ & =
  \left[
    1 -
    \P\big(
      \|
        \Theta_1
        -
        \vartheta
      \|
      \leq
      \tfrac{ \varepsilon }{ L }
    \big)
  \right]^N
\leq 
  \exp\bigl(
    - N
    \,
    \P\bigl(
      \|
        \Theta_1
        -
        \vartheta
      \|
      \leq
      \tfrac{ \varepsilon }{ L }
    \bigr)
  \bigr)
  .
\end{split}
\end{equation}
The proof of Lemma~\ref{lem:estimate_opt_error0} is thus completed.
\end{proof}
\endgroup

\subsubsection{Continuous uniformly distributed samples}

\begin{lemma}
\label{lem:estimate_opt_error1}
Let $ ( \Omega, \mathcal{F}, \P ) $ be a probability space, 
let $ \mathfrak{d}, N \in \N $, $ a \in \R $, $ b \in (a,\infty) $, 
$ \vartheta \in [a,b]^{ \mathfrak{d} } $, 
$ L,\eps \in (0,\infty) $, 
let 
$
  \mathfrak{E} \colon [a,b]^{ \mathfrak{d} } \times \Omega \to \R
$
be a $(\Borel([a,b]^{\mf d})\otimes\mathcal{F})$/$\Borel(\R)$-measurable function, 
assume for all $ x, y \in [a,b]^{ \mathfrak{d} } $, $ \omega \in \Omega $ that
$
  | \mathfrak{E}(x, \omega) - \mathfrak{E}( y, \omega ) | \leq L \infnorm{x - y}
$,
let 
$ \Theta_n \colon \Omega \to [a,b]^{\mf d} $, 
$ n \in \{ 1, 2, \dots, N \} $, 
be i.i.d.\ random variables, 
and assume that
$ \Theta_1 $ is continuous uniformly distributed on $ [a,b]^{ \mathfrak{d} } $
(cf.\ \cref{def:infnorm}).
Then 
\begin{equation}
\label{eq:estimate_opt_error1_con}
\begin{split}
&
  \P\bigl(
    \bigl[
      \min\nolimits_{ n \in \{ 1, 2, \dots, N \} } 
        \mathfrak{E}( \Theta_n )
    \bigr]
    -
    \mathfrak{E}( \vartheta )
    > \varepsilon
  \bigr)
\leq 
  \exp\biggl(
    - N
    \min\biggl\{ 
      1 ,
        \frac{
          \varepsilon^{ \mathfrak{d} }
        }{
          L^{ \mathfrak{d} } ( b - a )^{ \mathfrak{d} }
        }
    \biggr\}
  \biggr)
  .
\end{split}
\end{equation}
\end{lemma}
\begin{proof}[Proof of Lemma~\ref{lem:estimate_opt_error1}]
Note that \enum{
  the hypothesis that $\Theta_1$ is continuous uniformly distributed on $[a,b]^{\mf d}$
}[ensure] that
\begin{equation}
\begin{split}
    \P\bigl(
      \infnorm{
        \Theta_1
        -
        \vartheta
      }
      \leq
      \tfrac{ \varepsilon }{ L }
    \bigr) 
&
  \geq 
    \P\bigl(
      \infnorm{
        \Theta_1
        -
        ( a, a, \dots, a )
      }
      \leq
      \tfrac{ \varepsilon }{ L }
    \bigr) 
  = 
    \P\bigl(
      \infnorm{
        \Theta_1
        -
        ( a, a, \dots, a )
      }
      \leq
      \min\{ \tfrac{ \varepsilon }{ L }, b - a \}
    \bigr) 
\\ & 
=
  \left[
    \frac{
      \min\{ \tfrac{ \varepsilon }{ L }, b - a \}
    }{
      \left( b - a \right)
    }
  \right]^{ \mathfrak{d} }
  =
  \min\!\left\{ 
    1 ,
    \left[ 
      \frac{
        \varepsilon
      }{
        L \, ( b - a )
      }
    \right]^{ \mathfrak{d} }
  \right\}
  \!.
\end{split}
\end{equation}
Combining
  this
with
  \cref{lem:estimate_opt_error0}
proves
  \eqref{eq:estimate_opt_error1_con}.
The proof of Lemma~\ref{lem:estimate_opt_error1} is thus completed. 
\end{proof}

\section{Overall error analysis}
\label{sec:overall_error_analysis}

In this section we combine the separate error analyses of the approximation error, the generalization error, and the optimization error in \cref{sec:error_sources} 
to obtain an overall analysis (cf.~\cref{prop:error_decomposition} below). 
We note that, e.g., \cite[Lemma~2.4]{kolmogorov2018solving} ensures that the integral appearing on the left-hand side of
\eqref{error_decomposition:claim} in \cref{prop:error_decomposition} and
subsequent results (cf.\ \eqref{unspecific_architecture_no_noise:claim} in \cref{cor:unspecific_architecture_no_noise},
\eqref{generic_constants:claim} in \cref{cor:generic_constants}, \eqref{l1_norm:claim} in \cref{cor:l1_norm},
and \eqref{lp_error:claim} in \cref{cor:lp_error})
is indeed measurable.
In \cref{lem:decomposition} below we present the well-known bias-variance decomposition result. 
To formulate this bias-variance decomposition lemma 
we observe that for 
every probability space $ ( \Omega, \mathcal{F}, \P ) $, 
every measurable space $ ( S, \mathcal{S} ) $, 
every random variable $ X \colon \Omega \to S $, 
and every $ A \in \mathcal{S} $
it holds that
$
\P_X( A ) 
=
\P( X \in A ) 
$. 
Moreover, note that for every probability space $ ( \Omega, \mathcal{F}, \P ) $, 
every measurable space $ ( S, \mathcal{S} ) $, 
every random variable $ X \colon \Omega \to S $, 
and every $ \mathcal{S} $/$ \mathcal{B}( \R ) $-measurable 
function $ f \colon S \to \R $ it holds that
$ 
\int_{ S } | f |^2 \, d\P_X
=
\int_{ S } | f( x ) |^2 \, \P_X( dx )
=
\int_{ \Omega } | f( X( \omega ) ) |^2 \, \P( d\omega )
=
\int_{ \Omega } | f( X ) |^2 \, d\P
=
\E\big[ | f(X) |^2 \big]
$.
A result related to \cref{lem:decomposition,lem:error_decomposition} can, e.g., be found in Berner et al.~\cite[Lemma 2.8]{BernerGrohsJentzen2018}.

\subsection{Bias-variance decomposition}

\begin{lemma}[Bias-variance decomposition]
\label{lem:decomposition}
Let $ ( \Omega, \mathcal{F}, \P ) $ be a probability space, 
let $ ( S, \mathcal{S} ) $ be a measurable space, 
let $ X \colon \Omega \to S $
and $ Y \colon \Omega \to \R $
be random variables with $ \E[ | Y |^2 ] < \infty $, 
and 
let $ \mathcal{E} \colon \mathcal{L}^2( \P_X ; \R ) \to [0,\infty) $ 
satisfy 
for all $ f \in \mathcal{L}^2( \P_X ; \R ) $ that
$
  \mathcal{E}( f )
  =
  \E\!\left[ 
    | f( X ) - Y |^2
  \right]
$. 
Then 
\begin{enumerate}[label=(\roman{*})]
\item 
\label{item:dec_i}
it holds for all $ f \in \mathcal{L}^2( \P_X ; \R ) $ that
\begin{equation}
  \mathcal{E}( f ) 
  =
  \E\!\left[
    \left|
      f( X )
      -
      \E\!\left[ Y | X \right]
    \right|^2
  \right]
  +
  \E\bigl[ 
    \left|
      Y
      -
      \E\!\left[ Y | X \right]
    \right|^2
  \bigr]
  ,
\end{equation}
\item 
\label{item:dec_ii}
it holds for all $ f, g \in \mathcal{L}^2( \P_X ; \R ) $ that
\begin{equation}
  \mathcal{E}( f ) - \mathcal{E}( g ) 
  =
  \E\!\left[ 
    \left| f( X ) - \E[ Y | X ] \right|^2
  \right]
  -
  \E\bigl[ 
    \left| g( X ) - \E[ Y | X ] \right|^2
  \bigr]
  ,
\end{equation}
and 
\item 
\label{item:dec_iii}
it holds for all $ f, g \in \mathcal{L}^2( \P_X ; \R ) $ that
\begin{equation}
  \E\!\left[ 
    \left| f( X ) - \E[ Y | X ] \right|^2
  \right]
=
  \E\bigl[ 
    \left| g( X ) - \E[ Y | X ] \right|^2
  \bigr]
  +
  \bigl(
    \mathcal{E}( f ) - \mathcal{E}( g ) 
  \bigr)
  .
\end{equation}
\end{enumerate}
\end{lemma}
\begin{proof}[Proof of Lemma~\ref{lem:decomposition}]
First, observe that
  the hypothesis that for all
    $f\in\mc L^2(\P_X;\R)$
  it holds that
    $\mathcal{E}( f ) 
    =
      \E[ 
        \left| f( X ) - Y \right|^2
      ]$
shows that for all $ f \in \mathcal{L}^2( \P_X ; \R ) $ 
it holds that
\begin{equation}
\label{eq:lem_dec1}
\begin{split}
  \mathcal{E}( f ) 
& =
  \E\!\left[ 
    \left| f( X ) - Y \right|^2
  \right]
  =
  \E\!\left[ 
    \left| ( f( X ) - \E[ Y | X ] ) + ( \E[ Y | X ] - Y ) \right|^2
  \right]
\\ &  
=
  \E\!\left[ 
    \left| f( X ) - \E[ Y | X ] \right|^2
  \right]
  +
  2 
  \,
  \E\big[ 
    \big( f( X ) - \E[ Y | X ] \big) 
    \big( \E[ Y | X ] - Y \big) 
  \big]
  +
  \E\!\left[
    \left| \E[ Y | X ] - Y \right|^2
  \right]
\\ &  
=
  \E\!\left[ 
    \left| f( X ) - \E[ Y | X ] \right|^2
  \right]
  +
  2 
  \,
  \E\Big[ 
    \E\big[
      \big( f( X ) - \E[ Y | X ] \big) 
      \big( \E[ Y | X ] - Y \big) 
      \big|
      X
    \big]
  \Big]
  +
  \E\!\left[
    \left| \E[ Y | X ] - Y \right|^2
  \right]
\\ &  
=
  \E\!\left[ 
    \left| f( X ) - \E[ Y | X ] \right|^2
  \right]
  +
  2 
  \,
  \E\Big[ 
    \big( f( X ) - \E[ Y | X ] \big) 
    \E\big[
      \big( \E[ Y | X ] - Y \big) 
      \big|
      X
    \big]
  \Big]
  +
  \E\!\left[
    \left| \E[ Y | X ] - Y \right|^2
  \right]
\\ &  
=
  \E\!\left[ 
    \left| f( X ) - \E[ Y | X ] \right|^2
  \right]
  +
  2 
  \,
  \E\big[ 
    \big( f( X ) - \E[ Y | X ] \big) 
    \big( \E[ Y | X ] - \E[ Y | X ] \big) 
  \big]
  +
  \E\!\left[
    \left| \E[ Y | X ] - Y \right|^2
  \right]
\\ &  
=
  \E\!\left[ 
    \left| f( X ) - \E[ Y | X ] \right|^2
  \right]
  +
  \E\bigl[
    \left| \E[ Y | X ] - Y \right|^2
  \bigr]
  .
\end{split}
\end{equation}
This implies that 
for all $ f, g \in \mathcal{L}^2( \P_X ; \R ) $ 
it holds that
\begin{equation}
\label{eq:lem_dec2}
  \mathcal{E}( f ) - \mathcal{E}( g ) 
=
  \E\!\left[ 
    \left| f( X ) - \E[ Y | X ] \right|^2
  \right]
  -
  \E\bigl[ 
    \left| g( X ) - \E[ Y | X ] \right|^2
  \bigr]
  .
\end{equation}
Hence, we obtain that
for all $ f, g \in \mathcal{L}^2( \P_X ; \R ) $ 
it holds that
\begin{equation}
\label{eq:lem_dec3}
  \E\!\left[ 
    \left| f( X ) - \E[ Y | X ] \right|^2
  \right]
=
  \E\!\left[ 
    \left| g( X ) - \E[ Y | X ] \right|^2
  \right]
  +
  \mathcal{E}( f ) - \mathcal{E}( g ) 
  .
\end{equation}
Combining this
with 
\eqref{eq:lem_dec1} and \eqref{eq:lem_dec2} establishes 
\cref{item:dec_i,item:dec_ii,item:dec_iii}.
The proof of Lemma~\ref{lem:decomposition} is thus completed.
\end{proof}

\subsection{Overall error decomposition} \label{subsec:overall_analysis}

\begin{lemma}
\label{lem:error_decomposition}
Let $ ( \Omega, \mathcal{F}, \P ) $ be a probability space, 
let $ d, M \in \N $, 
let $ D \subseteq \R^d $ be a compact set, 
let 
$X_m
\colon\Omega\to D$, $m\in\{1,2,\dots,M\}$,
and $Y_m\colon\Omega\to \R$, $m\in\{1,2,\dots,M\}$,
be functions,
assume that $ 
  (X_m, Y_m ) 
$, 
$ m \in \{ 1, 2, \dots, M \} 
$,
are i.i.d.\ random variables,
assume $ \E[ | Y_1 |^2 ] < \infty $, 
let $ \mathcal{E} \colon C( D, \R ) \to [0,\infty) $ 
satisfy 
for all $ f \in C( D, \R ) $ that
$
  \mathcal{E}( f )
  =
  \E\!\left[ 
    | f( X_1 ) - Y_1 |^2
  \right]
$,
and let 
$ \mathfrak{E} \colon C( D, \R ) \times \Omega \to [0,\infty) $
satisfy 
for all $ f \in C( D, \R ) $, $ \omega \in \Omega $ that
\begin{equation}
  \mathfrak{E}( f, \omega )
  =
  \frac{ 1 }{ M }
  \Biggl[
    \sum\limits_{ m = 1 }^M
    | f( X_m( \omega ) ) - Y_m( \omega ) |^2
  \Biggr]
  .
\end{equation}
Then it holds for all $ f, \phi \in C(D,\R) $ that
\begin{equation}
\begin{split}
&
  \E\!\left[
    \left|
      f( X_1 )
      -
      \E\!\left[ Y_1 | X_1 \right]
    \right|^2
  \right]
  =
  \E\!\left[
    \left|
      \phi( X_1 )
      -
      \E\!\left[ Y_1 | X_1 \right]
    \right|^2
  \right]
  +
  \mathcal{E}( f )
  - 
  \mathcal{E}( \phi )
\\ &
  \leq
  \E\!\left[
    \left|
      \phi( X_1 )
      -
      \E\!\left[ Y_1 | X_1 \right]
    \right|^2
  \right]
  +
  \big[
    \mathfrak{E}( f )
    -
    \mathfrak{E}( \phi )
  \big]
  +
  2
  \biggl[
    \max_{ 
      v
      \in 
      \{f,\phi\}
    }
    |
      \mathfrak{E}( v )
      - 
      \mathcal{E}( v )
    |
  \biggr]
  .
\end{split}
\end{equation}
\end{lemma}
\begin{proof}[Proof of Lemma~\ref{lem:error_decomposition}] 
Note that Lemma~\ref{lem:decomposition} ensures 
that for all $ f, \phi \in C( D, \R ) $ it holds that
\begin{equation}
\begin{split}
&
  \E\!\left[
    \left|
      f( X_1 )
      -
      \E\!\left[ Y_1 | X_1 \right]
    \right|^2
  \right]
\\ & 
  =
  \E\!\left[
    \left|
      \phi( X_1 )
      -
      \E\!\left[ Y_1 | X_1 \right]
    \right|^2
  \right]
  +
  \mathcal{E}( f )
  - 
  \mathcal{E}( \phi )
\\ & =
  \E\!\left[
    \left|
      \phi( X_1 )
      -
      \E\!\left[ Y_1 | X_1 \right]
    \right|^2
  \right]
  +
  \mathcal{E}( f )
  -
  \mathfrak{E}( f )
  +
  \mathfrak{E}( f )
  -
  \mathfrak{E}( \phi )
  +
  \mathfrak{E}( \phi )
  - 
  \mathcal{E}( \phi )
\\ & =
  \E\!\left[
    \left|
      \phi( X_1 )
      -
      \E\!\left[ Y_1 | X_1 \right]
    \right|^2
  \right]
  +
  \big[ 
    \big(
      \mathcal{E}( f )
      -
      \mathfrak{E}( f )
    \big)
    +
    \big(
      \mathfrak{E}( \phi )
      - 
      \mathcal{E}( \phi )
    \big)
  \big]
  +
  \bigl[
    \mathfrak{E}( f )
    -
    \mathfrak{E}( \phi )
  \bigr]
\\ & \leq 
  \E\!\left[
    \left|
      \phi( X_1 )
      -
      \E\!\left[ Y_1 | X_1 \right]
    \right|^2
  \right]
  +
  \Biggl[
    \sum_{ 
      v
      \in 
      \left\{ 
        f, \phi
      \right\}
    }
    |
      \mathfrak{E}( v )
      - 
      \mathcal{E}( v )
    |
  \Biggr]
  +
  \big[
    \mathfrak{E}( f )
    -
    \mathfrak{E}( \phi )
  \big]
\\ & \leq 
  \E\!\left[
    \left|
      \phi( X_1 )
      -
      \E\!\left[ Y_1 | X_1 \right]
    \right|^2
  \right]
  +
  2
  \biggl[
    \max_{ 
      v
      \in 
      \left\{ 
        f, \phi
      \right\}
    }
    |
      \mathfrak{E}( v )
      - 
      \mathcal{E}( v )
    |
  \biggr]
  +
  \big[
    \mathfrak{E}( f )
    -
    \mathfrak{E}( \phi )
  \big]
  .
\end{split}
\end{equation}
The proof of Lemma~\ref{lem:error_decomposition} 
is thus completed.
\end{proof}

	\begin{lemma} \label{lem:error_decomposition_parametrization}
		Let $ ( \Omega, \mathcal{F}, \P ) $ be a probability space, 
		let $ d, \mathfrak{d}, M \in \N $, 
		let $ D \subseteq \R^d $ be a compact set, 
		let $ B \subseteq \R^{\mathfrak{d}} $ be a set, 
		let $ H = (H_\theta)_{\theta\in B} \colon B \to C(D,\R) $ be a function, 
    let 
    $X_m
    \colon\Omega\to D$, $m\in\{1,2,\dots,M\}$,
    and $Y_m\colon\Omega\to \R$, $m\in\{1,2,\dots,M\}$,
      be functions,
    assume that $ 
      (X_m, Y_m ) 
    $, 
    $ m \in \{ 1, 2, \dots, M \} 
    $,
    are i.i.d.\ random variables,
    assume $ \E[ | Y_1 |^2 ] < \infty $, 
    let $\varphi\colon D\to\R$ 	be a $\Borel(D)$/$\Borel(\R)$-measurable function,  
		assume that it holds $\P$-a.s.~that 
		$\varphi(X_1) = \E\!\left[ Y_1 | X_1 \right]$,
		let $ \mathcal{E} \colon C(D,\R) \to [0,\infty) $ satisfy for all 
		$ f \in C( D, \R ) $ 
		that 
		$ \mathcal{E}(f) = \E[ | f(X_1) - Y_1 |^2 ] $, 
		and let 
		$ \mathfrak{E} \colon B \times \Omega \to [0,\infty) $ 
		satisfy for all 
		$ \theta \in B $, 
		$ \omega \in \Omega $ 
		that 
		\begin{equation} 
		\mathfrak{E}( \theta, \omega ) 
		= 
		\frac{1}{M} \Biggl[ 
		\sum_{m = 1}^M | H_{\theta}(X_m(\omega)) - Y_m(\omega) |^2 
    \Biggr]
    .
		\end{equation}  
		Then it holds for all 
		$ \theta,\vartheta \in B $ 
		that 
		\begin{equation} \label{error_decomposition_parametrization:claim}
		\begin{split}
		& 
		\int_D \left| H_{\theta}(x) - \varphi(x) \right|^2 \,\P_{X_1}(dx)
		=
		\int_D \left| H_{\vartheta}(x) - \varphi(x) \right|^2 \,\P_{X_1}(dx) 
		+ 
		\mc E(H_{\theta}) - \mc E(H_{\vartheta}) 
		\\
		& 
		\leq 
		\int_D \left| H_{\vartheta}(x) - \varphi(x) \right|^2 \,\P_{X_1}(dx) 
    + 
    \bigl[		
      \mathfrak{E}(\theta) 
  		- 
      \mathfrak{E}(\vartheta) 
    \bigr]
    + 
		2 \biggl[ \sup_{\eta\in B} 
		\left| \mathfrak{E}(\eta) - \mathcal{E}(H_{\eta}) \right| \biggr].  
		\end{split}
		\end{equation} 
	\end{lemma} 
  \begin{proof}[Proof of Lemma~\ref{lem:error_decomposition_parametrization}]
    First, observe that 
      \cref{lem:error_decomposition} (with 
        $(\Omega,\mathcal{F},\P)\is(\Omega,\mathcal{F},\P)$,
        $d\is d$, 
        $M\is M$, 
        $D\is D$,  
        $(X_m)_{m\in\{1,2,\dots,M\}}\is (X_m)_{m\in\{1,2,\dots,M\}}$, 
        $(Y_m)_{m\in\{1,2,\dots,M\}}\is (Y_m)_{m\in\{1,2,\dots,M\}}$, 
        $\mc{E}\is \mc{E}$, 
        $\mf{E} \is  \bigl(C(D,\R)\times\Omega\ni (f,\omega)\mapsto \frac{1}{M}\bigl[\sum_{m=1}^{M} |f(X_m(\omega))-Y_m(\omega)|^2\bigr] \in [0,\infty)\bigr)$
      in the notation of \cref{lem:error_decomposition})
    shows that for all
      $\theta,\vartheta\in B$
    it holds that
    \begin{equation}
      \label{eq:edp.1}
      \begin{split}
      &
      \E\!\left[
        \left|
          H_\theta( X_1 )
          -
          \E\!\left[ Y_1 | X_1 \right]
        \right|^2
      \right]
      =
      \E\!\left[
        \left|
          H_\vartheta( X_1 )
          -
          \E\!\left[ Y_1 | X_1 \right]
        \right|^2
      \right]
      +
      \mathcal{E}( H_\theta )
      - 
      \mathcal{E}( H_\vartheta )
      \\ &
      \leq
      \E\!\left[
        \left|
          H_\vartheta( X_1 )
          -
          \E\!\left[ Y_1 | X_1 \right]
        \right|^2
      \right]
      +
      \big[
        \mathfrak{E}( \theta )
        -
        \mathfrak{E}( \vartheta )
      \big]
      +
      2
      \biggl[
        \max_{ 
          \eta
          \in 
          \{\theta,\vartheta\}
        }
        |
          \mathfrak{E}( \eta )
          - 
          \mathcal{E}( H_\eta )
        |
      \biggr]
      \\ &
      \leq
      \E\!\left[
        \left|
          H_\vartheta( X_1 )
          -
          \E\!\left[ Y_1 | X_1 \right]
        \right|^2
      \right]
      +
      \big[
        \mathfrak{E}( \theta )
        -
        \mathfrak{E}( \vartheta )
      \big]
      +
      2
      \biggl[
        \sup_{ 
          \eta
          \in 
          B
        }
        |
          \mathfrak{E}( \eta )
          - 
          \mathcal{E}( H_\eta )
        |
      \biggr]
      .
      \end{split}
    \end{equation}
		In addition, note that the hypothesis that it holds $\P$-a.s.~that 
		$\varphi(X_1) = \E\!\left[ Y_1 | X_1 \right]$		
		ensures that for all 
		$\eta\in B$ 
		it holds that 
		\begin{equation}
      \E\bigl[ \left| H_{\eta}(X_1) - \E\!\left[ Y_1 | X_1 \right] \right|^2 \bigr]
	  	= 
		  \E\!\left[ \left| H_{\eta}(X_1) - \varphi(X_1) \right|^2 \right] 
		  =
      \int_D \left| H_{\eta}(x) - \varphi(x) \right|^2 \,\P_{X_1}(dx)
      .
    \end{equation}
		Combining this with \eqref{eq:edp.1}
    establishes \eqref{error_decomposition_parametrization:claim}. 
    The proof of Lemma~\ref{lem:error_decomposition_parametrization} is thus completed. 
	\end{proof} 
	
\subsection{Analysis of the convergence speed}
\subsubsection{Convergence rates for convergence in probability} 

	\begin{lemma}\label{lem:existence_of_optimal_network_weights} 
	Let 
		$ ( \Omega, \mathcal{F}, \P ) $ 
	be a probability space, 
	let 
    $ u \in \R $, 
    $ v \in (u,\infty) $, 
    $ \mf{d},L \in \N $,
  let
    $l=(l_0,l_1,\dots,l_L)\in\N^{L+1}$
    satisfy 
      $l_L=1$
      and $\sum_{i=1}^L l_i(l_{i-1}+1)\leq\mf d$,
	let 
		$ B \subseteq \R^{\mf d} $ 
	be a non-empty compact set, 
  and let 
    $X\colon \Omega\to\R^{l_0}$
    and $Y\colon\Omega\to [u,v]$
	be random variables. Then 
    \begin{enumerate}[label=(\roman{*})]
      \item \label{existence_of_optimal_network_weights:bounded}
      it holds for all $\theta\in B$, $\omega\in\Omega$ that
      $\abs{\ClippedRealV{\theta}{ l}{u}{v}(X(\omega)) - Y(\omega)}^2\in[0,(v-u)^2]$,
			\item \label{existence_of_optimal_network_weights:continuity}
			it holds that 
      $ B \ni \theta \mapsto \E\!\left[|\ClippedRealV{\theta}{ l}{u}{v}(X) - Y|^2 \right] \in [0,\infty)$ is continuous,
       and 
			\item 
			\label{existence_of_optimal_network_weights:existence_statement}
			there exists 
				$\vartheta\in B$ 
			such that 
				$\E\!\left[|\ClippedRealV{\vartheta}{ l}{u}{v}(X) - Y|^2 \right] 
				= 
				\inf\limits_{\theta \in B}
				\E\!\left[|\ClippedRealV{\theta}{ l}{u}{v}(X) - Y|^2 \right]$
    \end{enumerate} 
    (cf.\ \cref{def:rectclippedFFANN}).
	\end{lemma} 
	
\begin{proof}[Proof of \cref{lem:existence_of_optimal_network_weights}] 
  First, note that \enum{
    the fact that for all 
      $\theta\in \R^{\mf d}$, 
      $x\in\R^{l_0}$
    it holds that
      $\ClippedRealV{\theta}{ l}{u}{v}(x)\in[u,v]$ ;
    the hypothesis that for all 
      $\omega\in\Omega$
    it holds that
      $Y(\omega)\in [u,v]$
  }[demonstrate] \cref{existence_of_optimal_network_weights:bounded}.
  Next observe that \cref{cor:ClippedRealNNLipsch} ensures that for all 
    $\omega\in\Omega$ 
  it holds that 
    $B\ni\theta\mapsto |\ClippedRealV{\theta}{\mf l}{u}{v}(X(\omega)) - Y(\omega)|^2\in [0,\infty)$ 
  is a continuous function. 
  Combining 
    this 
    and \cref{existence_of_optimal_network_weights:bounded} 
  with 
    Lebesgue's dominated convergence theorem
  establishes item~\ref{existence_of_optimal_network_weights:continuity}. 
  Furthermore, note that 
    item~\ref{existence_of_optimal_network_weights:continuity} 
    and the assumption that $B\subseteq\R^{\mf d}$ is a non-empty compact set
  prove  item~\ref{existence_of_optimal_network_weights:existence_statement}.  
  The proof of \cref{lem:existence_of_optimal_network_weights} is thus completed. 
\end{proof} 
  
\begingroup
\newcommand{\vN}{N}
\begin{theorem} \label{prop:error_decomposition}
Let $ ( \Omega, \mc{F}, \P ) $ be a probability space, 
let $ d, \mf{d}, K, M \in \N $, 
	$ \varepsilon\in (0,\infty) $, 
	$ L, u \in \R $, 
  $ v \in (u,\infty) $,
let $ D \subseteq \R^d $ be a compact set, 
assume $\card D\geq 2$,
let 
$X_m\colon\Omega\to D$, $m\in\{1,2,\dots,M\}$,
and $Y_m\colon\Omega\to [u,v]$, $m\in\{1,2,\dots,M\}$,
be functions,
assume that $ 
  (X_m, Y_m ) 
$, 
$ m \in \{ 1, 2, \dots, M \} 
$,
are i.i.d.\ random variables,
let $\delta\colon D\times D\to [0,\infty)$
  satisfy for all 
    $x=(x_1,x_2,\dots,x_d),y=(y_1,y_2,\dots,y_d)\in D$
  that
    $\delta(x,y)=\sum_{i=1}^d\abs{x_i-y_i}$,
let $ \varphi \colon D \to [u,v] $ satisfy $ \P $-a.s.\ that
	$\varphi( X_1 ) = \E[ Y_1 | X_1 ]$, 
assume for all 
	$x,y \in D$ 
that 
  $|\varphi(x) - \varphi(y)| \leq L\delta(x,y)$,
let $N\in\N\cap[\max\{2,\CovNum{(D,\delta),\frac{\eps}{4L}}\},\infty)$,
let $l\in\N\cap(N,\infty)$,
let $\mathfrak{l} = (\mf l_0, \mf l_1, \ldots, \mf l_l) \in \N^{l+1}$
satisfy for all 
$ i \in \N\cap[2,\vN] $,
$ j\in\N\cap[\vN,l)$
that 
\enum{
	$ \mf l_0 = d $;
  $ \mf l_1 \geq 2d\vN $;
	$ \mf l_i \geq 2\vN-2i+3 $ ;
	$ \mf l_j \geq 2 $ ;
	$ \mf l_l = 1 $;
  $\sum_{k=1}^{l} \mf l_k(\mf l_{k-1}+1)\leq \mathfrak{d}$ 
},
let $R\in [\max\{1,L,\sup_{z\in D} \infnorm{z}, 2[\sup_{z\in D} |\varphi(z)|]\},\infty)$,
let $ B\subseteq\R^{\mf d}$ satisfy $B = [-R,R]^{\mf{d}} $, 
let
	$ \mathfrak{E} \colon B \times \Omega \to [0,\infty) $
satisfy for all 
	$ \theta \in B $, 
	$ \omega \in \Omega $ 
that
	\begin{equation} \label{error_decomposition:empirical_risk}
	\mathfrak{E}( \theta, \omega )
	=
	\frac{ 1 }{ M }
	\Biggl[
	\sum\limits_{ m = 1 }^M
	| \ClippedRealV{\theta}{\mathfrak{l}}{u}{v}( X_m( \omega ) ) - 		Y_m( \omega ) |^2
	\Biggr],
	\end{equation}
let 
	$\Theta_k\colon\Omega\to B$, 
	$k\in\{1,2,\ldots,K\}$, 
be i.i.d.~random variables, 
assume that 
	$\Theta_1$ 
is continuous uniformly distributed on 
	$B$, 
and let 
	$\Xi\colon\Omega\to B$ 
satisfy 
	$\Xi = \Theta_{\min\{k\in\{1,2,\ldots,K\}\colon\mathfrak{E}(\Theta_k) = \min_{l\in\{1,2,\ldots,K\}} \mathfrak{E}(\Theta_l) \}} 
	$ (cf.\ \cref{def:coveringnumber,def:infnorm,def:rectclippedFFANN}).
Then
	\begin{multline} \label{error_decomposition:claim}
		\P\!\left(
		\int_D 
		| 
		\ClippedRealV{ \Xi }{ \mathfrak{l} }{ u }{ v }( x )
		- 
		\varphi( x ) 
		|^2
		\,  
		\P_{ X_1 }( dx )
		> 
		\eps^2 
		\right)
		\leq
\exp\!\left(-K\min\!\left\{1, \frac{ \varepsilon^{2\mathfrak{d}}}{
  (16(v-u) l(\infnorm{\mf l}+1)^{l}R^{l+1})^{\mf d}
  } \right\}\right)
\\+
2\exp\!\left(
  \mf d\ln\!\left(
    \max\biggl\{1,\frac{128l(\infnorm{\mf l}+1)^lR^{l+1}(v-u)}{\eps^2}\biggr\}
  \right)
  -
  \frac{\eps^4M}{32(v-u)^4}
\right)
		\!.
		\end{multline}
\end{theorem}
\begin{proof}[Proof of \cref{prop:error_decomposition}] 
	Throughout this proof
    let $\cM\subseteq D$ satisfy 
    $\card\cM=\max\{2,\CovNum{(D,\delta),\frac{\eps}{4L}}\}$ and
    \begin{equation}
      \label{eq:ed.defM}
      4L\biggl[\sup_{x\in D}\biggl(\inf_{y\in\cM}\delta(x,y)\biggr)\biggr]\leq\eps,
    \end{equation}
    let $b\in[0,\infty)$ satisfy $b=\sup_{z\in D}\infnorm{z}$,
    let 
		  $ \mathcal{E} \colon C( D, \R ) \to [0,\infty) $ 
    satisfy for all 
      $ f \in C( D, \R ) $ 
    that
      $
      \mathcal{E}( f )
      =
      \E\!\left[ 
      | f( X_1 ) - Y_1 |^2
      \right]
      $,
    and let 
      $ \vartheta \in B $ 
    satisfy 
      $\mathcal{E}( 	\ClippedRealV{ \vartheta }{ 	\mathfrak{l} }{u}{v}|_D )
      =
      \inf_{ \theta \in B }
      \mathcal{E}( 	\ClippedRealV{ \theta }{ \mathfrak{l} }{u}{v}|_D  )
    $
  (cf.\ \cref{lem:existence_of_optimal_network_weights}).  
  Observe that 
    the hypothesis that for all
      $x,y\in D$
    it holds that
      $\abs{\varphi(x)-\varphi(y)}\leq L\delta(x,y)$
  implies that
    $\varphi$ is a $\mc B(D)$/$\mc B([u,v])$-measurable function.
    \cref{lem:error_decomposition_parametrization} (with 
      $ (\Omega,\mathcal{F},\P) \is (\Omega,\mathcal{F},\P) $,  
      $ d \is d $, 
      $ \mf{d} \is \mf{d} $, 
      $ M \is M $, 
      $ D \is D $, 
      $ B \is B $, 
      $ H \is (B\ni \theta \mapsto 
        \ClippedRealV{\theta}{\mathfrak{l}}{u}{v}|_D
      \in C(D,\R)) $, 
      $(X_m)_{m\in\{1,2,\dots,M\}}\is (X_m)_{m\in\{1,2,\dots,M\}}$, 
      $(Y_m)_{m\in\{1,2,\dots,M\}}\is ((\Omega\ni\omega\mapsto Y_m(\omega)\in \R))_{m\in\{1,2,\dots,M\}}$,
      $ \varphi \is (D\ni x\mapsto \varphi(x)\in \R) $, 
      $ \cE \is \cE $, 
      $ \mf{E} \is \mf{E} $
    in the notation of \cref{lem:error_decomposition_parametrization}) 
  therefore ensures that for all 
		$\omega\in\Omega$ 
	it holds that 
  \begin{equation}
    \label{eq:ed.1}
		\begin{split}
    &
		\int_D 
		| 
		\ClippedRealV{ \Xi(\omega) }{ \mathfrak{l} }{u}{v}( x )
		- 
		\varphi( x ) 
		|^2
		\,
		\P_{ X_1 }( dx )
		\\ & \leq 
    \underbrace{
		\int_D 
		\abs{
		\ClippedRealV{ \vartheta }{ \mathfrak{l} }{u}{v}( x )
		-
		\varphi( x ) 
    }^2\,\P_{X_1}(dx)
    }_{
      \text{Approximation error}
    }
    +
    \underbrace{
		\bigl[
		\mathfrak{E}( \Xi(\omega), \omega ) 
		-
		\mathfrak{E}( \vartheta, \omega )
    \bigr]
    }_{
      \text{Optimization error}
    }
		\mathbin{+}
    \underbrace{
		2
		\biggl[
		\sup_{ 
			\theta
			\in 
      B	}
		|
		\mathfrak{E}( \theta, \omega )
		- 
		\mathcal{E}( 	\ClippedRealV{ \theta }{ \mathfrak{l} }{u}{v}|_D )
    |
		\biggr]
    }_{\text{Generalization error}}
    \!.
		\end{split}
	\end{equation}
  Next observe that
    the assumption that
      $N\geq\max\{2,\CovNum{(D,\delta),\frac{\eps}{4L}}\}=\card\cM$
  shows that for all
    $i\in\N\cap[2,\vN]$
  it holds that
    $l \geq \card\cM+1 $,
    $\mf l_1 \geq 2d\card\cM$
    and
    $\mf l_i \geq 2\card\cM-2i+3 $.
  The hypothesis that for all
    $x,y\in D$ 
  it holds that
    $\abs{\varphi(x)-\varphi(y)}\leq L\delta(x,y)$,
  the hypothesis that
    $ R \geq \max\{1,L,\sup_{z\in D} \infnorm{z}, 2[\sup_{z\in D} |\varphi(z)|]\} $,  
  \cref{cor:existence_of_clipped_neural_net_approximation_vectorized_description2} (with 
    $ d \is d $, 
    $ \mf{d} \is \mf{d}$, 
    $ \mf L\is l$,
    $ L\is L$,
    $ u\is u$,
    $ v\is v$,
    $ D\is D$,
    $ f\is \varphi$,
    $ \cM \is \cM $,
    $ l\is\mf l$
  in the notation of \cref{cor:existence_of_clipped_neural_net_approximation_vectorized_description2}),
  and \eqref{eq:ed.defM}
  hence ensure that there exists 
    $\eta\in B$ 
  which satisfies 
  \begin{equation} 
    \begin{split}
    \sup_{x\in D} |\ClippedRealV{\eta}{ \mathfrak{l} }{u}{v}(x) - \varphi(x)| 
    &\leq 
    2L 
    \Biggl[\sup_{x=(x_1,x_2,\ldots,x_d)\in D}  \left(
    \inf_{y=(y_1,y_2,\ldots,y_d)\in\cM} 
    \sum_{i=1}^d |x_i-y_i| 
    \right) 
    \Biggr] 
    \\&=
    2L 
    \biggl[\sup_{x\in D}  \left(
    \inf_{y\in\cM} 
    \delta(x,y)
    \right) 
    \biggr] 
    \leq 
    \frac{\varepsilon}{2}. 
    \end{split}
  \end{equation} 
    \cref{lem:error_decomposition_parametrization} (with 
      $ (\Omega,\mathcal{F},\P) \is (\Omega,\mathcal{F},\P) $, 
      $ d \is d $, 
      $ \mf{d} \is \mf{d}$, 
      $ M \is M $, 
      $ D \is D $, 
      $ B \is B $, 
      $ H \is (B\ni \theta \mapsto 
        \ClippedRealV{\theta}{\mathfrak{l}}{u}{v}|_D
      \in C(D,\R)) $, 
      $(X_m)_{m\in\{1,2,\dots,M\}}\is (X_m)_{m\in\{1,2,\dots,M\}}$, 
      $(Y_m)_{m\in\{1,2,\dots,M\}}\is ((\Omega\ni\omega\mapsto Y_m(\omega)\in \R))_{m\in\{1,2,\dots,M\}}$,
      $ \varphi \is (D\ni x\mapsto \varphi(x)\in \R) $, 
      $ \mc E \is \mc E $, 
      $ \mf E \is \mf E $
    in the notation of \cref{lem:error_decomposition_parametrization})
    and the assumption that       
      $\mathcal{E}( 	\ClippedRealV{ \vartheta }{ 	\mathfrak{l} }{u}{v}|_D )
      =
      \inf_{ \theta \in B }
      \mathcal{E}( 	\ClippedRealV{ \theta }{ \mathfrak{l} }{u}{v}|_D  )
    $
    therefore 
  prove that 
	\begin{equation} 
	\begin{split} 
	& \int_D 
		\abs{ \ClippedRealV{\vartheta}{\mathfrak{l}}{u}{v} ( x ) - \varphi( x ) }^2 \,\P_{ X_1 }(dx) 
		= 
		\int_D 
		\abs{ \ClippedRealV{\eta}{\mathfrak{l}}{u}{v} (x) - \varphi }^2 \,\P_{ X_1 }(dx)
		+ 
		\underbrace{\mc E(\ClippedRealV{\vartheta}{\mathfrak{l}}{u}{v}|_D ) - 	\mc E(\ClippedRealV{\eta}{\mathfrak{l}}{u}{v}|_D )}_{\leq 0}
		\\
		& \leq 
		\int_D 
		\abs{ \ClippedRealV{\eta}{ \mathfrak{l}}{u}{v}( x ) - \varphi( x ) }^2\,\P_{X_1}(dx) 
		\leq 
		\sup_{x\in D}\,\abs{\ClippedRealV{\eta}{ \mathfrak{l} }{u}{v}(x) - \varphi(x)}^2 
		\leq 
		\frac{\varepsilon^2}{4}. 
		\end{split}
  \end{equation} 
  Combining
    this
  with
    \eqref{eq:ed.1}
  shows that for all
    $\omega\in\Omega$
  it holds that
  \begin{equation}
		\int_D 
		| 
		\ClippedRealV{ \Xi(\omega) }{ \mathfrak{l} }{u}{v}( x )
		- 
		\varphi( x ) 
		|^2
		\,
		\P_{ X_1 }( dx )
		\leq 
		\frac{\eps^2}4
		+
		\bigl[
		\mathfrak{E}( \Xi(\omega), \omega ) 
		-
		\mathfrak{E}( \vartheta, \omega )
		\bigr]
		+
		2
		\biggl[
		\sup_{ 
			\theta
			\in 
			B	}
		|
		\mathfrak{E}( \theta, \omega )
		- 
		\mathcal{E}( 	\ClippedRealV{ \theta }{ \mathfrak{l} }{u}{v}|_D )
		|
    \biggr]
    .    
  \end{equation}
	Hence, we obtain that
    \begin{align} \label{error_decomposition:eq02}
    \nonumber 
    & 
		\P\!\left(
		\int_D 
		| 
		\ClippedRealV{ \Xi }{ \mathfrak{l} }{u}{v}( x )
		- 
		\varphi( x ) 
		|^2
		\,  
		\P_{ X_1 }( dx )
		> 
		\varepsilon^2 
		\right)
		\leq
		\P\!\left( 
		\bigl[	\mf{E}( \Xi ) - \mf{E}( \vartheta ) \bigr]
		+
		2
		\biggl[
		\sup_{ 
			\theta \in B
		}
		|
		\mf{E}( \theta )
		- 
		\cE( \ClippedRealV{\theta}{\mathfrak{l}}{u}{v}|_D )
		|
		\biggr]
		> 
		\frac{3\varepsilon^2}4
		\right)
		\\
		& \leq
		\P\biggl(
		\mf{E}( \Xi )
		- 
		\mf{E}( \vartheta )
		> \frac{ \varepsilon^2 }{ 4 }
		\biggr)
		+
		\P\biggl(
		\sup_{\theta\in B}
		|
		\mf{E}( \theta )
		- 
		\cE( \ClippedRealV{\theta}{\mathfrak{l}}{u}{v}|_D )
		|
		> \frac{ \varepsilon^2 }{ 4 }
    \biggr)
    .
		\end{align}
Next observe that \cref{cor:ClippedRealNNLipsch} (with 
	$ a \is -b $, 
	$ b \is  b $,
	$ u \is u $,
	$ v \is v $,  
	$ d \is \mathfrak{d} $, 
	$ L \is l$, 
	$ l \is \mf l$
in the notation of \cref{cor:ClippedRealNNLipsch}) demonstrates that for all 
	$ \theta, \xi \in B $ 
it holds that 
	\begin{equation} \label{error_decomposition:lipschitz_estimate}
		\begin{split}
		\sup_{x\in D}\,
		\abs{\ClippedRealV{\theta}{\mf l}{u}{v}(x)-\ClippedRealV{\xi}{\mf{l}}{u}{v}(x)} 
		& \leq 
		\sup_{x\in[-b,b]^{d}}\abs{\ClippedRealV{\theta}{\mf l}{u}{v}(x)-\ClippedRealV{\xi}{\mf l}{u}{v}(x)}
		\\
    & \leq
    l\max\{1,b\}(\infnorm{\mf l}+1)^l(\max\{1,\infnorm\theta,\infnorm\xi\})^{l-1}\infnorm{\theta-\xi}
		\\
    & \leq
    lR(\infnorm{\mf l}+1)^lR^{l-1}\infnorm{\theta-\xi}
  =
  l(\infnorm{\mf l}+1)^lR^{l}\infnorm{\theta-\xi}
  . 
		\end{split}
		\end{equation}
Combining this with \enum{
the fact that for all 
	$ \theta \in \R^{\mathfrak{d}}$, 
	$ x \in D $ 
it holds that 
	$ \ClippedRealV{\theta}{\mathfrak{l}}{u}{v}(x) \in [u,v]$ ;
the hypothesis that 
  for all 
    $m\in\{1,2,\dots,M\}$,
    $\omega\in\Omega$ 
  it holds that
    $ Y_m(\omega) \in [u,v]$ ;
 the fact that for all 
  $x_1,x_2,y\in\R$ 
it holds that 
  $(x_1-y)^2-(x_2-y)^2=(x_1-x_2)((x_1-y)+(x_2-y))$ ;
\eqref{error_decomposition:empirical_risk}
} ensures that for all 
	$\theta,\xi\in B$,
	$\omega\in\Omega$
it holds that 
	\begin{align} 
\nonumber
	& 
	|\mf{E}(\theta,\omega) 
	- \mf{E}(\xi,\omega)| 
	\\
	& = 
	\left| 
	\frac{1}{M}\!\left[ 
		\sum_{m=1}^M 
		| \ClippedRealV{\theta}{\mathfrak{l}}{u}{v}
		(X_m(\omega)) 
		- Y_m(\omega)|^2
	\right]
		- 
		\frac{1}{M}\!\left[
		\sum_{m=1}^M 
		|
		\ClippedRealV{\xi}{\mathfrak{l}}{u}{v}
		(X_m(\omega))
		- Y_m(\omega)|^2
	\right]
		\right|  
    \\\nonumber
    &=
		\frac{1}{M}\Biggl\lvert
		\sum_{m=1}^M \Bigl(
		\bigl(\ClippedRealV{\theta}{\mathfrak{l} }{u}{v}(X_m(\omega)) 
		- 
		\ClippedRealV{\xi}{ \mathfrak{l} }{u}{v}(X_m(\omega))
		\bigr)
		\left[ 
			\bigl(
      \ClippedRealV{\theta}{\mathfrak{l} }{u}{v}(X_m(\omega)) - Y_m(\omega)
      \bigr)
			+ 
			\bigl(
        \ClippedRealV{\xi}{ \mathfrak{l} }{u}{v}(X_m(\omega))- Y_m(\omega)	
      \bigr)
      \right]
		\Bigr)
		\Biggr\rvert
    \\
    \nonumber
		& 
		\leq 
		\frac{1}{M}\!\Biggl[
		\sum_{m=1}^M \Bigl(
		\left|\ClippedRealV{\theta}{\mathfrak{l} }{u}{v}(X_m(\omega)) 
		- 
		\ClippedRealV{\xi}{ \mathfrak{l} }{u}{v}(X_m(\omega))
		\right|
		\underbrace{	\left[ 
			\abs{
			\ClippedRealV{\theta}{\mathfrak{l} }{u}{v}(X_m(\omega)) - Y_m(\omega) }
			+ 
			\abs{	\ClippedRealV{\xi}{ \mathfrak{l} }{u}{v}(X_m(\omega))- Y_m(\omega)	}\right]}_{\leq 2(v-u)}
		\Bigr)
		\Biggr]
		\\
    \nonumber
		& 
		\leq 
		2(v-u)l(\infnorm{\mf l}+1)^{l}R^l
	\infnorm{\theta-\xi}. 
		\end{align} 
\cref{lem:estimate_opt_error1} (with 
	$ (\Omega,\mathcal{F},\P) \is (\Omega,\mathcal{F},\P) $, 
	$ \mathfrak{d} \is \mathfrak{d} $, 
	$ N \is K $, 
	$ a \is -R $, 
	$ b \is R $, 
	$\vartheta \is \vartheta $, 
  $ L \is 2(v-u) l(\infnorm{\mf l}+1)^{l}R^l $, 
  $\eps\is\frac{\eps^2}4$,
	$ \mf E \is \mf E $, 
	$ (\Theta_n)_{n\in\{1,2,\dots,N\}} \is (\Theta_k)_{k\in\{1,2,\dots,K\}} $ 
in the notation of \cref{lem:estimate_opt_error1}) therefore shows that 
    \begin{equation} 
      \label{eq:ed.opterrorestimate}
		\begin{split}
		\P\!\left( 
		\mf{E}( \Xi ) 
		- 
		\mf{E}( \vartheta ) > \frac{\varepsilon^2}{4} \right)  
		& = 
		\P\!\left( 
		\left[ 
		\min_{k\in\{1,2,\ldots,K\}} \mf{E}(\Theta_k) \right] 
		- 
		\mf{E}( \vartheta ) > \frac{\varepsilon^2}{4}
    \right) 
    \\&\leq
    \exp\biggl(
      -K\min\biggl\{1,\frac{\bigl(\frac{\eps^2}4\bigr)^{\mf d}}{[2(v-u) l(\infnorm{\mf l}+1)^{l}R^l]^{\mf d}(2R)^{\mf d}}\biggr\}
    \biggr)
		\\&=
		\exp\biggl(-K\min\!\left\{1, \frac{ \varepsilon^{2\mathfrak{d}}}{
      (16(v-u) l(\infnorm{\mf l}+1)^{l}R^{l+1})^{\mf d}
    } \right\}\biggr)
    . 
		\end{split}
		\end{equation} 
		\newcommand{\ConstForClippedHoeffding}{2c R^{\card\cM+1}(6d\card\cM+3)^{\card\cM+1}}%
Moreover, note that \enum{
  \cref{lem:cov6} (with 
	$ \mf d \is \mf d $,
	$ M \is M $, 
	$ L \is l $, 
	$ u \is u $,
	$ v \is v $,
	$ R \is R $, 
	$ \varepsilon \is \tfrac{\varepsilon^2}4 $, 
  $b\is b$,
  $ l \is \mf l $, 
  $D\is D$,
  $ (\Omega,\mc F,\P) \is (\Omega,\mathcal{F},\P) $, 
  $(X_m)_{m\in\{1,2,\dots,M\}}\is(X_m)_{m\in\{1,2,\dots,M\}}$,
	$ (Y_m)_{m\in\{1,2,\ldots,M\}} \is (Y_m)_{m\in\{1,2,\ldots,M\}} $, 
	$ \mc E \is \mc E $, 
	$ \mf E \is \mf E $ 
in the notation of \cref{lem:cov6})
}[establish] that 	
	\begin{equation} 
		\begin{split}
		& 
		\P\!\left(\sup\nolimits_{\theta\in B}\abs{\mf E(\theta)-\mc E(\ClippedRealV{\theta}{\mf l}{u}{v}|_D)}\geq \frac{\eps^2}{4}\right)
    \\&\leq
    2\max\biggl\{1,\biggl[\frac{128l\max\{1,b\}(\infnorm{\mf l}+1)^lR^{l}(v-u)}{\eps^2}\biggr]^{\mf d}\biggr\}\exp\biggl(\frac{-\eps^4M}{32(v-u)^4}\biggr)
    \\&\leq
    2\max\biggl\{1,\biggl[\frac{128l(\infnorm{\mf l}+1)^lR^{l+1}(v-u)}{\eps^2}\biggr]^{\mf d}\biggr\}\exp\biggl(\frac{-\eps^4M}{32(v-u)^4}\biggr)
    \\&=
    2\exp\!\left(\mf d\ln\!\left(\max\biggl\{1,\frac{128l(\infnorm{\mf l}+1)^lR^{l+1}(v-u)}{\eps^2}\biggr\}\right)-\frac{\eps^4M}{32(v-u)^4}\right)
    \!.
		\end{split}
		\end{equation}	
Combining 
  this 
  and \eqref{eq:ed.opterrorestimate}
with \eqref{error_decomposition:eq02} proves that 
		\begin{multline} 
		\P\!\left(
		\int_D 
		| 
		\ClippedRealV{ \Xi }{ \mathfrak{l} }{u}{v}( x )
		- 
		\varphi( x ) 
		|^2
		\,  
		\P_{ X_1 }( dx )
		> 
		\eps^2 
		\right)
		\leq
		\exp\!\left(-K\min\!\left\{1, \frac{ \varepsilon^{2\mathfrak{d}}}{
      (16(v-u) l(\infnorm{\mf l}+1)^{l}R^{l+1})^{\mf d}
      } \right\}\right)
		\\+
    2\exp\!\left(\mf d\ln\!\left(\max\biggl\{1,\frac{128l(\infnorm{\mf l}+1)^lR^{l+1}(v-u)}{\eps^2}\biggr\}\right)-\frac{\eps^4M}{32(v-u)^4}\right)
		\!. 
		\end{multline} 
		The proof of \cref{prop:error_decomposition} is thus completed.
	\end{proof}
\endgroup

\begin{cor}\label{cor:unspecific_architecture_no_noise}
	Let 
	$ ( \Omega, \mathcal{F}, \P ) $ 
	be a probability space, 
	let 
	$ d, \mathfrak{d}, K, M, \tau \in \N $, 
	$ \varepsilon \in (0,\infty) $, 
	$ L,a,u \in \R $, 
	$ b \in (a,\infty) $,
	$ v \in (u,\infty) $, 
	$ R \in [\max\{1,L,|a|,|b|, 2|u|, 2|v|\},\infty) $, 
	let 
	$ X_m  \colon \Omega \to [a,b]^d $, 
	$ m \in \{ 1, 2, \dots, M \} $,
	be i.i.d.\ random variables, 
	let 
	$ \norm{\cdot}\colon \R^d \to [0,\infty) $ 
	be the standard norm on $ \R^d $, 
	let 
	$ \varphi \colon [a,b]^d \to [u,v] $ 
	satisfy for all 
	$x,y \in [a,b]^d$ 
	that 
	$|\varphi(x) - \varphi(y)| \leq L\norm{x-y}$,    
  assume \enum{
  $\tau\geq 2d(2dL(b-a)\varepsilon^{-1}+2)^d$ ;
  $\mf d\geq \tau(d+1)+(\tau-3)\tau(\tau+1)+\tau+1$
  },
	let 
	$ \mf{l}\in \N^{\tau} $ 
	satisfy 
	$ \mf{l}=(d,\tau,\tau,\ldots,\tau,1) $, 
	let 
	$ B \subseteq \R^{\mf d}$ satisfy $B = [-R,R]^{\mathfrak{d}} $, 
	let
	$ \mathfrak{E} \colon B \times \Omega \to [0,\infty) $
	satisfy for all 
	$ \theta \in B $, 
	$ \omega \in \Omega $ 
	that
	\begin{equation} \label{unspecific_architecture_no_noise:empirical_risk}
	\mathfrak{E}( \theta, \omega )
	=
	\frac{ 1 }{ M }\!
	\left[
	\sum\limits_{ m = 1 }^M
	| \ClippedRealV{\theta}{\mathfrak{l}}{u}{v}( X_m( \omega ) ) - \varphi(X_m( \omega )) |^2
  \right]
  \!,
	\end{equation}
	let 
	$\Theta_k\colon\Omega\to B$, 
	$k\in\{1,2,\ldots,K\}$, 
	be i.i.d.~random variables, 
	assume that 
	$\Theta_1$ 
	is continuous uniformly distributed on 
	$B$, 
	and let 
	$\Xi\colon\Omega\to B$ 
	satisfy 
	$\Xi = \Theta_{\min\{k\in\{1,2,\ldots,K\}\colon\mathfrak{E}(\Theta_k) = \min_{l\in\{1,2,\ldots,K\}} \mathfrak{E}(\Theta_l) \}} 
	$
	(cf.\ \cref{def:rectclippedFFANN}).
	Then
	\begin{multline} \label{unspecific_architecture_no_noise:claim}
	\P\!\left(
	\left[ 
	\int_{[a,b]^d}
	| 
	\ClippedRealV{ \Xi }{ \mathfrak{l} }{u}{v}( x )
	- 
	\varphi( x ) 
	|^2
	\,  
	\P_{ X_1 }( dx )
	\right]^{\nicefrac12}
	> 
	\eps 
	\right)
	\leq
  \exp\!\left(-K\min\!\left\{1, \frac{ \varepsilon^{2\mathfrak{d}}}{
    (16(v-u) (\tau+1)^{\tau}R^{\tau})^{\mf d}
    } \right\}\right)
  \\+
  2\exp\biggl(\mf d\ln\!\left(\max\biggl\{1,\frac{128(\tau+1)^{\tau}R^{\tau}(v-u)}{\eps^2}\biggr\}\right)-\frac{\eps^4M}{32(v-u)^4}\biggr)
.
	\end{multline}
\end{cor}

\begin{proof}[Proof of \cref{cor:unspecific_architecture_no_noise}]
Throughout this proof
let $N\in\N$ satisfy 
	\begin{equation}\label{unspecific_architecture_no_noise:definition_of_N}
	N = \min\biggl\{k\in\N\colon k\geq \frac{2dL(b-a)}{\varepsilon}\biggr\},
  \end{equation}
let 
  $\cM\subseteq [a,b]^d$ satisfy $\cM = \{a,a+\frac{b-a}{N},\ldots,a+\frac{(N-1)(b-a)}{N},b\}^d$,
let $\delta\colon[a,b]^d\times[a,b]^d\to [0,\infty)$ satisfy for all
  $x=(x_1,x_2,\dots,x_d), y=(y_1,y_2,\dots,y_d)\in[a,b]^d$
that
  $\delta(x,y)=\sum_{i=1}^d\abs{x_i-y_i}$,
and let $l_0,l_1,\dots,l_{\tau-1}\in\N$ satisfy
$\mf l=(l_0,l_1,\dots,l_{\tau-1})$.
Observe that for all
  $x\in[a,b]$
there exists
  $y\in\{a,a+\frac{b-a}{N},\ldots,a+\frac{(N-1)(b-a)}{N},b\}$
such that $\abs{x-y}\leq \tfrac{b-a}{2N}$.
This demonstrates that 
\begin{equation} \label{unspecific_architecture_no_noise:mesh_density}
	4L\Biggl[\sup_{x=(x_1,x_2,\ldots,x_d)\in [a,b]^d}
	\left(
	\inf_{y=(y_1,y_2,\ldots,y_d)\in \mc M} 
	\sum_{i=1}^{d} |x_i-y_i| 
  \right)
  \Biggr]
	\leq 
	\frac{2Ld(b-a)}{N}
	\leq 
	\eps.  
\end{equation} 
Hence, we obtain that
\begin{equation}
  \label{eq:uann.covnum}
  \CovNum{([a,b]^d,\delta),\frac{\eps}{4L}}
  \leq
  \card\cM
  =
  (N+1)^d
  .
\end{equation}
Next note that
  \eqref{unspecific_architecture_no_noise:definition_of_N}
implies that
  $N<2dL(b-a)\eps^{-1}+1$.
\enum{
  The hypothesis that 
    $\tau \geq 2d(2dL(b-a)\eps^{-1}+2)^d$ ;
} therefore ensures that 
\begin{equation}
  \label{eq:uann.tau}
  \tau
  >
  2d(N+1)^d
  \geq
  (N+1)^d+2
  .
\end{equation}
Hence, we obtain that for all 
  $ i \in \{2,3,\ldots,(N+1)^d\} $,
  $ j \in \{(N+1)^d+1,(N+1)^d+2,\dots,\tau-2\} $
it holds that 
\begin{equation}
  \label{eq:uann.l}
	l_0 = d, \quad
	l_1 = \tau \geq 2 d (N+1)^d, \quad
  l_{\tau-1} = 1,\quad
  l_i = \tau 
  \geq 2(N+1)^d-2i + 3,
  \quad\text{and}\quad
  l_j = \tau \geq 2
  .  
\end{equation}
Furthermore, observe that the hypothesis that for all 
$x,y\in [a,b]^d$ 
it holds that 
$|\varphi(x)-\varphi(y)|\leq L\norm{x-y}$ 
implies that for all 
$x
,y
\in [a,b]^d$ 
it holds that 
$
|\varphi(x)-\varphi(y)| 
\leq 
L\delta(x,y)
$. 			
Combining \enum{
  this ;
  \eqref{eq:uann.covnum} ;
  \eqref{eq:uann.tau} ;
  \eqref{eq:uann.l} ;
  the hypothesis that
  $\mf d\geq \tau(d+1)+(\tau-3)\tau(\tau+1)+\tau+1=\sum_{i=1}^{\tau-1}l_i(l_{i-1}+1)$ ;
} with
\cref{prop:error_decomposition} (with 
	$ (\Omega,\mathcal{F},\P) \is (\Omega,\mathcal{F},\P) $, 
	$ d \is d $, 
	$ \mf d \is \mf d $, 
	$ K \is K $, 
	$ M \is M $, 
	$ \eps \is \eps $, 
	$ L \is L $, 
	$ u \is u $, 
	$ v \is v $, 
	$ D \is [a,b]^d $, 
	$ (X_m)_{m\in\{1,2,\dots,M\}} \is (X_m)_{m\in\{1,2,\dots,M\}} $, 
  $ (Y_m)_{m\in\{1,2,\dots,M\}} \is (\varphi(X_m))_{m\in\{1,2,\dots,M\}} $, 
  $\delta\is \delta$,
  $ \varphi \is \varphi $, 
  $ N\is (N+1)^d$,
	$ l \is \tau-1 $, 
	$ \mf l \is \mf l $,  
	$ R \is R $, 
  $ B \is B $, 
	$ \mf E \is \mf E $,  
	$ (\Theta_k)_{k\in\{1,2,\dots,K\}} \is (\Theta_k)_{k\in\{1,2,\dots,K\}} $, 
	$ \Xi \is \Xi $ 
in the notation of \cref{prop:error_decomposition}) 
establishes that 
	\begin{equation} \begin{split}
	&
	\P\!\left(
	\left[ 
	\int_{[a,b]^d}
	| 
	\ClippedRealV{ \Xi }{ \mathfrak{l} }{u}{v}( x )
	- 
	\varphi( x ) 
	|^2
	\,  
	\P_{ X_1 }( dx )
	\right]^{\nicefrac12}
	> 
	\eps 
	\right)
  \\&\leq
  \exp\!\left(-K\min\!\left\{1, \frac{ \varepsilon^{2\mathfrak{d}}}{
    (16(v-u) (\tau-1)(\tau+1)^{\tau-1}R^{\tau})^{\mf d}
    } \right\}\right)
  \\&\qquad\qquad+
  2\exp\!\left(\mf d\ln\!\left(\max\biggl\{1,\frac{128(\tau-1)(\tau+1)^{\tau-1}R^{\tau}(v-u)}{\eps^2}\biggr\}\right)-\frac{\eps^4M}{32(v-u)^4}\right)
  \\&\leq
  \exp\!\left(-K\min\!\left\{1, \frac{ \varepsilon^{2\mathfrak{d}}}{
    (16(v-u) (\tau+1)^{\tau}R^{\tau})^{\mf d}
    } \right\}\right)
  \\&
  \qquad\qquad +
  2\exp\!\left(\mf d\ln\!\left(\max\biggl\{1,\frac{128(\tau+1)^{\tau}R^{\tau}(v-u)}{\eps^2}\biggr\}\right)-\frac{\eps^4M}{32(v-u)^4}\right)
  \!.
	\end{split} \end{equation} 
The proof of \cref{cor:unspecific_architecture_no_noise} is thus completed. 
\end{proof} 

\begin{cor}\label{cor:generic_constants}
Let 
	$ ( \Omega, \mathcal{F}, \P ) $ 
be a probability space, 
let 
	$ d \in \N $, 
	$ L,a,u \in \R $, 
	$ b \in (a,\infty) $,
	$ v \in (u,\infty) $, 
	$ R \in [\max\{1, L, |a|, |b|, 2|u|, 2|v|\}, \infty)  $,  
let 
	$ X_m \colon \Omega \to [a,b]^d $, 
	$ m \in \N $,
be i.i.d.\ random variables, 
let 
	$ \norm{\cdot}\colon \R^d \to [0,\infty) $ 
be the standard norm on $ \R^d $, 
let 
	$ \varphi \colon [a,b]^d \to [u,v] $ 
satisfy for all 
	$ x, y \in [a,b]^d $ 
that 
	$| \varphi(x) - \varphi(y) | \leq L\norm{x-y}$,  
let  
  $ \mf{l}_{\tau} \in \N^{\tau} $, $ \tau \in \N $,  
  satisfy for all
    $\tau\in\N\cap[3,\infty)$
  that
  $ \mf{l}_{\tau}=(d,\tau,\tau,\ldots,\tau,1) $,
let
	$ \mathfrak{E}_{\mf d,M,\tau} \colon [-R,R]^{\mf d} \times \Omega \to [0,\infty) $, $ \mf d, M, \tau \in \N $, 
satisfy for all 
	$ \mf d, M \in \N $, 
	$ \tau \in \N\cap[3,\infty) $, 
	$ \theta \in [-R,R]^{\mf d} $,  
  $ \omega \in \Omega $ 
with
  $\mf d\geq \tau(d+1)+(\tau-3)\tau(\tau+1)+\tau+1$
that
\begin{equation} 
	\mathfrak{E}_{ \mf d, M, \tau }( \theta, \omega )
	=
	\frac{ 1 }{ M }
	\left[
	\sum\limits_{ m = 1 }^M
	| \ClippedRealV{\theta}{\mf{l_{\tau}}}{u}{v}( X_m( \omega ) ) - \varphi(X_m( \omega )) |^2
  \right]
  \!,
	\end{equation} 
for every $\mf d \in \N$ let 
	$\Theta_{\mf d,k}\colon\Omega\to [-R,R]^{\mf d}$, 
	$k\in\N$, 
be i.i.d.~random variables, 
assume for all $\mf d\in\N$ that 
	$\Theta_{\mf d,1}$ 
is continuous uniformly distributed on 
	$[-R,R]^{\mf d}$, 
and let 
	$\Xi_{\mf d,K,M,\tau}\colon\Omega\to [-R,R]^{\mf d}$, $\mf d,K,M,\tau\in\N$, 
satisfy for all $ \mf d,K,M,\tau \in \N $ that
	$\Xi_{\mf d,K,M,\tau} = \Theta_{\mf d, \min\{k\in\{1,2,\ldots,K\}\colon\mf{E}_{\mf d,M,\tau}(\Theta_{\mf d,k}) = \min_{l\in\{1,2,\ldots,K\}} \mf{E}_{\mf d,M,\tau}(\Theta_{\mf d,l}) \}} 
  $
  (cf.\ \cref{def:rectclippedFFANN}).
Then there exists $c\in (0,\infty)$ such that for all 
	$\mathfrak{d}, K, M, \tau \in \N$, 
  $ \varepsilon \in (0,\sqrt{v-u}] $  
  with
  $\tau\geq 2d(2dL(b-a)\varepsilon^{-1}+2)^d$
  and
  $\mf d\geq \tau(d+1)+(\tau-3)\tau(\tau+1)+\tau+1$
	it holds that
	\begin{equation} \label{generic_constants:claim}
	\begin{split}
	&
	\P\!\left(
	\left[ 
	\int_{[a,b]^d}
	| 
	\ClippedRealV{ \Xi_{\mf d,K,M,\tau} }{ \mathfrak{l}_{\tau} }{u}{v}( x )
	- 
	\varphi( x ) 
	|^2
	\,  
	\P_{ X_1 }( dx )
	\right]^{\nicefrac12}
	> 
	\eps 
	\right)
	\\
	& 
	\leq 
	\exp\!\left(-K (c\tau)^{-\tau\mf d}\varepsilon^{2\mf{d}} \right)
	+
  2\exp\!\left(\mf d\ln\!\left( (c\tau)^{\tau}\eps^{-2} \right) - c^{-1}\varepsilon^4 M \right)
  \!. 
	\end{split}
	\end{equation}
	\end{cor}

\begin{proof}[Proof of \cref{cor:generic_constants}] Throughout this proof let $c\in (0,\infty)$ satisfy
	\begin{equation}\label{generic_constants:definition_of_c}
		c = 
		\max\{ 
    32 (v-u)^4, 
    256(v-u+1)R
		\}. 
	\end{equation} 
Note that \cref{cor:unspecific_architecture_no_noise} establishes that for all 
	$\mathfrak{d}, K, M, \tau \in \N$, 
	$ \varepsilon \in (0,\infty) $  
  with \enum{
    $\tau\geq 2d({2dL(b-a)}\varepsilon^{-1}+2)^d$ ;
    $\mf d\geq \tau(d+1)+(\tau-3)\tau(\tau+1)+\tau+1$
  }
  it holds that
	\begin{multline} \label{generic_constants:estimate}
	\P\!\left(
	\left[ 
	\int_{[a,b]^d}
	| 
	\ClippedRealV{ \Xi_{\mf d,K,M,\tau} }{ \mathfrak{l}_{\tau} }{u}{v}( x )
	- 
	\varphi( x ) 
	|^2
	\,  
	\P_{ X_1 }( dx )
	\right]^{\nicefrac12}
	> 
	\eps 
	\right)
	\leq
  \exp\!\left(-K\min\!\left\{1, \frac{ \varepsilon^{2\mathfrak{d}}}{
    (16(v-u) (\tau+1)^{\tau}R^{\tau})^{\mf d}
    } \right\}\right)
  \\+
  2\exp\biggl(\mf d\ln\!\left(\max\biggl\{1,\frac{128(\tau+1)^{\tau}R^{\tau}(v-u)}{\eps^2}\biggr\}\right)-\frac{\eps^4M}{32(v-u)^4}\biggr)
  .
	\end{multline} 
Next observe that 
  \eqref{generic_constants:definition_of_c} 
ensures that for all
  $\tau\in\N$
it holds that
\begin{equation}
  16(v-u)(\tau+1)^\tau R^\tau
  \leq 
  (16(v-u+1)(\tau+1)R)^\tau
  \leq 
  (32(v-u+1)R\tau)^\tau
  \leq 
  (c\tau)^\tau
  .
\end{equation}
  The fact that for all
    $\eps\in(0,\sqrt{v-u}]$,
    $\tau\in\N$
  it holds that
    $\eps^2\leq 16(v-u)(\tau+1)^\tau R^\tau$
  therefore
shows that for all
  $\eps\in(0,\sqrt{v-u}]$,
  $\tau\in\N$
it holds that
\begin{equation}
  \label{eq:gc.1}
  {-}\min\!\left\{1, \frac{ \varepsilon^{2\mathfrak{d}}}{
    (16(v-u) (\tau+1)^{\tau}R^{\tau})^{\mf d}
    } \right\}
  =
  \frac{-\varepsilon^{2\mathfrak{d}}}{
    (16(v-u) (\tau+1)^{\tau}R^{\tau})^{\mf d}
    }
  \leq 
  \frac{-\varepsilon^{2\mathfrak{d}}}{
    (c\tau)^{\tau\mf d}
    }
    .
\end{equation}
Furthermore, note that
  \eqref{generic_constants:definition_of_c} 
implies that for all
  $\tau\in\N$
it holds that
\begin{equation}
  128(\tau+1)^\tau R^\tau(v-u) 
  \leq 
  128(2\tau)^\tau R^\tau(v-u)
  \leq  
  (256R\tau(v-u+1))^\tau 
  \leq 
  (c\tau)^{\tau} 
  .
\end{equation}
  The fact that for all
    $\eps\in(0,\sqrt{v-u}]$,
    $\tau\in\N$
  it holds that
    $\eps^2\leq 128(\tau+1)^\tau R^\tau(v-u)$
  hence
proves that for all
  $\eps\in(0,\sqrt{v-u}]$,
  $\tau\in\N$
it holds that
\begin{equation}
  \label{eq:gc.2}
  \ln\!\left(\max\biggl\{1,\frac{128(\tau+1)^{\tau}R^{\tau}(v-u)}{\eps^2}\biggr\}\right)
  =
  \ln\!\left(\frac{128(\tau+1)^{\tau}R^{\tau}(v-u)}{\eps^2}\right)
  \leq
  \ln\!\left(\frac{(c\tau)^\tau}{\eps^2}\right)
\end{equation}
In addition, observe that
  \eqref{generic_constants:definition_of_c} 
ensures that
\begin{equation}
  \frac{-1}{32(v-u)^4}
  \leq 
  \frac{-1}{c}
  .
\end{equation}
Combining \enum{
  this ;
  \eqref{eq:gc.1};
  \eqref{eq:gc.2};
} with \eqref{generic_constants:estimate} proves that for all 
	$\mf d, K, M, \tau\in\N$, 
	$\eps \in (0,\sqrt{v-u}]$ 
with \enum{
  $\tau\geq 2d({2dL(b-a)}\varepsilon^{-1}+2)^d$ ;
  $\mf d\geq \tau(d+1)+(\tau-3)\tau(\tau+1)+\tau+1$
}
it holds that
	\begin{equation} 
	\begin{split}
	&
	\P\!\left(
	\left[ 
	\int_{[u,v]^d}
	| 
	\ClippedRealV{ \Xi_{\mf d,K,M,\tau} }{ \mathfrak{l}_{\tau} }{u}{v}( x )
	- 
	\varphi( x ) 
	|^2
	\,  
	\P_{ X_1 }( dx )
	\right]^{\nicefrac12}
	> 
	\eps 
	\right)
	\\ & \leq
  \exp\biggl( \frac{ -K\varepsilon^{2\mathfrak{d}}}{
    (c\tau)^{\tau\mf d}
    } \biggr)
  +
  2\exp\!\left(\mf d\ln\!\left(\frac{(c\tau)^\tau}{\eps^2}\right)-\frac{\eps^4M}{c}\right)
  \!.
	\end{split}
	\end{equation}  
The proof of \cref{cor:generic_constants} is thus completed. 
\end{proof}

\begin{cor} \label{cor:l1_norm}
  Let 
  $ ( \Omega, \mathcal{F}, \P ) $ 
  be a probability space, 
  let 
$ d \in \N $, 
$ L,a,u \in \R $, 
$ b \in (a,\infty) $,
$ v \in (u,\infty) $, 
$ R \in [\max\{1, L, |a|, |b|, 2|u|, 2|v|\}, \infty)  $,  
let 
$ X_m \colon \Omega \to [a,b]^d $, 
$ m \in \N $,
be i.i.d.\ random variables, 
let 
$ \norm{\cdot}\colon \R^d \to [0,\infty) $ 
be the standard norm on $ \R^d $, 
let 
$ \varphi \colon [a,b]^d \to [u,v] $ 
satisfy for all 
$ x, y \in [a,b]^d $ 
that 
$| \varphi(x) - \varphi(y) | \leq L\norm{x-y}$,  
let  
  $ \mf{l}_{\tau} \in \N^{\tau} $, $ \tau \in \N $,  
  satisfy for all
    $\tau\in\N\cap[3,\infty)$
  that
  $ \mf{l}_{\tau}=(d,\tau,\tau,\ldots,\tau,1) $,
let
	$ \mathfrak{E}_{\mf d,M,\tau} \colon [-R,R]^{\mf d} \times \Omega \to [0,\infty) $, $ \mf d, M, \tau \in \N $, 
satisfy for all 
	$ \mf d, M \in \N $, 
	$ \tau \in \N\cap[3,\infty) $, 
	$ \theta \in [-R,R]^{\mf d} $,  
  $ \omega \in \Omega $ 
with
  $\mf d\geq \tau(d+1)+(\tau-3)\tau(\tau+1)+\tau+1$
that
\begin{equation} 
\mathfrak{E}_{ \mf d, M, \tau }( \theta, \omega )
=
\frac{ 1 }{ M }
\left[
\sum\limits_{ m = 1 }^M
| \ClippedRealV{\theta}{\mf{l_{\tau}}}{u}{v}( X_m( \omega ) ) - \varphi(X_m( \omega )) |^2
\right]
\!,
\end{equation} 
for every $\mf d \in \N$ let 
$\Theta_{\mf d,k}\colon\Omega\to [-R,R]^{\mf d}$, 
$k\in\N$, 
be i.i.d.~random variables, 
assume for all $\mf d\in\N$ that 
$\Theta_{\mf d,1}$ 
is continuous uniformly distributed on 
$[-R,R]^{\mf d}$, 
and let 
$\Xi_{\mf d,K,M,\tau}\colon\Omega\to [-R,R]^{\mf d}$, $\mf d,K,M,\tau\in\N$, 
satisfy for all $ \mf d,K,M,\tau \in \N $ that
$\Xi_{\mf d,K,M,\tau} = \Theta_{\mf d, \min\{k\in\{1,2,\ldots,K\}\colon\mf{E}_{\mf d,M,\tau}(\Theta_{\mf d,k}) = \min_{l\in\{1,2,\ldots,K\}} \mf{E}_{\mf d,M,\tau}(\Theta_{\mf d,l}) \}} 
$
(cf.\ \cref{def:rectclippedFFANN}).
Then there exists $c\in (0,\infty)$ such that for all 
$\mathfrak{d}, K, M, \tau \in \N$, 
$ \varepsilon \in (0,\sqrt{v-u}] $  
with
$\tau\geq 2d(2dL(b-a)\varepsilon^{-1}+2)^d$
and
$\mf d\geq \tau(d+1)+(\tau-3)\tau(\tau+1)+\tau+1$
it holds that
\begin{equation} \label{l1_norm:claim}
\begin{split}
&
\P\!\left(
\int_{[a,b]^d}
| 
\ClippedRealV{ \Xi_{\mf d,K,M,\tau} }{ \mathfrak{l}_{\tau} }{u}{v}( x )
- 
\varphi( x ) 
|
\,  
\P_{ X_1 }( dx )
> 
\eps 
\right)
\\
& 
\leq 
\exp\bigl(-K (c\tau)^{-\tau\mf d}\varepsilon^{2\mf{d}}\bigr)
+
2\exp\bigl(\mf d\ln\!\left( (c\tau)^{\tau}\eps^{-2} \right) - c^{-1}\varepsilon^4 M \bigr). 
\end{split}
\end{equation}
\end{cor} 
\begin{proof}[Proof of \cref{cor:l1_norm}]
  Note that
    Jensen's inequality
  shows that for all
    $f\in C([a,b]^d,\R)$
  it holds that
  \begin{equation}
    \int_{[a,b]^d}\abs{f(x)}\,\P_{X_1}(dx)
    \leq
    \biggl[
      \int_{[a,b]^d}\abs{f(x)}^2\,\P_{X_1}(dx)
    \biggr]^{\frac12}
    .
  \end{equation}
  Combining
    this
  with
    \cref{cor:generic_constants}
  proves \eqref{l1_norm:claim}.
The proof of \cref{cor:l1_norm} is thus completed. 
\end{proof} 

\subsubsection{Convergence rates for strong convergence}

\begin{lemma} \label{lem:quantitative_lebesgue} 
Let $ ( \Omega, \mathcal{F}, \P ) $ be a probability space, 
let $ c \in [0,\infty) $, 
and 
let $ X \colon \Omega \to [-c, c] $ be a random variable. 
Then it holds for all $ \eps, p \in (0,\infty) $ that 
	\begin{equation} \label{quantitative_lebesgue:claim}
	\E\!\left[|X|^p\right] 
	\leq 
	\eps^p \, \P\!\left( |X| \leq \eps \right) + c^p \, \P\!\left( |X| > \eps \right)
	\leq 
	\eps^p + c^p \, \P( |X| > \eps ). 
	\end{equation}
\end{lemma}

\begin{proof}[Proof of \cref{lem:quantitative_lebesgue}]
Observe that 
  the hypothesis that for all
    $\omega\in\Omega$
  it holds that
    $\abs{X(\omega)}\leq c$
ensures that  for all 
  $ \eps, p \in (0,\infty) $ 
it holds that 
	\begin{equation} 
	\E\!\left[|X|^p\right] 
	= 
	\E\!\left[|X|^p\mathbbm{1}_{\{|X|\leq\eps\}}\right]
	+ 
	\E\!\left[|X|^p\mathbbm{1}_{\{|X|>\eps\}}\right]	
	\leq 
	\eps^p \, \P\!\left( |X| \leq \eps \right) 
	+ 
	c^p \, \P( |X| > \eps )
  \leq 
  \eps^p + c^p \, \P\!\left( |X| > \eps \right).  
	\end{equation} 
The proof of \cref{lem:quantitative_lebesgue} is thus completed. 
\end{proof} 

\begin{cor} \label{cor:lp_error}
Let $ ( \Omega, \mathcal{F}, \P ) $ be a probability space, 
let $ d \in \N $, 
	$ L,a,u \in \R $, 
	$ b \in (a,\infty) $,
	$ v \in (u,\infty) $, 
	$ R \in [\max\{1, L, |a|, |b|, 2|u|, 2|v|\}, \infty)  $,  
let 
	$ X_m 
	\colon \Omega \to [a,b]^d $, 
	$ m \in \N $,
	be i.i.d.\ random variables, 
  let $ \norm{\cdot}\colon \R^d \to [0,\infty) $ be the standard norm on $ \R^d $, 
	let 
	$ \varphi \colon [a,b]^d \to [u,v] $ 
	satisfy for all 
	$ x, y \in [a,b]^d $ 
	that 
	$| \varphi(x) - \varphi(y) | \leq L\norm{x-y}$,  
  let  
  $ \mf{l}_{\tau} \in \N^{\tau} $, $ \tau \in \N $,  
  satisfy for all
    $\tau\in\N\cap[3,\infty)$
  that
  $ \mf{l}_{\tau}=(d,\tau,\tau,\ldots,\tau,1) $,
let
	$ \mathfrak{E}_{\mf d,M,\tau} \colon [-R,R]^{\mf d} \times \Omega \to [0,\infty) $, $ \mf d, M, \tau \in \N $, 
satisfy for all 
	$ \mf d, M \in \N $, 
	$ \tau \in \N\cap[3,\infty)$, 
	$ \theta \in [-R,R]^{\mf d} $,  
  $ \omega \in \Omega $ 
with
  $\mf d\geq \tau(d+1)+(\tau-3)\tau(\tau+1)+\tau+1$
that
	\begin{equation} 
	\mathfrak{E}_{ \mf d, M, \tau }( \theta, \omega )
	=
	\frac{ 1 }{ M }
	\left[
	\sum\limits_{ m = 1 }^M
	| \ClippedRealV{\theta}{\mf{l_{\tau}}}{u}{v}( X_m( \omega ) ) - \varphi(X_m( \omega )) |^2
	\right]\!,
	\end{equation} 
	for every $\mf d\in\N$ let 
	$\Theta_{\mf d,k}\colon\Omega\to [-R,R]^{\mf d}$, 
	$k\in\N$, 
	be i.i.d.~random variables, 
	assume for all $\mf d\in\N$ that 
	$\Theta_{\mf d,1}$ 
	is continuous uniformly distributed on 
	$[-R,R]^{\mf d}$, 
	and let 
	$\Xi_{\mf d,K,M,\tau}\colon\Omega\to [-R,R]^{\mf d}$, $\mf d,K,M,\tau\in\N$,
	satisfy for all $\mf d,K,M,\tau\in\N$ that
	$\Xi_{\mf d,K,M,\tau} = \Theta_{\mf d, \min\{k\in\{1,2,\ldots,K\}\colon\mf{E}_{\mf d,M,\tau}(\Theta_{\mf d,k}) = \min_{l\in\{1,2,\ldots,K\}} \mf{E}_{\mf d,M,\tau}(\Theta_{\mf d,l}) \}} 
  $
  (cf.\ \cref{def:rectclippedFFANN}).
	Then there exists $c\in (0,\infty)$ such that for all 
	$\mathfrak{d}, K, M, \tau \in \N$, 
	$ p \in [1,\infty) $, 
	$ \varepsilon \in (0,\sqrt{v-u}] $  
  with 
  $\tau\geq 2d(2dL(b-a)\varepsilon^{-1}+2)^d $
  and
  $\mf d\geq \tau(d+1)+(\tau-3)\tau(\tau+1)+\tau+1$
  it holds that
	\begin{equation} \label{lp_error:claim}
	\begin{split}
	&
	\left( \E\!\left[ \left(
	\int_{[a,b]^d}
	| 
	\ClippedRealV{ \Xi_{\mf d,K,M,\tau} }{ \mathfrak{l}_{\tau} }{u}{v}( x )
	- 
	\varphi( x ) 
	|^2
	\,  
	\P_{ X_1 }( dx ) 
	\right)^{\!\!\nicefrac{p}{2}} 
	\right]\right)^{\!\!\nicefrac{1}{p}}
	\\
	& 
	\leq 
	(v-u) 
	\left[ 
	\exp\!\left(-K (c\tau)^{-\tau\mf d}\varepsilon^{2 \mf{d}} \right)
	+
	2 \exp\!\left(\mf d\ln\!\left( (c\tau)^{\tau}\eps^{-2} \right) - c^{-1}\varepsilon^4 M \right)\right]^{\nicefrac{1}{p}}  
	+ \varepsilon. 
	\end{split}
	\end{equation}
\end{cor} 

\begin{proof}[Proof of \cref{cor:lp_error}] 
First, observe that \cref{cor:generic_constants} ensures that there exists $c\in (0,\infty)$ which satisfies for all 
	$ \mf d, K, M, \tau \in \N $, 
	$ \eps \in (0,\sqrt{v-u}] $
  with \enum{
    $\tau\geq 2d({2dL(b-a)}\varepsilon^{-1}+2)^d$ ;
    $\mf d\geq \tau(d+1)+(\tau-3)\tau(\tau+1)+\tau+1$
  }
that 
	\begin{equation}
	\begin{split}
	& 
	\P\!\left(
	\left[ \int_{[a,b]^d}
	| 
	\ClippedRealV{ \Xi_{\mf d,K,M,\tau} }{ \mathfrak{l}_{\tau} }{u}{v}( x )
	- 
	\varphi( x ) 
	|^2
	\,  
	\P_{ X_1 }( dx ) 
	\right]^{\nicefrac12} 
	> \eps 
	\right)
	\\
	& 
	\leq 
	\exp\!\left(-K (c\tau)^{-\tau \mf d}		
	\varepsilon^{2 \mf{d}}\right)
	+
	2\exp\bigl(\mf d\ln\!\left( (c\tau)^{\tau}\eps^{-2} \right) - c^{-1}\varepsilon^4 M \bigr). 
	\end{split}
	\end{equation} 	
\cref{lem:quantitative_lebesgue} (with 
	$ (\Omega, \mathcal{F}, \P ) \is ( \Omega, \mathcal{F}, \P ) $, 
	$ c \is v-u $, 
	$ X \is (\Omega \ni \omega \mapsto \bigl[
	\int_{[a,b]^d}
	| 
	\ClippedRealV{ \Xi_{\mf d,K,M,\tau}(\omega) }{ \mathfrak{l}_{\tau} }{u}{v}( x )
	- 
	\varphi( x ) 
	|^2
	\,  
	\P_{ X_1 }( dx ) 
	\bigr]^{\smash{\nicefrac12}} \in [u-v,v-u] ) $
in the notation of \cref{lem:quantitative_lebesgue}) hence ensures that for all 
	$ \mf d, K, M, \tau \in \N $, 
  $ \eps\in(0,\sqrt{v-u}]$, 
  $p \in (0,\infty) $ 
  with \enum{
    $\tau\geq 2d({2dL(b-a)}\varepsilon^{-1}+2)^d$ ;
    $\mf d\geq \tau(d+1)+(\tau-3)\tau(\tau+1)+\tau+1$
  }
it holds that 
	\begin{equation} 
	\begin{split}
	&
	\E\!\left[ \left(
	\int_{[a,b]^d}
	| 
	\ClippedRealV{ \Xi_{\mf d,K,M,\tau} }{ \mathfrak{l}_{\tau} }{u}{v}( x )
	- 
	\varphi( x ) 
	|^2
	\,  
	\P_{ X_1 }( dx ) 
	\right)^{\!\!\nicefrac{p}{2}} 
	\right]
	\\
	& 
	\leq 
	\eps^p + 
	(v-u)^p 
	\Bigl[ 
	\exp\!\left(-K (c\tau)^{-\tau\mf d}\varepsilon^{2\mf{d}} \right)
	+
  2\exp\!\left(\mf d\ln\!\left( (c\tau)^{\tau}\eps^{-2} \right) - c^{-1}\varepsilon^4 M \right)
  \Bigr] 
  . 
	\end{split}
	\end{equation}
The fact that for all 
	$ p \in [1,\infty)$, 
	$ x,y \in [0,\infty) $ 
it holds that 
	$ (x+y)^{\nicefrac{1}{p}} \leq x^{\nicefrac{1}{p}} 
 	+ y^{\nicefrac{1}{p}} $ 
therefore establishes \eqref{lp_error:claim}. The proof of \cref{cor:lp_error} is thus completed. 
\end{proof} 

\subsubsection*{Acknowledgements}
This work has been funded by the Deutsche Forschungsgemeinschaft (DFG, German Research Foundation) 
under Germany's Excellence Strategy EXC 2044--390685587, 
Mathematics M\"unster: Dynamics--Geometry--Structure.

\bibliographystyle{acm}
\bibliography{Training_Networks}

\begin{thebibliography}{10}

\bibitem{bach2017breaking}
{\sc Bach, F.}
\newblock Breaking the curse of dimensionality with convex neural networks.
\newblock {\em J. Mach. Learn. Res. 18\/} (2017), 53 pages.

\bibitem{bach2013non}
{\sc Bach, F., and Moulines, E.}
\newblock Non-strongly-convex smooth stochastic approximation with convergence
  rate {$O(\nicefrac{1}{n})$}.
\newblock In {\em Proceedings of the 26th International Conference on Neural
  Information Processing Systems\/} (USA, 2013), NIPS'13, Curran Associates
  Inc., pp.~773--781.

\bibitem{barron1993universal}
{\sc Barron, A.~R.}
\newblock Universal approximation bounds for superpositions of a sigmoidal
  function.
\newblock {\em IEEE Trans. Inform. Theory 39}, 3 (1993), 930--945.

\bibitem{barron1994approximation}
{\sc Barron, A.~R.}
\newblock Approximation and estimation bounds for artificial neural networks.
\newblock {\em Machine Learning 14}, 1 (1994), 115--133.

\bibitem{Bartlett2005}
{\sc Bartlett, P.~L., Bousquet, O., and Mendelson, S.}
\newblock Local {R}ademacher complexities.
\newblock {\em Ann. Statist. 33}, 4 (2005), 1497--1537.

\bibitem{kolmogorov2018solving}
{\sc Beck, C., Becker, S., Grohs, P., Jaafari, N., and Jentzen, A.}
\newblock Solving stochastic differential equations and {K}olmogorov equations
  by means of deep learning.
\newblock {\em arXiv:1806.00421\/} (2018), 56 pages.

\bibitem{beck2019machine}
{\sc Beck, C., E, W., and Jentzen, A.}
\newblock Machine {L}earning {A}pproximation {A}lgorithms for
  {H}igh-{D}imensional {F}ully {N}onlinear {P}artial {D}ifferential {E}quations
  and {S}econd-order {B}ackward {S}tochastic {D}ifferential {E}quations.
\newblock {\em J. Nonlinear Sci. 29}, 4 (2019), 1563--1619.

\bibitem{Bellman1957}
{\sc Bellman, R.}
\newblock {\em Dynamic programming}.
\newblock Princeton Landmarks in Mathematics. Princeton University Press,
  Princeton, NJ, 2010.
\newblock Reprint of the 1957 edition.

\bibitem{bercu2011generic}
{\sc Bercu, B., and Fort, J.-C.}
\newblock Generic stochastic gradient methods.
\newblock {\em Wiley Encyclopedia of Operations Research and Management
  Science\/} (2011), 1--8.

\bibitem{BernerGrohsJentzen2018}
{\sc Berner, J., Grohs, P., and Jentzen, A.}
\newblock Analysis of the generalization error: Empirical risk minimization
  over deep artificial neural networks overcomes the curse of dimensionality in
  the numerical approximation of {B}lack--{S}choles partial differential
  equations.
\newblock {\em arXiv:1809.03062\/} (2018), 35 pages.

\bibitem{blum1991approximation}
{\sc Blum, E.~K., and Li, L.~K.}
\newblock Approximation theory and feedforward networks.
\newblock {\em Neural Networks 4}, 4 (1991), 511--515.

\bibitem{BolcskeiGrohsKutyniokPetersen2019OptimalApproximation}
{\sc B\"{o}lcskei, H., Grohs, P., Kutyniok, G., and Petersen, P.}
\newblock Optimal approximation with sparsely connected deep neural networks.
\newblock {\em SIAM J. Math. Data Sci. 1}, 1 (2019), 8--45.

\bibitem{BurgerNeubauer2001}
{\sc Burger, M., and Neubauer, A.}
\newblock Error bounds for approximation with neural networks.
\newblock {\em J. Approx. Theory 112}, 2 (2001), 235--250.

\bibitem{candes1998ridgelets}
{\sc Candes, E.~J.}
\newblock {\em Ridgelets: theory and applications}.
\newblock PhD thesis, Stanford University Stanford, 1998.

\bibitem{chau2019stochastic}
{\sc Chau, N.~H., Moulines, {\'E}., R{\'a}sonyi, M., Sabanis, S., and Zhang,
  Y.}
\newblock On stochastic gradient {L}angevin dynamics with dependent data
  streams: the fully non-convex case.
\newblock {\em arXiv:1905.13142\/} (2019), 27 pages.

\bibitem{chen1995approximation}
{\sc Chen, T., and Chen, H.}
\newblock Approximation capability to functions of several variables, nonlinear
  functionals, and operators by radial basis function neural networks.
\newblock {\em IEEE Trans. Neural Netw. 6}, 4 (1995), 904--910.

\bibitem{ChuiMhaskar1994}
{\sc Chui, C.~K., Li, X., and Mhaskar, H.~N.}
\newblock Neural networks for localized approximation.
\newblock {\em Math. Comp. 63}, 208 (1994), 607--623.

\bibitem{CuckerSmale2002FoundationsLearning}
{\sc Cucker, F., and Smale, S.}
\newblock On the mathematical foundations of learning.
\newblock {\em Bull. Amer. Math. Soc. (N.S.) 39}, 1 (2002), 1--49.

\bibitem{Cybenko1989}
{\sc Cybenko, G.}
\newblock Approximation by superpositions of a sigmoidal function.
\newblock {\em Math. Control Signals Systems 2}, 4 (1989), 303--314.

\bibitem{DereichMuellerGronbach2019}
{\sc Dereich, S., and M\"{u}ller-Gronbach, T.}
\newblock General multilevel adaptations for stochastic approximation
  algorithms of {R}obbins-{M}onro and {P}olyak-{R}uppert type.
\newblock {\em Numer. Math. 142}, 2 (2019), 279--328.

\bibitem{DeVore1997}
{\sc DeVore, R.~A., Oskolkov, K.~I., and Petrushev, P.~P.}
\newblock Approximation by feed-forward neural networks.
\newblock In {\em The heritage of P. L. Chebyshev: a Festschrift in honor of
  the 70th birthday of T. J. Rivlin}, vol.~4. Baltzer Science Publishers BV,
  Amsterdam, 1997, pp.~261--287.

\bibitem{wang2018exponential}
{\sc E, W., and Wang, Q.}
\newblock Exponential convergence of the deep neural network approximation for
  analytic functions.
\newblock {\em arXiv:1807.00297\/} (2018), 7 pages.

\bibitem{ElbraechterSchwab2018}
{\sc Elbr{\"a}chter, D., Grohs, P., Jentzen, A., and Schwab, C.}
\newblock {DNN} expression rate analysis of high-dimensional {PDE}s:
  Application to option pricing.
\newblock {\em arXiv:1809.07669\/} (2018), 50 pages.

\bibitem{eldan2016power}
{\sc Eldan, R., and Shamir, O.}
\newblock The power of depth for feedforward neural networks.
\newblock In {\em 29th Annual Conference on Learning Theory\/} (Columbia
  University, New York, New York, USA, 23--26 Jun 2016), V.~Feldman,
  A.~Rakhlin, and O.~Shamir, Eds., vol.~49 of {\em Proceedings of Machine
  Learning Research}, PMLR, pp.~907--940.

\bibitem{Ellacott1994}
{\sc Ellacott, S.~W.}
\newblock Aspects of the numerical analysis of neural networks.
\newblock In {\em Acta numerica, 1994}, Acta Numer. Cambridge University Press,
  Cambridge, 1994, pp.~145--202.

\bibitem{FehrmanGessJentzen2019convergence}
{\sc Fehrman, B., Gess, B., and Jentzen, A.}
\newblock Convergence rates for the stochastic gradient descent method for
  non-convex objective functions.
\newblock {\em arXiv:1904.01517\/} (2019), 59 pages.

\bibitem{funahashi1989approximate}
{\sc Funahashi, K.-I.}
\newblock On the approximate realization of continuous mappings by neural
  networks.
\newblock {\em Neural Networks 2}, 3 (1989), 183--192.

\bibitem{Goodfellow2016DeepLearning}
{\sc Goodfellow, I., Bengio, Y., and Courville, A.}
\newblock {\em Deep learning}.
\newblock Adaptive Computation and Machine Learning. MIT Press, Cambridge, MA,
  2016.

\bibitem{gribonval2019approximation}
{\sc Gribonval, R., Kutyniok, G., Nielsen, M., and Voigtlaender, F.}
\newblock Approximation spaces of deep neural networks.
\newblock {\em arXiv:1905.01208\/} (2019), 63 pages.

\bibitem{GrohsWurstemberger2018}
{\sc Grohs, P., Hornung, F., Jentzen, A., and von Wurstemberger, P.}
\newblock A proof that artificial neural networks overcome the curse of
  dimensionality in the numerical approximation of {B}lack--{S}choles partial
  differential equations.
\newblock {\em arXiv:1809.02362\/} (2018), 124 pages.
\newblock Revision requested from Memoirs of the AMS.

\bibitem{Grohs2019ANNCalculus}
{\sc {Grohs}, P., {Hornung}, F., {Jentzen}, A., and {Zimmermann}, P.}
\newblock {Space-time error estimates for deep neural network approximations
  for differential equations}.
\newblock {\em arXiv:1908.03833\/} (2019), 86 pages.

\bibitem{GrohsJentzenSalimova2019}
{\sc Grohs, P., Jentzen, A., and Salimova, D.}
\newblock Deep neural network approximations for {M}onte {C}arlo algorithms.
\newblock {\em arXiv:1908.10828\/} (2019), 45 pages.

\bibitem{grohs2019deep}
{\sc Grohs, P., Perekrestenko, D., Elbr{\"a}chter, D., and B{\"o}lcskei, H.}
\newblock Deep neural network approximation theory.
\newblock {\em arXiv:1901.02220\/} (2019), 60 pages.

\bibitem{guhring2019error}
{\sc G{\"u}hring, I., Kutyniok, G., and Petersen, P.}
\newblock Error bounds for approximations with deep {ReLU} neural networks in
  {$W^{s,p}$} norms.
\newblock {\em arXiv:1902.07896\/} (2019), 42 pages.

\bibitem{GyoerfiKohlerKrzyzakWalk2002}
{\sc Gy\"{o}rfi, L., Kohler, M., Krzy\.{z}ak, A., and Walk, H.}
\newblock {\em A distribution-free theory of nonparametric regression}.
\newblock Springer Series in Statistics. Springer-Verlag, New York, 2002.

\bibitem{hartman1990layered}
{\sc Hartman, E.~J., Keeler, J.~D., and Kowalski, J.~M.}
\newblock Layered neural networks with {G}aussian hidden units as universal
  approximations.
\newblock {\em Neural Comput. 2}, 2 (1990), 210--215.

\bibitem{Hoeffding}
{\sc Hoeffding, W.}
\newblock Probability inequalities for sums of bounded random variables.
\newblock {\em J. Amer. Statist. Assoc. 58}, 301 (1963), 13--30.

\bibitem{hornik1991approximation}
{\sc Hornik, K.}
\newblock Approximation capabilities of multilayer feedforward networks.
\newblock {\em Neural Networks 4}, 2 (1991), 251--257.

\bibitem{hornik1993some}
{\sc Hornik, K.}
\newblock Some new results on neural network approximation.
\newblock {\em Neural Networks 6}, 8 (1993), 1069--1072.

\bibitem{hornik1989multilayer}
{\sc Hornik, K., Stinchcombe, M., and White, H.}
\newblock Multilayer feedforward networks are universal approximators.
\newblock {\em Neural Networks 2}, 5 (1989), 359--366.

\bibitem{hornik1990universal}
{\sc Hornik, K., Stinchcombe, M., and White, H.}
\newblock Universal approximation of an unknown mapping and its derivatives
  using multilayer feedforward networks.
\newblock {\em Neural Networks 3}, 5 (1990), 551--560.

\bibitem{HutzenthalerJentzenKruse2019}
{\sc Hutzenthaler, M., Jentzen, A., Kruse, T., and Nguyen, T.~A.}
\newblock A proof that rectified deep neural networks overcome the curse of
  dimensionality in the numerical approximation of semilinear heat equations.
\newblock {\em arXiv:1901.10854\/} (2019), 29 pages.

\bibitem{JentzenKuckuckNeufeldWurstemberger2018}
{\sc Jentzen, A., Kuckuck, B., Neufeld, A., and von Wurstemberger, P.}
\newblock Strong error analysis for stochastic gradient descent optimization
  algorithms.
\newblock {\em arXiv:1801.09324\/} (2018), 75 pages.
\newblock Revision requested from IMA J.\ Numer.\ Anal.

\bibitem{JentzenSalimovaWelti2018}
{\sc Jentzen, A., Salimova, D., and Welti, T.}
\newblock A proof that deep artificial neural networks overcome the curse of
  dimensionality in the numerical approximation of {K}olmogorov partial
  differential equations with constant diffusion and nonlinear drift
  coefficients.
\newblock {\em arXiv:1809.07321\/} (2018), 48 pages.

\bibitem{JentzenWurstemberger2018LowerBound}
{\sc Jentzen, A., and von Wurstemberger, P.}
\newblock Lower error bounds for the stochastic gradient descent optimization
  algorithm: Sharp convergence rates for slowly and fast decaying learning
  rates.
\newblock {\em arXiv:1803.08600\/} (2018), 42 pages.
\newblock To appear in J.\ Complex.

\bibitem{Karimietal2019}
{\sc {Karimi}, B., {Miasojedow}, B., {Moulines}, E., and {Wai}, H.-T.}
\newblock {Non-asymptotic Analysis of Biased Stochastic Approximation Scheme}.
\newblock {\em arXiv:1902.00629\/} (2019), 32 pages.

\bibitem{KutyniokPetersen2019}
{\sc Kutyniok, G., Petersen, P., Raslan, M., and Schneider, R.}
\newblock A theoretical analysis of deep neural networks and parametric {PDEs}.
\newblock {\em arXiv:1904.00377\/} (2019), 43 pages.

\bibitem{Leietal2019}
{\sc {Lei}, Y., {Hu}, T., {Li}, G., and {Tang}, K.}
\newblock {Stochastic Gradient Descent for Nonconvex Learning without Bounded
  Gradient Assumptions}.
\newblock {\em arXiv:1902.00908\/} (2019), 6 pages.

\bibitem{leshno1993multilayer}
{\sc Leshno, M., Lin, V.~Y., Pinkus, A., and Schocken, S.}
\newblock Multilayer feedforward networks with a nonpolynomial activation
  function can approximate any function.
\newblock {\em Neural Networks 6}, 6 (1993), 861--867.

\bibitem{Maggi2012GMT}
{\sc Maggi, F.}
\newblock {\em Sets of finite perimeter and geometric variational problems},
  vol.~135 of {\em Cambridge Studies in Advanced Mathematics}.
\newblock Cambridge University Press, Cambridge, 2012.

\bibitem{Massart2007}
{\sc Massart, P.}
\newblock {\em Concentration inequalities and model selection}, vol.~1896 of
  {\em Lecture Notes in Mathematics}.
\newblock Springer, Berlin, 2007.
\newblock Lectures from the 33rd Summer School on Probability Theory held in
  Saint-Flour, July 6--23, 2003.

\bibitem{mhaskar1996neural}
{\sc Mhaskar, H.~N.}
\newblock Neural networks for optimal approximation of smooth and analytic
  functions.
\newblock {\em Neural Comput. 8}, 1 (1996), 164--177.

\bibitem{MhaskarMicchelli1995}
{\sc Mhaskar, H.~N., and Micchelli, C.~A.}
\newblock Degree of approximation by neural and translation networks with a
  single hidden layer.
\newblock {\em Adv. in Appl. Math. 16}, 2 (1995), 151--183.

\bibitem{MhaskarPoggio2016}
{\sc Mhaskar, H.~N., and Poggio, T.}
\newblock Deep vs. shallow networks: an approximation theory perspective.
\newblock {\em Anal. Appl. (Singap.) 14}, 6 (2016), 829--848.

\bibitem{nguyen1999approximation}
{\sc Nguyen-Thien, T., and Tran-Cong, T.}
\newblock Approximation of functions and their derivatives: A neural network
  implementation with applications.
\newblock {\em Appl. Math. Model. 23}, 9 (1999), 687--704.

\bibitem{NovakWozniakowski2008}
{\sc Novak, E., and Wo\'{z}niakowski, H.}
\newblock {\em Tractability of multivariate problems. {V}ol. 1: {L}inear
  information}, vol.~6 of {\em EMS Tracts in Mathematics}.
\newblock European Mathematical Society (EMS), Z\"{u}rich, 2008.

\bibitem{NovakWozniakowski2010}
{\sc Novak, E., and Wo\'{z}niakowski, H.}
\newblock {\em Tractability of multivariate problems. {V}olume {II}: {S}tandard
  information for functionals}, vol.~12 of {\em EMS Tracts in Mathematics}.
\newblock European Mathematical Society (EMS), Z\"{u}rich, 2010.

\bibitem{park1991universal}
{\sc Park, J., and Sandberg, I.~W.}
\newblock Universal approximation using radial-basis-function networks.
\newblock {\em Neural Comput. 3}, 2 (1991), 246--257.

\bibitem{perekrestenko2018universal}
{\sc Perekrestenko, D., Grohs, P., Elbr{\"a}chter, D., and B{\"o}lcskei, H.}
\newblock The universal approximation power of finite-width deep {R}e{LU}
  networks.
\newblock {\em arXiv:1806.01528\/} (2018), 16 pages.

\bibitem{petersen2018topological}
{\sc Petersen, P., Raslan, M., and Voigtlaender, F.}
\newblock Topological properties of the set of functions generated by neural
  networks of fixed size.
\newblock {\em arXiv:1806.08459\/} (2018), 56 pages.

\bibitem{petersen2018equivalence}
{\sc Petersen, P., and Voigtlaender, F.}
\newblock Equivalence of approximation by convolutional neural networks and
  fully-connected networks.
\newblock {\em arXiv:1809.00973\/} (2018), 10 pages.

\bibitem{petersen2018optimal}
{\sc Petersen, P., and Voigtlaender, F.}
\newblock Optimal approximation of piecewise smooth functions using deep {ReLU}
  neural networks.
\newblock {\em Neural Networks 108\/} (2018), 296--330.

\bibitem{Pinkus1999}
{\sc Pinkus, A.}
\newblock Approximation theory of the {MLP} model in neural networks.
\newblock In {\em Acta numerica, 1999}, vol.~8 of {\em Acta Numer.} Cambridge
  University Press, Cambridge, 1999, pp.~143--195.

\bibitem{ReisingerZhang2019}
{\sc Reisinger, C., and Zhang, Y.}
\newblock Rectified deep neural networks overcome the curse of dimensionality
  for nonsmooth value functions in zero-sum games of nonlinear stiff systems.
\newblock {\em arXiv:1903.06652\/} (2019), 34 pages.

\bibitem{schmitt2000lower}
{\sc Schmitt, M.}
\newblock Lower bounds on the complexity of approximating continuous functions
  by sigmoidal neural networks.
\newblock In {\em Proceedings of the 12th International Conference on Neural
  Information Processing Systems\/} (Cambridge, MA, USA, 1999), NIPS'99, MIT
  Press, pp.~328--334.

\bibitem{SchwabZech2019}
{\sc Schwab, C., and Zech, J.}
\newblock Deep learning in high dimension: neural network expression rates for
  generalized polynomial chaos expansions in {UQ}.
\newblock {\em Anal. Appl. (Singap.) 17}, 1 (2019), 19--55.

\bibitem{ShahamCloningerCoifman2018}
{\sc Shaham, U., Cloninger, A., and Coifman, R.~R.}
\newblock Provable approximation properties for deep neural networks.
\newblock {\em Appl. Comput. Harmon. Anal. 44}, 3 (2018), 537--557.

\bibitem{shalev2014understanding}
{\sc Shalev-Shwartz, S., and Ben-David, S.}
\newblock {\em Understanding machine learning: From theory to algorithms}.
\newblock Cambridge University Press, Cambridge, 2014.

\bibitem{shen2019deep}
{\sc Shen, Z., Yang, H., and Zhang, S.}
\newblock Deep network approximation characterized by number of neurons.
\newblock {\em arXiv:1906.05497\/} (2019), 36 pages.

\bibitem{shen2019nonlinear}
{\sc Shen, Z., Yang, H., and Zhang, S.}
\newblock Nonlinear approximation via compositions.
\newblock {\em arXiv:1902.10170\/} (2019), 19 pages.

\bibitem{VanDeGeer2000}
{\sc van~de Geer, S.~A.}
\newblock {\em Applications of empirical process theory}, vol.~6 of {\em
  Cambridge Series in Statistical and Probabilistic Mathematics}.
\newblock Cambridge University Press, Cambridge, 2000.

\bibitem{voigtlaender2019approximation}
{\sc Voigtlaender, F., and Petersen, P.}
\newblock Approximation in {$L^p(\mu)$} with deep {ReLU} neural networks.
\newblock {\em arXiv:1904.04789\/} (2019), 4 pages.

\bibitem{yarotsky2017error}
{\sc Yarotsky, D.}
\newblock Error bounds for approximations with deep {ReLU} networks.
\newblock {\em Neural Networks 94\/} (2017), 103--114.

\bibitem{yarotsky2018universal}
{\sc Yarotsky, D.}
\newblock Universal approximations of invariant maps by neural networks.
\newblock {\em arXiv:1804.10306\/} (2018), 64 pages.

\end{thebibliography}
\end{document}